\documentclass[12pt,final]{article}
\usepackage{amsmath,amsthm,amsfonts,amssymb,graphicx,hyperref,enumerate,psfrag,tikz,pgfplots}
\usepackage{mathrsfs}
\usepackage{xcolor}
\newtheorem{lemma}{Lemma}[section]
\newtheorem{theorem}{Theorem}[section]

\newtheorem{coro}{Corollary}[section]
\newtheorem{definition}{Definition}[section]
\newtheorem{problem}{Problem}[section]
\newtheorem{remark}{Remark}[section]
\newtheorem{question}{Question}[section]
\newtheorem{example}{Example}[section]

\newtheorem{assume}{Assumption}[section]
\newtheorem{exo}{Exercise}[section]
\usetikzlibrary{automata,arrows,positioning,calc}
\usetikzlibrary{arrows.meta}
\usepackage{xcolor}
\definecolor{metrorange}{RGB}{235,131,32}
\definecolor{myblue}{RGB}{51,51,178}
\definecolor{myred}{RGB}{189,26,26}
\definecolor{mygreen}{RGB}{0,128,0}
\newcommand{\red[1]}{{{\textcolor{myblue}{#1}}}}
\usepackage{fullpage}
\renewcommand{\baselinestretch}{1.3}
\numberwithin{equation}{section}
\usepackage[utf8]{inputenc} 
\usepackage[english]{babel}

\newcommand{\dZ}{\mathbb {Z}}
\newcommand{\id}{\mathrm {id}}
\newcommand{\supp}{\mathrm {supp}}
\newcommand{\Lip}{\mathrm {Lip}}
\newcommand{\TV}{\mathrm {TV}}

\newcommand{\dN}{\mathbb {N}}
\newcommand{\dR}{\mathbb {R}}
\newcommand{\dC}{\mathbb {C}}

\newcommand{\dS}{\mathbb {S}}

\newcommand{\EE}{{\mathbb{E}}}
\newcommand{\bE}{{\mathbf{E}}}
\newcommand{\Ent}{\mathrm{Ent}}
\newcommand{\Varent}{\mathrm{Varent}}
\newcommand{\Var}{\mathrm{Var}}
\newcommand{\Cov}{\mathrm{Cov}}
\newcommand{\PP}{{\mathbb{P}}}

\newcommand{\dd}{\,\mathrm{d}}
\newcommand{\ee}{\mathfrak {e}}

\newcommand{\cE}{\mathscr {E}}
\newcommand{\cF}{\mathcal {F}}
\newcommand{\cN}{\mathcal {N}}

\newcommand{\dX}{\mathbb{X}}

\newcommand{\cP}{\mathcal{P}}
\newcommand{\dtv}{{\rm d}_{\textsc{tv}}}
\newcommand{\trel}{{\rm t}_{\textsc{rel}}}
\newcommand{\tmix}{{\rm t}_{\textsc{mix}}}
\newcommand{\wmix}{{\rm w}_{\textsc{mix}}}

\newcommand{\dist}{\mathrm{dist}}
\newcommand{\diam}{\mathrm{diam}}

\title{Modern aspects of Markov chains:\\entropy, curvature and the cutoff phenomenon}
\author{Justin Salez}
\date{}

\begin{document}
\maketitle

\begin{abstract}
The cutoff phenomenon is an abrupt transition from out of equilibrium to equilibrium undergone by certain Markov processes in the limit where the size of the state space tends to infinity: instead of decaying gradually over time, their distance to equilibrium remains close to its maximal value for a while and suddenly drops to zero as the time parameter reaches a critical threshold. Discovered four decades ago in the context of card shuffling, this surprising phenomenon has since then been observed in a variety of models, from random walks on groups or complex networks to interacting particle systems. It is now believed to be universal among fast-mixing high-dimensional processes. Yet, current proofs are heavily model-dependent, and identifying the general conditions that trigger a cutoff remains one of the biggest challenges in the quantitative analysis of finite Markov chains.  The purpose of these lecture notes is to  provide a self-contained introduction to this fascinating question, and to describe its recently-uncovered relations with entropy, curvature and concentration.\end{abstract}
\newpage
\renewcommand{\baselinestretch}{1.26}\normalsize 
{\tableofcontents} 
\renewcommand{\baselinestretch}{1.3}\normalsize 
\newpage
\section*{Acknowledgment}
The present lecture notes grew up from a series of mini-courses given in various  schools and conferences, including the Oberwolfach Seminar (Oberwolfach, November 2021), the  Latin American Congress of Probability and Mathematical Statistics (S\~ao Paulo, July 2023), the Clay-Heilbronn Summer School (Bristol, July 2024), and the Saint-Flour Probability Summer School (Saint-Flour, July 2025). The author warmly thanks the organizers and hosting institutes for the opportunity to present this material to a wide audience. The research that led to the  main results  has been supported by the ERC consolidator grant CUTOFF (101123174). Views and opinions expressed are however those of the authors only and do not necessarily reflect those of the European Union or the European Research Council Executive Agency. Neither the European Union nor the granting authority can be held responsible.
\newpage
\section{The cutoff phenomenon}
Markov processes are widely used to model complex random evolutions in our physical world. They also play an important role in modern algorithms for the exploration of massive networks, the enumeration of large-scale combinatorial structures, or the generation of random high-dimensional data. 
The broad aim of this course is to investigate the remarkable nature of their transition from out of equilibrium to equilibrium. In this first chapter, we introduce the  framework, notation and terminology that will allow us to formalize this question. 
\subsection{Continuous-time Markov chains}
The central object of our study is a stochastic matrix $T$ on a finite state space $\dX$. The term \red[stochastic] means that $T(x,y)\ge 0$ for all $x,y\in\dX$ and that
\begin{eqnarray*}
\forall x\in\dX,\qquad \sum_{y\in\dX}T(x,y) & = & 1.
\end{eqnarray*}
We call $T$ the \red[transition matrix] and think of it as describing the  evolution of a random system which jumps from its current location $x\in\dX$ to a new  state $y\in\dX$ with probability $T(x,y)$, independently of the past. If we start from a random state $X_0$ drawn according to some law $\mu_0\in\cP(\dX)$ and iterate this simple stochastic rule, we obtain a random sequence of states $(X_k)_{k\in \dN}$ whose finite-dimensional marginals admit the following product form:
\begin{eqnarray*}
\PP(X_0=x_0,\ldots,X_k=x_k) & = & \mu_0(x_0)T(x_0,x_1)\cdots T(x_{k-1},x_k),
\end{eqnarray*}
 for any $k\in\dN$ and any $(x_0,\ldots,x_k)\in\dX^{k+1}$. Stochastic processes of this form are famously known as \red[Markov chains]. In particular, summing over all trajectories that end up in a given state shows that the law of  $X_k$ is simply
\begin{eqnarray*}
\mu_k & = & \mu_0 T^k, 
\end{eqnarray*}
in the usual sense of vector-matrix multiplication, with $T^k$ denoting the $k-$th power of $T$ and with probability measures on $\dX$ being naturally viewed as row vectors.  

It will be slightly more convenient   to embed the above dynamics into continuous time by introducing independent mean-one exponential  waiting times between the jumps. This gives rise to a continuous-time Markov process, denoted by $(X_t)_{t\ge 0}$. Since the number of transitions occurring in the time interval $[0,t]$ is a Poisson random variable with mean $t$, the probability of the outcome $\{X_t=y\}$, given the initial datum $\{X_0=x\}$ is 
\begin{eqnarray}
\label{def:Pt}
P_t(x,y) &  := & e^{-t}\sum_{k=0}^\infty \frac{T^k(x,y)t^k}{k!}.
\end{eqnarray}
In exponential matrix notation, this rewrites compactly as $P_t=e^{tL}$, where  $L:=T-{\rm Id}$ is the \red[infinitesimal generator] of the process. 
Note that we then have the composition rule 
\begin{eqnarray}
\label{def:semi-group}
\forall s,t\in[0,\infty),\qquad P_tP_s & = &P_{t+s},
\end{eqnarray}
as well as the continuity property $P_t\to \rm Id$ as $t\to 0$. A family of  stochastic matrices $(P_t)_{t\ge 0}$ satisfying those two properties is called a \red[Markov semi-group]. In fact, \emph{any} Markov semi-group on $\dX$ can be put into the form (\ref{def:Pt}) after a trivial time re-scaling, making our story  quite general. For future use, let us note that $(P_t)_{t\ge 0}$ solves the \red[Kolmogorov equation]
\begin{eqnarray}
\label{kolmogorov}
\frac{\dd P_t}{\dd t} & = & LP_t\ = \ P_tL.
\end{eqnarray}

\begin{assume}[Irreducibility]\label{assume:irred}
From now on, we will always assume that $T$ is \red[irreducible]:
\begin{eqnarray*}
\forall (x,y)\in\dX^2,\qquad \exists k\in\dN, \qquad T^k(x,y)>0. 
\end{eqnarray*}
\end{assume}
In graph-theoretical terms, this condition expresses the strong connectedness of the \red[diagram of the chain], defined as the directed graph with vertex set $\dX$ and edge set 
\begin{eqnarray}
\label{def:E}
E & := & \left\{(x,y)\in\dX^2\colon T(x,y)>0\right\}.
\end{eqnarray}Note that irreducibility also means that all entries of $P_t$ are positive as soon as $t>0$. 
Our starting point is the following fundamental result, whose proof  will be given later.
\begin{theorem}[Mixing]\label{th:mixing}There is a unique probability law $\pi\in\cP(\dX)$ satisfying the \red[stationarity equation] $\pi T=\pi$. Moreover, $\pi$ has full support, and it is a global attractor:
\begin{eqnarray*}
\forall \mu\in\cP(\dX),\quad \mu P_t & \xrightarrow[t\to+\infty]{} & \pi.
\end{eqnarray*} 
\end{theorem}
In  probabilistic terms, this result asserts that our process $(X_t)_{t\ge 0}$ progressively \emph{forgets} the details of its particular initial condition and approaches a unique \emph{global} distribution which is spread out across the entire state space. This phenomenon is called \red[mixing], and it is ubiquitous: it occurs for example when a deck of cards is shuffled, or when a gas of particles equilibrates inside a box. It also plays a fundamental role in many modern applications, such as the exploration of complex networks by random walks, the enumeration of large-scale combinatorial structures in computer science, or the generation of random high-dimensional data in computational statistics. In all those situations (and many others), it is of paramount importance to understand the speed at which the convergence to equilibrium occurs. 
\begin{question}[Speed]\label{qu:main}How fast is the convergence promised by Theorem \ref{th:mixing} ?
\end{question}
This question is the starting point of a beautiful and actively growing theory known as \emph{Mixing Times of Markov Chains}. The latter lies at the crossroads between probability, functional analysis and discrete geometry, and has a broad range of applications. We warn the reader that the present lecture notes will only cover a small fraction of the many tools, concepts and results that have been developed so far, and we strongly encourage her to consult the classical textbooks  \cite{MR3726904,MR2341319} for a much more comprehensive account. Before honest work can begin on the above question, we need to agree on a way to quantify the divergence between the law of the system at a given time $t\ge 0$, and the  equilibrium distribution $\pi$. We will here focus on the total-variation distance, which is arguably the most natural choice from a probabilistic viewpoint. Of course, relative entropy, $L^p$ norms and Wasserstein distances constitute perfectly legit alternative options, and they will enter the scene in due time. 
\begin{definition}[Total-variation distance] The \red[total-variation distance] between two probability measures $\mu,\nu\in\cP(\dX)$ is defined by any of the following equivalent expressions:
\begin{eqnarray*}
\dtv(\mu,\nu) & = & \max_{A\subseteq\dX}\left|\mu(A)-\nu(A)\right|\\ 
 \label{dtv:ell1}
 & = & \sum_{x\in\dX}\left(\mu(x)-\nu(x)\right)_+\\
 & = & \frac{1}{2}\sum_{x\in \dX}\left|\mu(x)-\nu(x)\right|\\
 & = & \min_{X\sim\mu,Y\sim\nu}\PP(X\ne Y),
\end{eqnarray*}
where the minimum in the last expression runs over all possible couplings $(X,Y)$ of $(\mu,\nu)$. 
\end{definition}
\begin{exo}[Equivalence]Prove that the four definitions of $\dtv(\mu,\nu)$ coincide. 
\end{exo}
A first important observation is that the total-variation distance decreases under the action of the semi-group: indeed, for any $t\ge 0$ and any $\mu,\nu\in\cP(\dX)$, we can write
\begin{eqnarray*}
\label{dtv:contract}
\dtv\left(\mu P_t,\nu P_t\right) & = & \frac{1}{2}\sum_{y\in\dX}\left|\sum_{x\in\dX}(\mu(x)-\nu(x))P_t(x,y)\right| \\ 
& \le & \frac{1}{2}\sum_{x,y\in\dX}\left|\mu(x)-\nu(x)\right|P_t(x,y) \\ & = & \dtv(\mu,\nu).
\end{eqnarray*}
 Moreover, as soon as $t>0$, this inequality is strict unless $\mu=\nu$, because all entries of $P_t$ are positive. This is actually sufficient to deduce the Markov Chain Convergence Theorem.
\begin{proof}[Proof of Theorem \ref{th:mixing}]Fix any law $\nu\in\cP(\dX)$ and consider the  Cesar\'o average 
\begin{eqnarray*}
\pi_k & := & \frac{\nu + \nu T+\cdots + \nu T^{k-1}}{k},
\end{eqnarray*}
for $k\ge 1$. By compactness,  the sequence $(\pi_k)_{k\ge 1}$ admits a sub-sequential limit $\pi$, and the latter must be stationary because $\pi_k T-\pi_k = \frac{1}{k}\left(\nu T^k-\nu\right)\to 0$ as $k\to\infty$, proving existence. Note that we then have $\pi P_t=\pi$ for all $t\ge 0$ by definition of $P_t$, and since all entries of $P_t$ are positive, this implies that $\pi$ has full support. Now, consider an arbitrary initial law $\mu\in\cP(\dX)$ and set $\ell:=\lim_{t\to\infty}\dtv(\mu P_t,\pi)$, which exists by monotony. By compactness again, the limit $\mu_\infty:=\lim_{k\to \infty}\mu P_{t_k}$  exists along some diverging sequence $(t_k)_{k\ge 1}$, and we  have $\ell=\lim_{k\to\infty}\dtv(\mu P_{t_k},\pi) \ = \ \dtv(\mu_\infty,\pi)$. But for any fixed time $t>0$, we also have $\ell=\lim_{k\to\infty}\dtv(\mu P_{t_k+t},\pi) \ = \ \dtv(\mu_\infty P_t,\pi)$, which violates the strict decrease of total-variation distance unless $\mu_\infty=\pi$, i.e., $\ell=0$. In other words, our stationary measure $\pi$ is a global attractor, and must therefore be unique. 
\end{proof}
With Question \ref{qu:main} in mind, we now make the following definition.
\begin{definition}[Distance to equilibrium] The \red[distance to equilibrium] at time $t\ge 0$ is 
\begin{eqnarray*}
\dtv(t) & := &  \sup_{\mu\in\cP(\dX)}\dtv\left(\mu P_t,\pi\right) \  = \ \max_{x\in\dX}\dtv\left(P_t(x,\cdot),\pi\right).
\end{eqnarray*}
\end{definition}
The second equality is of course a direct consequence of the convexity of $\dtv(\cdot,\cdot)$. Note that $\dtv(0)=1-\pi_\star$, where $\pi_\star:=\min_{x\in\dX}\pi(x)$ denotes the minimum stationary mass (which is positive by Theorem \ref{th:mixing}, but never larger than $1/|\dX|$). Thus, the function $t\mapsto\dtv(t)$ is continuously decreasing from near $1$ (when $|\dX|$ is large) to $0$, and our task will consist in understanding the time-scale on which it does so. 

\subsection{Mixing times}\label{sec:mixrel}
The function $t\mapsto\dtv(t)$ admits a simple and natural operator-theoretic representation, which will be useful. First recall that any matrix $A\in\dR^{\dX\times\dX}$ can be viewed as a linear operator acting on functions $f\colon\dX\to\dR$ in the usual way:
\begin{eqnarray}
\label{def:action}
\forall x\in\dX,\qquad Af(x) & := & \sum_{y\in\dX}A(x,y)f(y).
\end{eqnarray}
Given $p\in[1,\infty]$, we will always write $L^p$ for the Banach space $L^p(\dX,\pi)$, $\|\cdot\|_p$ for the associated norm, and $\|A\|_{p\to q}:=\sup_{f\in \dR^\dX\setminus\{0\}}\left\{\frac{\|Af\|_p}{\|f\|_q}\right\}$ for  the operator norm of $A\colon L^p\to L^q$. Finally, we let $\Pi:=\lim_{t\to\infty}P_t$ be the rank-one transition matrix whose rows are equal to $\pi$:
\begin{eqnarray}
\label{def:Pi}
\forall x,y\in\dX,\qquad \Pi(x,y) & := & \pi(y).
\end{eqnarray}
\begin{lemma}[Operator-theoretic representation of the distance to equilibrium]We have
\label{lm:subadd}
\begin{eqnarray*}
\forall t\ge 0,\qquad 2\dtv(t) & = & \|P_t-\Pi\|_{\infty\to\infty}.
\end{eqnarray*}
In particular, the function $t\mapsto 2\dtv(t)$ is sub-multiplicative: 
\begin{eqnarray}
\forall s,t\ge 0,\qquad \dtv(t+s) & \le & 2\dtv(t)\dtv(s).
\end{eqnarray}
\end{lemma}
\begin{proof}
Fix a matrix $A\in\dR^{\dX\times\dX}$ and a state $x\in\dX$, and observe that 
\begin{eqnarray*}
\left|Af(x)\right| & \le & \|f\|_\infty\sum_{y\in\dX}|A(x,y)|,
\end{eqnarray*}
for all functions $f\colon\dX\to\dR$, with equality in the special case where $f(y)=\textrm{sign}(A(x,y))$. In operator-theoretic terms, we have just established the identity
\begin{eqnarray*}
\|A\|_{\infty\to\infty} & = & \max_{x\in\dX}\left\{\sum_{y\in\dX}|A(x,y)|\right\}.
\end{eqnarray*}
In particular, for any $t\ge 0$, we can apply this to the matrix $A=P_t-\Pi$ to obtain 
\begin{eqnarray*}
\|P_t-\Pi\|_{\infty\to\infty} & = & \max_{x\in\dX}\left\{\sum_{y\in\dX}\left|P_t(x,y)-\pi(y)\right|\right\}\\
& = & 2\max_{x\in\dX}\dtv(P_t(x,\cdot),\pi)\\
& = & 2\dtv(t),
\end{eqnarray*}
establishing the first claim. On the other hand, taking $s\to\infty$ (or $t\to\infty$, or both) in the semi-group property (\ref{def:semi-group}) gives $\Pi P_t= P_t\Pi = \Pi = \Pi^2$ for all $t\ge 0$, hence
\begin{eqnarray*}
(P_t-\Pi)(P_s-\Pi) & = &  P_{t+s}-\Pi,
\end{eqnarray*}
for all $t,s\ge 0$. Thus, the second claim is a consequence of the first one and the sub-multiplicativity of operator norms:  $\|AB\|_{\infty\to\infty}\le\|A\|_{\infty\to\infty}\|B\|_{\infty\to\infty}$, for all $A,B\in\dR^{\dX\times\dX}$. 
\end{proof}
The second part of Lemma \ref{lm:subadd} asserts that $f\colon t\mapsto \log(2\dtv(t))$ is a \red[sub-additive function]:
\begin{eqnarray}
\label{def:subadd}
\forall s,t\ge 0,\qquad f(t+s) & \le & f(t)+f(s).
\end{eqnarray}
We can thus apply a classical result known as Fekete's lemma, which we now recall. 
\begin{exo}[Fekete's lemma]Consider a measurable function $f\colon \dR_+\to \dR$ satisfying (\ref{def:subadd}). Prove that  $\lim_{t\to+\infty}\frac{f(t)}{t} $ exists in $\dR\cup\{-\infty\}$ and equals $\inf_{t>0}\frac{f(t)}{t}$. 
\end{exo}
Note that in our case,  the function $f(t)=\log (2\dtv(t))$ tends to $-\infty$ as $t\to\infty$,  hence $\inf_{t>0}\frac{f(t)}{t}<0$. This legitimates the following definition: 
\begin{definition}[Spectral gap]\label{def:trel}There is a number $\lambda>0$, called the \red[spectral gap], such that
\begin{enumerate}[(i)]
\item $\dtv(t) = \exp\left(-\lambda t+o(t)\right)$ as $t\to +\infty$;
\item $\dtv(t) \ge \frac{1}{2}\exp\left(-\lambda t\right)$, for all $t\ge 0$.
\end{enumerate}
\end{definition}
The asymptotic decay rate $\lambda$ appearing here happens to admit a remarkable spectral interpretation, which explains its name: it is exactly the distance between the two largest eigenvalues of $T$, once projected onto the real axis. This is a classical consequence of Gelfand's formula, which we recall as an exercise. 
\begin{exo}[Gelfand's formula]Recall (or prove) that for any linear operator $A$ on a finite-dimensional space, and any norm $\|\cdot\|$ on the space of such operators, we have
\begin{eqnarray*}
\frac{1}{t}\log\|e^{tA}\| & \xrightarrow[t\to\infty]{} & \max_{\theta\in\mathrm{spec}(A)}\Re(\theta),
\end{eqnarray*}
where $\mathrm{spec(A)}:=\{\theta\in\dC\colon \det(A-\theta{\rm Id})=0\}$ denotes the spectrum of $A$. 
Choosing $A$ to be the restriction of $L$ to the space of zero-mean functions on $\dX$, deduce that
\begin{eqnarray}
\label{def:gap}
\lambda & = & {1-\max_{\theta\in\mathrm{spec}(T)\setminus\{1\}}\Re(\theta)}.
\end{eqnarray}
\end{exo}
In view of Definition \ref{def:trel}, it is tempting to conclude that the  time-scale on which the distance to equilibrium decays is simply the inverse spectral gap
\begin{eqnarray}
\trel & := & \frac 1{\lambda},
\end{eqnarray}
sometimes called the \red[relaxation time]. However, this is not the right  answer to Question \ref{qu:main}, as one often needs to wait much  longer  before the chain even starts to mix. The reason is that the approximation $\dtv(t)\approx e^{-\lambda t}$ only holds asymptotically, in a regime  where $t$ is so large that our process could very well  be  infinitesimally close to equilibrium already. In other words, the relaxation time is a priori only a \emph{near-equilibrium} time-scale, which has little to say about the early behavior of the function $t\mapsto\dtv(t)$. In particular, it does not address the following practical question: \emph{how long do we need to wait before the distance to equilibrium becomes small?} This motivates the following definition, illustrated on Figure \ref{fig:cutoff}.
\begin{definition}[Mixing time]The \red[mixing time]  with precision $\varepsilon\in(0,1)$ is defined as
\begin{eqnarray*}
\tmix(\varepsilon) & := & \min\{t\ge 0\colon \dtv(t)\le \varepsilon\}.
\end{eqnarray*}
\end{definition}
A default choice for the precision parameter is $\varepsilon=\frac{1}{2e}$, in which case we simply write $\tmix:=\tmix\left(\frac{1}{2e}\right)$.
Achieving a better precision is then simply a matter of increasing $\tmix$ by a universal multiplicative factor. Indeed, the sub-multiplicativity of $2\dtv$ (Lemma  \ref{lm:subadd}) implies 
\begin{eqnarray*}
\label{eps:irrelev}
\forall\varepsilon\in\left(0,\frac{1}{2e}\right),\qquad 1 \ \le \ \frac{\tmix(\varepsilon)}{\tmix} & \le & \left\lceil\log\frac{1}{\varepsilon}\right\rceil.
\end{eqnarray*}
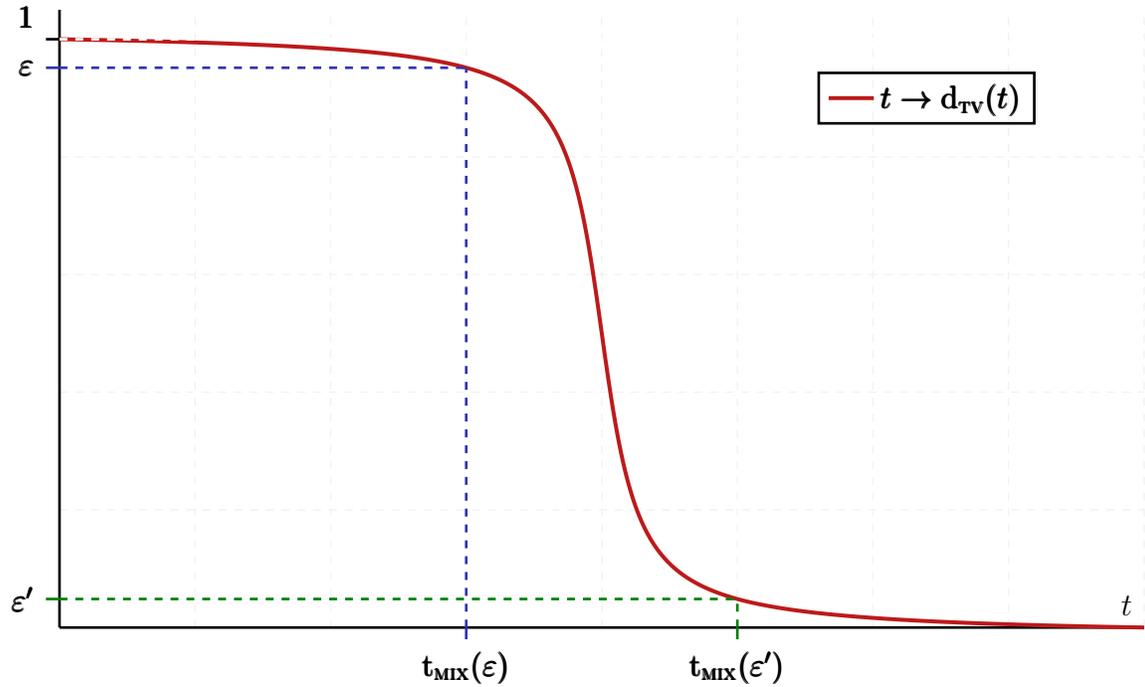
\begin{figure}
\begin{center}
\pgfplotsset{compat=1.16}
\pgfplotsset{ticks=none}
\begin{tikzpicture}
	\begin{axis}[
		width = 16cm,
		height = 9.8cm,
		axis x line=middle,
		axis y line=middle,
		xlabel = $t$,
		clip = false,
		grid=both,
		grid style={dashed, line width=.5pt, draw=gray!10},
		xmode = normal,
		ymode = normal,
		line width = 1pt,
		legend cell align = left,
		legend style = {fill=none, at={(0.9,0.9)}, anchor = north east},
		yticklabel style={above left},
		anchor = north west,
		tickwidth={5pt},
		xtick align = outside,
		ytick align = outside,
		ymax = 1.05,
		ymin = 0,
		xmin = 0,
		xmax = 40,
		x axis line style=-,
		y axis line style=-,
		domain = 0:40,
		samples = 1000
		]
		\def\scale{3}
\addplot[solid,  myred, line width = 1.5 pt] {
-rad(atan(x-20))/rad(atan(20))/2 + 0.5	
	};
\addplot[dashed, white] {1.00};
\legend
{
	$\pmb{t \to \dtv(t)}$
}
\draw[dashed, color= myblue] (15,0) -- (15,0.9515);
\draw[dashed, color= myblue] (0,0.9515) -- (15,0.9515);
\draw[dashed, color= mygreen] (25,0) -- (25,0.0484723);
\draw[dashed, color= mygreen] (0,0.0484723) -- (25,0.0484723);	

\draw[color= myblue] (0,0.9515) -- (-0.5,0.9515) node[color = black, anchor=east] {$\pmb{\varepsilon}$};
\draw[color= myblue] (15,0) -- (15,-0.02) node[color=black, anchor=north] {$\pmb{\tmix(\varepsilon)}$};
\draw[color= mygreen] (25,0) -- (25,-0.02) node[color = black, anchor=north] {$\pmb{\tmix(\varepsilon')}$};
\draw[color= mygreen] (0,0.0484723) -- (-0.5,0.0484723) node[color = black,anchor=east] {$\pmb{\varepsilon'}$};
\draw[color= black] (0,1.0) -- (-0.5,1.0) node[color = black,anchor=south east] {$\pmb{1}$};
\end{axis}
	
\end{tikzpicture}
\caption{Distance to equilibrium and mixing times.}
\label{fig:cutoff}
\end{center}
\end{figure} 

In this sense, the quantity $\tmix$ provides a rigorous formalization of  Question \ref{qu:main}, and its analysis will occupy the center of our attention.

\begin{remark}[Mixing vs relaxation]\label{rk:tmixvstrel}Because of Property (ii) in Definition \ref{def:trel}, we have
\begin{eqnarray}
\label{lb:rel}
\forall\varepsilon\in(0,1),\quad \tmix(\varepsilon) & \ge & \frac{1}{\lambda}\log\left(\frac 1{2\varepsilon}\right),
\end{eqnarray}
hence $\tmix\ge\trel$. However, this bound can be arbitrary loose without further assumptions. 
\end{remark}
Over the past decades, a rich variety of probabilistic, geometric and functional-analytic techniques have been developed for estimating mixing times, and we will survey some of them in due time. Of course, this task becomes particularly relevant when the number of states becomes large. Rather than working with a fixed transition matrix $T$, we will therefore often assume to be given an entire \red[model], by which we simply mean the following.

\begin{definition}[Model] A \red[model] is a sequence  of irreducible transition matrices $(T_n)_{n\ge 1}$ defined on a corresponding sequence of state spaces $(\dX_n)_{n\ge 1}$, such that $|\dX_n|\to\infty$ as $n\to\infty$.
\end{definition}
Given a model, our goal will naturally consist in determining the asymptotic behavior of the associated mixing times $(\tmix^{(n)})_{n\ge 1}$. If $(a_n)_{n\ge 1},(b_n)_{n\ge 1}$ are sequences of positive numbers, we will use the  classical asymptotic notation $a_n=o(b_n)$ or $a_n\ll b_n$ (resp. $a_n=\omega(b_n)$ or $a_n\gg b_n$) to mean that the ratio $a_n/b_n$ tends to $0$ (resp. $\infty$) as $n_to\infty$, and $a_n\sim b_n$ when this ratio tends to $1$. Similarly, we will write $a_n=O(b_n), a_n=\Omega(b_n)$ and $a_n=\Theta(b_n)$ to mean that the ratio $a_n/b_n$ is respectively bounded away from infinity, zero and both. As usual, we add the subscript $\PP$ to this notation when the sequences are random and the corresponding property holds in probability. Also, to lighten the exposition, we will often drop the subscript $n$ from the notation, and simply write $\dX$, $T$, $\pi$, $\tmix$, $\lambda$... The reader will have to keep in mind that those objects implicitly depends on an underlying `size parameter' $n$, and that what we mostly care about is the large-size limit $n\to\infty$.  For concreteness, let us analyze in detail two fundamental models which are simple enough to allow for explicit computations, and which will serve as our running examples throughout the  notes.

\subsection{Two toy models}
Our first toy model is called \red[random walk on  the $n-$cycle]. The state space $\dX$ is here the set of integers modulo $n$, and the transition matrix is
\begin{eqnarray*}
T(x,y) & := & \left\{
\begin{array}{ll}
1/2 & \textrm{if }y=x\pm 1\mod n\\
0 & \textrm{else}.
\end{array}
\right.
\end{eqnarray*}
By symmetry, the stationary distribution $\pi$ is uniform on $\dX$, and we may assume that the initial state is $0$. The position of the walk at any time $t\ge 0$ can then be represented as
\begin{eqnarray*}
X_t & = & U_t-V_t\mod n,
\end{eqnarray*}
where $U_t,V_t$ are independent Poisson random variables with mean $t/2$. When $t$ is large, the Central Limit Theorem (CLT) tells us that $U_t-V_t\approx \sqrt{t}Z$, where $Z\sim\cN(0,1)$. More formally, for any fixed $t>0$ and any $0\le a\le b\le 1$, we have 
\begin{eqnarray}
\label{CLT}
\PP\left(X_{tn^2}\in[an,bn]\right) & = & \int_a^b f_t(z)\, {\rm d}z +o(1),
\end{eqnarray}
as $n\to\infty$,  where $f_t\colon\dR/\dZ\to(0,\infty)$ denotes the density of the random variable $\sqrt{t}Z\mod 1$:
\begin{eqnarray*}
f_t(z) & := & \frac{1}{\sqrt{2\pi t}}\sum_{k\in\dZ}e^{-\frac{(z-k)^2}{2t}}.
\end{eqnarray*}
Unfortunately, the approximation (\ref{CLT}) is far too weak for our purpose: indeed, it is  restricted to macroscopic intervals of the form $A=[an,bn]$, whereas the definition of the total variation distance to equilibrium $\dtv(t)$ involves a maximum over \emph{all} possible regions $A\subseteq \dX$. To estimate this delicate quantity, we need a classical  local refinement of the CLT known as the \red[Local Central Limit Theorem], which asserts that for any fixed $z\in\dR/\dZ$,
\begin{eqnarray}
\label{localCLT}
n\PP(X_{tn^2}=\lfloor zn\rfloor) & = & f_t(z)+o(1).
\end{eqnarray}
This powerful microscopic estimate -- which implies the macroscopic estimate (\ref{CLT})  by integrating over $z\in[an,bn]$ --  allows us to conclude as follows: for any fixed $t>0$, we have
\begin{eqnarray*}
\dtv(tn^2) & = & \sum_{x\in\dX}\left(\pi(x)-\PP(X_{tn^2}=x)\right)_+\\
& = & \int_{0}^1\left(1-n\PP(X_{tn^2}=\lfloor zn\rfloor)\right)_+\,\mathrm{d}z\\
& = & \int_{0}^1\left(1-f_t(z)\right)_+\,\mathrm{d}z + o(1),
\end{eqnarray*}
 thanks to (\ref{localCLT}) and dominated convergence. Let us summarize this in a theorem. 
\begin{theorem}[Random walk on the $n-$cycle] In the $n\to\infty$ limit, the function $t\mapsto \dtv^{(n)}(tn^2)$ converges pointwise (in fact uniformly, by Dini's theorem) to the  limiting profile
\begin{eqnarray*}
F \colon t &  \longmapsto &  \int_{0}^1\left(1-f_t(z)\right)_+\,\mathrm{d}z \ = \ \frac{1}{2}\int_{0}^1\left|1-f_t(z)\right|\,\mathrm{d}z,
\end{eqnarray*}
which decays smoothly from $F(0)=1$ to $F(+\infty)=0$, as illustrated on Figure \ref{fig:jacobi}. Thus,
\begin{eqnarray}
\label{ex:cycle}
\forall \varepsilon\in(0,1),\qquad \tmix^{(n)}(\varepsilon) & \sim & n^2  F^{-1}(\varepsilon).
\end{eqnarray}
\end{theorem}
In words, the random walk on the $n-$cycle mixes gradually, on the timescale $n^2$. What is perhaps surprising, and quite exciting, is that this smooth convergence to stationarity is actually \emph{not} the typical picture: as we will see, many interesting systems undergo instead a remarkably abrupt transition from out of equilibrium to equilibrium, known as a \emph{cutoff}.
\begin{figure}[!h]
\begin{center}
\label{fig:jacobi}
\includegraphics[scale=0.8]{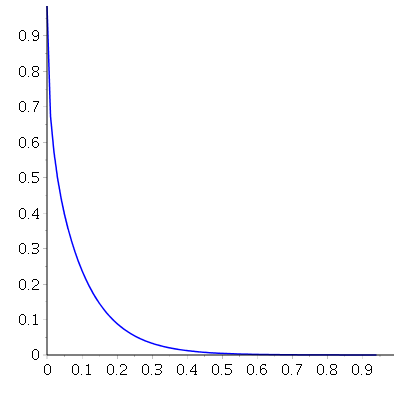}
\caption{The smooth limiting profile $F$ for simple random walk on the $n-$cycle.}
\end{center}
\end{figure}

To demonstrate this, we now introduce our second toy model, \red[random walk on the $n$-cube]. The latter is the Markov chain with state space $\dX=\{-1,1\}^n$ and transition matrix
\begin{eqnarray*}
T(x,y) & := & \left\{
\begin{array}{ll}
\frac 1n & \textrm{if }y\textrm{ differs from }x\textrm{ at a single coordinate};\\
0 & \textrm{else}.
\end{array}
\right.
\end{eqnarray*}
Here again, $\pi$ is the uniform distribution by symmetry, and we may assume without loss of generality that the initial state is $o:=(1,\ldots,1)$. Under our continuous-time dynamics, each coordinate is refreshed independently at rate $2/n$ according to the uniform law on $\{-1,1\}$. Consequently, the density  w.r.t. equilibrium has an explicit product form at any time $t\ge 0$:
\begin{eqnarray*}
\frac{P_t(o,x)}{\pi(x)} & = & \prod_{i=1}^n\left(1+x_ie^{-\frac{2t}{n}}\right).
\end{eqnarray*}
If we let $X=(X_1,\ldots,X_n)$ denote a uniform random element of $\dX$, we obtain the following simple representation for  the distance to equilibrium:
\begin{eqnarray*}
\dtv(t) & = & \sum_{x\in\dX}\pi(x)\left(1-\frac{P_t(o,x)}{\pi(x)}\right)_+ \\
& = & \EE\left[\left(1-\frac{P_t(o,X)}{\pi(X)}\right)_+\right]\\
& = & \EE\left[\left(1-e^{W_t}\right)_+\right],
\end{eqnarray*}
where we have introduced the log-likelihood ratio 
\begin{eqnarray*}
W_t & := & \log\left(\frac{P_t(o,X)}{\pi(X)}\right) \ = \ \sum_{i=1}^n\log\left(1+X_ie^{-\frac{2t}{n}}\right).
\end{eqnarray*}
In particular, if $t=t(n)$ is chosen so that $W_t\to -\infty$ in probability as $n\to\infty$, then $\dtv(t)=1-o(1)$ and mixing has not started. On the other hand, if $W_t\to 0$  in probability as $n\to\infty$, then $\dtv(t)=o(1)$ and the chain is completely mixed. Thus, the transition from out of equilibrium to equilibrium should occur  when $t=t(n)$ is such that $W_t=\Theta_\PP(1)$. To make this precise, observe that as soon as $t\gg n\gg 1$, we have the Taylor expansion
\begin{eqnarray*}
W_t & = & e^{-\frac{2t}{n}}\sum_{i=1}^n X_i-\frac{n}{2}e^{-\frac{4t}{n}}+o_\PP(ne^{-\frac{4t}{n}})\\
& = & \sqrt{n}e^{-\frac{2t}{n}}(Z+o_\PP(1))-\frac{n}{2}e^{-\frac{4t}{n}}+o_\PP(ne^{-\frac{4t}{n}}),
\end{eqnarray*}
thanks to the CLT, with $Z\sim\cN(0,1)$. This strongly suggests looking at times of the form
\begin{eqnarray*}
t & = & \frac{n\log n}{4} + sn+o(n),
\end{eqnarray*}
where $s\in\dR$ remains fixed as $n\to\infty$. In this regime, we have $W_t = e^{-2s} Z-\frac{e^{-4s}}{2}+o_\PP(1)$ and therefore, $\dtv(t)=F(s)+o(1)$, where
\begin{eqnarray*}
F(s) & := & \EE\left[\left(1-e^{e^{-2s} Z-\frac{e^{-4s}}{2}}\right)_+\right]\\
& = & \EE\left[\left(1-e^{\frac{Z^2-(Z-e^{-2s})^2}{2}}\right){\bf 1}_{\left\{Z\le \frac{e^{-2s}}{2}\right\}}\right]\\
& = & \PP\left(Z\le \frac{e^{-2s}}{2}\right)-\PP\left(Z+e^{-2s}\le \frac{e^{-2s}}{2}\right)\\
& = & \mathrm{erf}\left(\frac{e^{-2s}}{2\sqrt{2}}\right).
\end{eqnarray*}
Here, $\mathrm{erf}\colon u\mapsto \frac{2}{\sqrt{\pi}}\int_0^ue^{-z^2}\,\mathrm{d}z$ is the Gauss error function. Let us summarize this in a theorem. 
\begin{theorem}[Simple random walk on the $n-$cube]\label{th:ncube} For any fixed $s\in\dR$, we have
\begin{eqnarray*}
\dtv^{(n)}\left(\frac{n\log n}{4} + sn\right) & \xrightarrow[n\to\infty]{} & F(s),
\end{eqnarray*}
where $F(s)=\mathrm{erf}\left(\frac{e^{-2s}}{2\sqrt{2}}\right)$. In other words, for any fixed $\varepsilon\in(0,1)$,
\begin{eqnarray}
\label{ex:cube}
\tmix^{(n)}(\varepsilon) & = & \frac{n\log n}{4} +n F^{-1}(\varepsilon)+o(n).
\end{eqnarray}
\end{theorem} 
What is striking here, compared to the previous example, is that the leading-order asymptotics of $\tmix(\varepsilon)$ is independent of the precision parameter $\varepsilon$. This means that our process undergoes an abrupt transition from out of equilibrium to equilibrium, in the following precise sense: for any fixed $t>0$, as $n\to\infty$, 
\begin{eqnarray}
\label{cutoff:cube}
\dtv^{(n)}(tn\log n) & = & \left\{
\begin{array}{ll}
1-o(1) & \textrm{if }t<1/4\smallskip\\
o(1) & \textrm{if }t>1/4.
\end{array}
\right.
\end{eqnarray}
In words, the distance to equilibrium remains arbitrarily close to one until the critical time $t=(n\log n)/4$, where it abruptly drops to zero. In fact, the statement (\ref{ex:cube}) is even more precise: not only is there an abrupt transition from out of equilibrium to equilibrium  around the critical time $(n\log n)/4$ but moreover,  this transition occurs over a time interval of width $n$, and the distance to equilibrium inside this window approaches a smooth explicit shape. 
\subsection{The cutoff phenomenon}

The surprising phase transition appearing above is our first example of a generic phenomenon known as a cutoff, and which can be formally defined as follows.
\begin{definition}[Cutoff]\label{def:cutoff}A model $(T_n)_{n\ge 1}$  exhibits a (worst-case)  \red[cutoff] if
\begin{eqnarray*}
\forall \varepsilon\in(0,1),\qquad \frac{\tmix^{(n)}(1-\varepsilon)}{\tmix^{(n)}(\varepsilon)} & \xrightarrow[n\to\infty]{} &  1.
\end{eqnarray*}
Equivalently, the function $t\mapsto \dtv^{(n)}(t)$, suitably rescaled,  approaches a step function as $n\to\infty$, i.e. there is a sequence of non-negative numbers $(\theta_n)_{n\ge 1}$ such that for any fixed $t\ge 0$, 
\begin{eqnarray*}
\dtv^{(n)}(\theta_n t) & \xrightarrow[n\to\infty]{} & \left\{
\begin{array}{ll}
1 & \textrm{if }t<1\smallskip\\
0 & \textrm{if }t>1.
\end{array}
\right.
\end{eqnarray*}

\end{definition}
Although the name \emph{cutoff} was coined   in 1986 by D. Aldous and P. Diaconis \cite{aldous1986shuffling}, the phenomenon itself was actually discovered in 1981 by P. Diaconis and M. Shahshahani  \cite{diaconis1981generating}, and several instances of it were collected under the generic name \emph{abrupt switch} in lecture notes published by D. Aldous in 1983 \cite{aldous1983mixing}. Further historical examples can be found in the 1996 survey paper \emph{The cutoff phenomenon in finite Markov chains}, by P. Diaconis \cite{diaconis1996cutoff}. Over the past few years, this dynamical phase transition has been established in nearly a hundred different models. An exhaustive account would take too long, but here is a very personal selection of recent works that the author finds particularly inspiring:
\begin{itemize}
\item the work of N. Berestycki and B. Seng\"ul  on conjugacy-invariant random walks on the symmetric group \cite{MR3936154}, which substantially generalizes the historical result \cite{diaconis1981generating} and bypasses the delicate use of Fourier analysis on non-Abelian groups  \cite{MR964069,MR4152644,MR4312845,MR4367947,MR4021252};
\item the proof of cutoff for the random walk on random regular graphs by E. Lubetzky and A. Sly \cite{MR2667423}, which initiated a long line of works establishing cutoff for random walks on various models of sparse random graphs \cite{MR3758735,MR4132638,MR4385126,MR4476123,hermon2021cutoff,arxiv.2212.04469};
\item the result of E. Lubetzky and Y. Peres \cite{MR3558308} establishing cutoff for random walks on Ramanujan graphs, which later received several alternative proofs \cite{MR3693771,MR4178418,MR4476123};
\item the work of E. Lubetzky and A. Sly on the stochastic Ising model \cite{MR3020173}, which went for the first time beyond the mean-field case \cite{levin2010glauber,cuff2012glauber} and introduced \emph{Information Percolation}  to prove cutoff for Glauber dynamics on high-temperature spin systems \cite{lubetzky2014cutoff,MR3486171,MR4038041};
\item the work of H. Lacoin on the one-dimensional symmetric exclusion process \cite{MR3474475}, which is the source of many refinements and extensions  \cite{MR3551201,MR3689972,MR3320314,MR3945753,MR4021252,MR4291457,MR4152648,MR4421608,arxiv.2003.03781,arxiv.2110.06353}.
\end{itemize}
Despite the accumulation of many examples, the cutoff phenomenon remains vastly unexplained: the above proofs are model-specific, and ultimately based on the explicit determination of $\tmix(\varepsilon)$ up to a $1 + o(1)$ factor. More precisely, the current theory consists of a rich set of tools and techniques to obtain upper or lower bounds on the mixing time of a Markov chain, in terms of various natural statistics. Implementing those generic methods on a particular model will result in a conclusion of the form
\begin{eqnarray}
\label{bruteforce}
\tmix^-(\varepsilon) \ \le & \tmix(\varepsilon) & \le \ \tmix^+(\varepsilon),
\end{eqnarray}
with the quantities $\tmix^\pm(\varepsilon)$ being  hopefully simple enough to admit explicit first-order asymptotics as the size parameter $n$ of the model tends to infinity. If we are lucky enough, those upper and lower asymptotics will have the same order of magnitude. However, in order to determine whether a cutoff occurs or not, we need them to coincide exactly, i.e.,
\begin{eqnarray}
\frac{\tmix^+(\varepsilon)}{\tmix^-(\varepsilon)} & = & 1+o(1).
\end{eqnarray}
Beyond the fact that it does not really provide any conceptual explanation as to why a sharp transition occurs, such a brute-force approach is inherently fragile: it typically relies on the exactly solvable nature of certain features of the model (dominant eigenvalues and eigenvectors, lower-dimensional projections, single-site marginals), and is therefore intrinsically restricted to highly structured examples: product or near-product chains,
one-dimensional processes, mean-field systems, Abelian walks, etc. This limitation is particularly
frustrating in regard to the widely shared belief that cutoff should be a universal
feature of rapidly mixing high-dimensional systems. Bridging this  four-decade-old gap has  become one of the most challenging open problems in the quantitative analysis of Markov processes.  
\begin{problem}[Holy grail]\label{pb:main}
Find an  easily verifiable and model-independent  criterion allowing to predict cutoff without having to determine mixing times within a $1+o(1)$ precision.
\end{problem}

To better explain what is hoped for here,  it is perhaps worth spending a minute on a crude analogy with another ubiquitous -- but far better understood -- phase transition in probability: the \emph{concentration-of-measure} phenomenon. Let us say that a sequence of non-negative random variables $(Z_n)_{n\ge 1}$  \red[concentrates] if  there is a deterministic  sequence of positive numbers $(\theta_n)_{n\ge 1}$ -- typically the means or medians -- such that 
\begin{eqnarray*}
\frac{Z_n}{\theta_n} & \xrightarrow[n\to\infty]{\PP} & 1.
\end{eqnarray*}
In terms of the tail distribution function $F_n\colon t\mapsto \PP(Z_n\ge t)$, this is equivalent to the assertion
\begin{eqnarray*}
F_n(\theta_n t) & \xrightarrow[n\to\infty]{} & \left\{
\begin{array}{ll}
1 & \textrm{if }t<1;\\
0 & \textrm{if }t>1,
\end{array}
\right.
\end{eqnarray*}
and  replacing $F_n$ with $\dtv^{(n)}$ yields exactly the definition of a cutoff. The point we are trying to make with this analogy is as follows. There is  a very well developed and highly effective theory for establishing that random variables concentrate, without having to explicitly estimate their tail distribution functions, nor even their means: the conclusion is guaranteed, for example, as soon as $Z_n$ is a function of many independent random variables, each having very little influence by itself (we refer the unfamiliar reader to the excellent textbook \cite{MR3185193} for an introduction). Problem \ref{pb:main} rightfully asks whether a similar theory can be developed for the distance to equilibrium of Markov chains. Another fruitful analogy is provided by the theory of \emph{sharp thresholds for monotone boolean functions}, which allows one to systematically predict the emergence of phase transitions in large random structures without computing the critical threshold at which they occur; see the survey \cite{MR3921417} and the references therein. 

\subsection{The product condition}
\label{sec:PC}
 In a famous attempt to solve Problem \ref{pb:main}, Y. Peres proposed in 2004 an important criterion for cutoff based on the spectral gap and known as the \emph{product condition} \cite{peresamerican}.
\begin{definition}[Product condition] A model $(T_n)_{n\ge 1}$ satisfies the \red[product condition] if
\begin{eqnarray*}
\lambda_n\,\tmix^{(n)}  & \xrightarrow[n\to\infty]{}  & \infty.
\end{eqnarray*}
\end{definition}
This condition is easy to verify in practice because it only requires rough lower bounds on the mixing time and the spectral gap, in contrast with the brute-force approach (\ref{bruteforce}) which consists in determining the precise pre-factor in front of mixing times. It is not hard to show that the product condition is at least \emph{necessary} for cutoff.
\begin{lemma}[Necessity]Any model exhibiting cutoff must satisfy  the product condition.
\end{lemma}
\begin{proof}Suppose that $(T_n)_{n\ge 1}$ exhibits cutoff, and fix a precision $\varepsilon\in(0,\frac 12)$. We then have
$
t_{\textsc{mix}}^{(n)}  \sim t_{\textsc{mix}}^{(n)}(\varepsilon)
$ 
as $n\to\infty$. On the other hand, we have seen at (\ref{lb:rel}) that
\begin{eqnarray*}
\lambda_n\tmix^{(n)}(\varepsilon) & \ge & \log\left(\frac{1}{2\varepsilon}\right).
\end{eqnarray*}
Combining those two estimates, we deduce that
\begin{eqnarray*}
\liminf_{n\to\infty}\lambda_n\tmix^{(n)} & \ge & {\log\left(\frac{1}{2\varepsilon}\right)},
\end{eqnarray*}
and the right-hand side can be made arbitrarily large by choosing $\varepsilon$ small.
\end{proof}
Unfortunately, the converse statement -- which is the one that would have solved Problem \ref{pb:main} -- does not always hold. Even worse, any model  exhibiting cutoff can be perturbed so as to produce a counter-example. To see this, consider the rank-one perturbation 
\begin{eqnarray}
\label{def:bias}
\widetilde{T}(x,y) & := & (1-\theta)T(x,y)+\theta\pi(y),
 \end{eqnarray} 
where $\theta\in(0,1)$ a parameter to be adjusted later.  Note that $\widetilde{T}$ remains irreducible and with stationary law $\pi$. The interpretation of (\ref{def:bias}) is simple: before each jump, a biased coin with parameter $\theta$ is tossed, and the outcome determines whether the next state is chosen  according to the original rule, or to the stationary law $\pi$. The effect of this perturbation on the mixing and relaxation times is summarized in the following result, in which we naturally use tildas to distinguish the perturbed model from the original one. 
 \begin{lemma}[Rank-one perturbations preserve the product condition but destroy cutoff]Consider a model $(T_n)_{n\ge 1}$ which exhibits cutoff, and choose $(\theta_n)_{n\ge 1}$  in $(0,1)$ so that
\begin{eqnarray*}
\frac{1}{\tmix^{(n)}} \ \ll \ {\theta_n} \ \ll \ \lambda_n.
\end{eqnarray*}
Then, the model $(\widetilde{T}_n)_{n\ge 1}$ produced by the rank-one perturbation (\ref{def:bias}) satisfies
\begin{eqnarray*}
\widetilde{\lambda}_n  \sim \ \lambda_n,& \quad\textrm{and}\quad & 
\widetilde{{\rm t}}^{(n)}_{\textsc{mix}}(\varepsilon) \ \sim \ \frac{1}{\theta_n}\log\left(\frac{1}{\varepsilon}\right).
\end{eqnarray*}
In particular, $(\widetilde{T}_n)_{n\ge 1}$ still satisfies the product condition, but fails to exhibit cutoff. 
\end{lemma}
\begin{proof}The effect of the rank-one perturbation (\ref{def:bias}) on the semi-group is simple: 
\begin{eqnarray*}
\widetilde{P}_t(x,y) & = &  e^{-\theta t}P_{(1-\theta)t}(x,y)+(1-e^{-\theta t})\pi(y),
\end{eqnarray*}
for all $t\ge 0$ and all $x,y\in\dX$. Now, substract $\pi(y)$ from both sides, take absolute values, sum over all $y\in\dX$ and finally maximize over the choice of the starting state $x\in\dX$ to obtain
\begin{eqnarray*}
\widetilde{\rm d}_{\textsc{tv}}(t) & = &  e^{-\theta t} {\dtv}((1-\theta)t).
\end{eqnarray*}
Examining the $t\to\infty$ behavior, we deduce that $\widetilde{\lambda} =   1-\theta+\theta \lambda$, and specializing those general identities to $T=T_n$ and $\theta=\theta_n$ easily leads to the claim.
\end{proof}
The  transformation (\ref{def:bias}) happens to preserve many structural properties of the  model, and can therefore be used to produce counter-examples to Peres' conjecture even within very well-behaved ensembles, such as random walks on Abelian groups (as defined below). Note that the perturbed model is extremely close to the original one, not only entry-wise but also in terms of eigenvalues and eigenvectors. Yet, the latter exhibits cutoff whereas the former does not,  and any general criterion for predicting cutoff will have to be subtle enough to distinguish between them. This demonstrates that cutoff is a highly delicate phenomenon. 
 One  key feature which is radically affected by (\ref{def:bias}), however, is \emph{sparsity}: while  $T$ may have  very few non-zero entries, its perturbed version $\widetilde{T}$ is always full, meaning that its diagram is the complete graph. This might be taken as a hint that sparsity should play an important role, and we will come back to this point later. Still, restricting attention to sparse models is not enough to repair the product condition, since E. Lubetzky and A. Sly produced explicit families of cubic expanders without cutoff \cite{MR2774096}. Nevertheless, those constructions are widely considered as \emph{pathological}, and the general consensus is that the product condition should correctly predict cutoff on \emph{reasonable} models. Giving an honest mathematical content to this vague claim is a major open problem.
\begin{problem}[Cutoff for reasonable chains]\label{pb:PC}Prove that  cutoff  is equivalent to the product condition for all `reasonable' models, in a sense that remains to be defined.
\end{problem}
Limited progress has been made on this question so far: J. Ding, E. Lubeztky and Y. Peres proved the equivalence for birth-and-death processes \cite{ding2010total}, and this was later extended to random walks on weighted trees by R. Basu, J. Hermon and Y. Peres \cite{MR3650406}. Unfortunately, those results remain extremely scarce in comparison with the rich variety of \emph{reasonable} models that satisfy the product condition, but have resisted all attempts to prove cutoff. The next chapter contains a selection of some of the most emblematic such processes.

\newpage
\section{Some important models}
\label{sec:models}
In this chapter, we describe several important classes of Markov chains with good structural properties, for which cutoff is still far from being understood. 
\subsection{Reversible chains}
Viewing a matrix $A\in\dR^{\dX\times\dX}$ as a linear operator acting on the Hilbert space $L^2=L^2(\dX,\pi)$ via (\ref{def:action}), we may consider its \red[adjoint] $A^\star$, abstractly defined by the  duality relation
\begin{eqnarray*}
\forall f,g\in L^2,\qquad \langle A^\star f,g \rangle & = & \langle f,Ag\rangle.
\end{eqnarray*}
Choosing $f=\delta_y$ and $g=\delta_x$, we arrive at the more concrete expression
\begin{eqnarray*}
\forall (x,y)\in\dX^2,\qquad A^\star(x,y) & = & \frac{\pi(y)A(y,x)}{\pi(x)}.
\end{eqnarray*}
When applied to our transition matrix $T$, this formula clearly produces another stochastic matrix  $T^\star$, with the same stationary law $\pi$. From a dynamical viewpoint, the duality $T\mapsto T^\star$ can be interpreted as \emph{time reversal}, as explained in the following classical exercise.
\begin{exo}[Time reversal]\label{exo:reversal}
Let $(X_t)_{t\ge 0}$ and $(X_t^\star)_{t\ge 0}$ be Markov processes with transition matrices $T$ and $T^\star$ respectively, both initialized from the stationary law $\pi$. Show that
\begin{eqnarray*}
\forall t\ge 0,\qquad \left(X^\star_s\right)_{s\in[0,t]} & \stackrel{d}{=} & \left(X_{t-s}\right)_{s\in[0,t]}.
\end{eqnarray*}
\end{exo}
Of course, the continuous-time Markov semi-group generated by the transition matrix $T^\star$ is then simply the adjoint $(P_t^\star)_{t\ge 0}$ of the semi-group generated by $T$. Because self-adjoint operators are much better behaved than non self-adjoint ones, Markov chains that satisfy the self-duality relation $T^\star=T$ play a distinguished role in the theory of mixing times. 
\begin{definition}[Reversibility]$T$ is called \red[reversible] when $T^\star=T$, or equivalently, 
\begin{eqnarray}
\label{DBE}
\forall (x,y)\in\dX^2,\qquad \pi(x)T(x,y) & = & \pi(y)T(y,x).
\end{eqnarray}
\end{definition}
Note that summing the  identity (\ref{DBE}) over all $x\in\dX$ yields exactly the stationarity property $\pi T=\pi$,  also known as the \red[global balance equation]. For this reason, (\ref{DBE}) is sometimes referred to as the \red[detailed balance equation]. This powerful local symmetry happens to hold in many fundamental models, as we will see. Whenever it does, the spectral theorem for self-adjoint operators provides a particularly convenient decomposition of the form
\begin{eqnarray}
\frac{T^k(x,y)}{\pi(y)} & = & 1+\sum_{j=2}^N \lambda_j^k\phi_j(x)\phi_j(y),
\end{eqnarray}
for all $x,y\in\dX$ and all $k\in\dN$, where $1=\lambda_1>\lambda_2\ge\ldots\ge \lambda_N$ denote the ordered eigenvalues of $T$, and $1=\phi_1,\ldots,\phi_N$ a corresponding orthonormal basis of eigenvectors. It follows that
\begin{eqnarray}
\label{spectraldec}
\frac{P_t(x,y)}{\pi(y)}  & = & 1+\sum_{j=2}^N e^{-\gamma_j t}\phi_j(x)\phi_j(y),
\end{eqnarray}
for all times $t\ge 0$, with $\gamma_j=1-\lambda_j$. This exact expansion is the starting point of various classical estimates on the mixing times of reversible chains, whose strength depends on the amount of spectral information that one is ready to use. The most basic instance is the following converse to (\ref{lb:rel}), which shows that the ratio between the mixing time and the  relaxation time of a reversible chain is at most logarithmic in the volume parameter $\frac{1}{\pi_\star}$.
\begin{lemma}[Mixing vs relaxation, reversed]\label{lm:L2bound} For reversible chains, we have
\begin{eqnarray*}
\forall\varepsilon\in(0,1),\qquad \tmix(\varepsilon) & \le & \frac{1}{2\lambda} \log\left(\frac{1-\pi_\star}{4\pi_\star\varepsilon^2}\right).
\end{eqnarray*}
\end{lemma}
\begin{proof}
Fix  $x\in\dX$ and $t\ge 0$, and let us use the Cauchy-Schwartz inequality to write
\begin{eqnarray*}
2\dtv(P_t(x,\cdot),\pi) & = & \left\|\frac{P_t(x,\cdot)}{\pi(\cdot)}-1\right\|_1
\  \le \ \left\|\frac{P_t(x,\cdot)}{\pi(\cdot)}-1\right\|_2.
\end{eqnarray*}
The motivation for doing so is that, as any orthonormal decomposition, the expansion (\ref{spectraldec}) is particularly convenient for estimating $L^2$ norms. Specifically, we have
\begin{eqnarray}
\label{Parseval}
\left\|\frac{P_t(x,\cdot)}{\pi(\cdot)}-1\right\|_2^2 & = & {\sum_{j=2}^N e^{-2\gamma_j t}|\phi_j(x)|^2}\\
\label{gamma2}
& \le & e^{-2\gamma_2 t}\left(\frac{1}{\pi(x)}-1\right),
\end{eqnarray}
thanks to the crude bound $\gamma_j\ge \gamma_2$ for all $2\le j\le N$. Combining those two inequalities readily leads to the result, since $\lambda=\gamma_2$, thanks to (\ref{def:gap}). 
\end{proof}
\begin{remark}[Room for improvement]
The bound $\gamma_j\ge\gamma_2$ used at (\ref{gamma2}) is of course very crude, and leaves plenty of room for improvement. One should keep in mind, however, that the exact formula (\ref{Parseval}) is rarely used at full strength in practice, because most  eigenvalues and eigenvectors do not  admit simple, explicit and tractable expressions on concrete models. 
\end{remark}
\begin{remark}[No cutoff in $L^2$]In the above proof, we have upper bounded $\dtv(t)$ by 
\begin{eqnarray*}
d_2(t) & := & \frac{1}{2}\max_{x\in\dX}\left\|\frac{P_t(x,\cdot)}{\pi(\cdot)}-1\right\|_2.
\end{eqnarray*}
This quantity is naturally called the \red[$L^2-$distance to equilibrium], and it constitutes a perfectly legitimate alternative to the total-variation distance. Unlike the latter however, it can \emph{not} be normalized so as to approach a step function in the sense of Definition \ref{def:cutoff}. Indeed, the spectral expansion (\ref{Parseval}) implies that the function $t\mapsto d_2(t)$ is log-convex, hence convex: it can not  drop abruptly to zero after having remained flat for a long time.
\end{remark}
Let us introduce a substantial relaxation of reversibility, which suffices for many purposes. 
\begin{definition}[Weak reversibility]$T$ is called \red[weakly reversible] when $E$ is symmetric, i.e.,  
\begin{eqnarray}
\label{assume:sym}
\forall x,y\in\dX,\qquad 
T(x,y)>0 & \Longrightarrow & T(y,x)>0.
\end{eqnarray}
\end{definition}
Since the diagram of a weakly reversible chain is undirected, it naturally induces a graph metric on $\dX$, henceforth  called the  \red[hop-count distance], and defined as follows:
\begin{eqnarray}
\label{def:hopcount}
\forall x,y\in\dX,\qquad \dist(x,y) & := & \min\left\{k\in\dN\colon T^k(x,y)>0\right\}.
\end{eqnarray}
 As a consequence, the usual notions pertaining to metric spaces apply to $\dX$, and will be used  repeatedly here. In particular,  the  \red[Lipschitz constant] of a  function $f\colon\dX\to\dR$ is
\begin{eqnarray}
\label{def:lip}
 \Lip(f) & := & \max_{x\ne y}\left\{\frac{|f(x)-f(y)|}{\dist(x,y)}\right\} \ = \ \max_{\{x,y\}\in E}\left\{|f(x)-f(y)|\right\},
 \end{eqnarray}  
while the \red[diameter] of the space is defined as
\begin{eqnarray}
\label{def:diam}
\diam(\dX) & := & \max_{(x,y)\in\dX^2}\dist(x,y),
\end{eqnarray}
 The latter represents a natural \emph{geometric time-scale} for our semi-group, and it will provide a simple  and useful lower bound for many more sophisticated quantities used to quantify mixing. Here is a first instance of this general principle.
\begin{lemma}[Diameter bound]\label{lm:diammix}If $T$ is weakly reversible, then 
\begin{eqnarray*}
\forall\varepsilon\in(0,1),\quad \tmix(\varepsilon) & \ge & \left\lfloor\frac{(1-\varepsilon)^2\diam(\dX)}{10}\right\rfloor.
\end{eqnarray*}
\end{lemma}
\begin{proof}We can assume that $(1-\varepsilon)\diam(\dX)\ge 10$, otherwise the claim  trivializes. Now, consider two integers $k,r\ge 1$, to be adjusted later, satisfying the condition
\begin{eqnarray}
\label{assume:diam}
(k-1)\times(2r-1) & \le & \diam(\dX).  
\end{eqnarray}
Then, along a geodesic path realizing the diameter, one can find $k$ points at  distance at least $2r-1$ from each other. Since the balls of radius $r-1$ around those $k$ points are pairwise disjoint, one of those balls -- call it $B$ -- must satisfy $\pi(B)\le 1/k$. On the other hand, when started from the center of $B$, our Markov process $(X_t)_{t\ge 0}$ satisfies, 
\begin{eqnarray*}
\PP(X_t\notin B) & = & \PP\left(\dist(X_t,X_0)\ge r\right)\
 \le \ \frac{t}{r},
\end{eqnarray*}
for any $t\ge 0$, by Markov's inequality and the fact that the number of transitions occurring in the interval $[0,t]$ is a Poisson random variable with mean $t$. Therefore, we deduce that
\begin{eqnarray*}
\dtv(t) & \ge & \PP(X_t\in B)-\pi(B) \ \ge \ 1-\frac{t}{r}-\frac{1}{k}.
\end{eqnarray*}
To make the right-hand side at least $\varepsilon$, it is enough to choose $k,r,t$ as follows:
\begin{eqnarray*}
k \ :=\ \left\lceil \frac{2}{1-\varepsilon}\right\rceil,\qquad
r \ :=\ \left\lceil \frac{(1-\varepsilon)\diam(\dX)}{5}\right\rceil,\qquad
t \ := \ \left\lfloor\frac{(1-\varepsilon)^2\diam(\dX)}{10}\right\rfloor.
 \end{eqnarray*} 
Note that this is compatible with the constraint (\ref{assume:diam}), since
\begin{eqnarray*}
(k-1)(2r-1) & \le & \frac{2}{1-\varepsilon}\left(\frac{2}{5}\left(1-\varepsilon\right)\diam(\dX)+1\right),
\end{eqnarray*}
which is at most $\diam(\dX)$ thanks to our initial assumption that $(1-\varepsilon)\diam(\dX)\ge 10$. 
\end{proof}
\begin{remark}[Room for improvement]The ballistic estimate $\EE[\dist(X_t,X_0)]\le t$ used in the above argument is very pessimistic compared to the diffusive behavior $\EE[\dist(X_t,X_0)])=O(\sqrt{t})$ observed in various models, including random walk on the $n-$cycle. Similarly, the estimate on the volume of balls can be considerably improved on the $n-$cube and other models. Nevertheless, there are important models in which $\tmix=\Theta(\diam(\dX))$, including random walks on expanders (as defined below).
\end{remark}
\subsection{Random walks on graphs and groups}
\label{sec:RW}
\paragraph{Random walks on graphs.}
Consider a finite, connected, directed  graph $G$ with vertex set 
$V_G$ and  edge set $E_G$. The most natural stochastic dynamics that one could think of on this graph consists in repeatedly traversing an arc chosen uniformly at random among all those that start from the current position. More formally, we define the \red[random walk] on  $G$ as the Markov chain with state space $\dX=V_G$ and transition matrix 
\begin{eqnarray*}
T(x,y) & := & \left\{
\begin{array}{ll}
\frac{1}{\deg(x)} & \textrm{ if }(x,y)\in E_G\\
0 & \textrm{ else},
\end{array}
\right.
\end{eqnarray*}
where $\deg(x):=\#\{y\in V_G\colon (x,y)\in E_G\}$ here denotes the out-degree of the vertex $x$. 
\begin{exo}[Reversibility]Prove that the above Markov chain is reversible if and only if the graph is undirected (i.e. $(x,y)\in E_G\Longleftrightarrow (y,x)\in E_G$), and explicitate $\pi$ in that case. 
\end{exo}
The time it takes for a random walk to forget its initial position is a gauge for an array of properties of the underlying graph: typical distances, local traps, global bottlenecks, high-degree hubs, etc. Moreover, the ability of random walks to quickly approximate their equilibrium law plays a prominent role in many modern applications, such as error-correcting codes, pseudo-random generators, and link-based ranking of massive databases. Those applications have motivated the detailed analysis of random walks on \emph{random graphs}, which serve as good models for complex networks. Over the past decade, our understanding of their  asymptotic geometry and, in particular, of their intrinsic tree-like nature, has become sufficiently developed to obtain a very precise description of  typical random-walk trajectories. While this has led to rigorous proofs of cutoff in several models  \cite{MR2667423,MR3758735,MR3650414,MR3773804,MR4132638,MR4385126,MR4476123,hermon2021cutoff,arxiv.2212.04469}, the common mechanism at play behind all those instances has not yet been identified. 
\begin{problem}[Derandomization]\label{pb:derandom}Find a structural property which is almost-surely satisfied by the various random graphs considered above, and which deterministically triggers  cutoff. 
\end{problem}
\paragraph{Vertex-transitive graphs.} Vertex-transitive graphs form a distinguished ensemble with enough structure and symmetry to raise hopes for a better understanding of cutoff. Recall that a graph $G=(V_G,E_G)$  is \red[vertex-transitive] if for any vertices $x,y\in V_G$, there is an edge-preserving bijection $\phi\colon V_G\to V_G$ that maps $x$ to $y$. In words, $G$ \emph{looks the same from every vertex}. 
This strong spatial homogeneity precludes many pathological phenomena observed in more heterogeneous settings, and entails considerably simplified expressions for a number of random-walk statistics \cite{MR994089}. Very recently, a tremendous progress in our understanding of the structure of finite vertex-transitive graphs was achieved by R. Tessera and M. Tointon in the remarkable series of papers \cite{arxiv.2001.01467,MR4253426}. 
Yet and perhaps surprisingly,  no general positive result is known regarding cutoff on this natural ensemble. This problem was explicitly raised  two decades ago at the AIM Workshop on Mixing Times \cite[Conjecture 7]{peresamerican}.
\begin{problem}[Vertex-transitive graphs]\label{pb:transitive}Characterize cutoff on vertex-transitive graphs.
\end{problem}

\paragraph{Expander graphs.}\red[Expanders] are sequences of finite undirected graphs  $(G_n)_{n\ge 1}$ whose sizes diverge, whose degrees are uniformly bounded, and whose spectral gaps are bounded away from zero. Those fascinating graphs enjoy nearly as good connectivity properties as complete graphs, but at a much lower cost in terms of edges. Consequently, they have found numerous practical applications,  some of which are described in the beautiful survey paper \cite{MR2247919} by S. Hoory, N. Linial and A. Wigderson. It follows from Lemmas \ref{lm:L2bound} and \ref{lm:diammix} that the mixing time of expanders scales logarithmically with the number of vertices, while their spectral gap is by definition bounded away from $0$. In particular,  the product condition is satisfied. Consequently, expanders constitute natural (and highly valuable) targets for cutoff hunters. In particular,  proving cutoff on vertex-transitive expanders is one of the most famous open problems in the field  (see \cite[Open Question 34]{peresamerican} or \cite[Question 5]{MR3726904}), but positive results remain  frustratingly limited to the extremal case of (near-)\emph{Ramanujan graphs} \cite{MR3558308,MR3693771,MR4178418,MR4476123}, whose remarkable structure is very well understood. \begin{problem}[Expanders]\label{pb:exp}Prove that all vertex-transitive expanders exhibit cutoff.
\end{problem}

\paragraph{Random walks on groups.} One place where vertex-transitive graphs arise naturally is group theory: let $\dX$ be a finite group and let $S\subseteq \dX$ be a set of elements that generates the group. By definition, the \red[Cayley graph]  $G=\mathrm{Cay}(\dX,S)$ is the graph with vertex set $V_G=\dX$ and edge set $E_G=\left\{(x,sx)\colon x\in\dX,s\in S\right\}$. This graph is vertex-transitive, since for any $(x,y)\in\dX^2$ the map $\sigma_{xy}\colon z\mapsto zx^{-1}y$ is an edge-preserving bijection that sends $x$ to $y$. Note that a transition of the random walk  on $G$ simply consists in left-multiplying the current position by an element drawn uniformly at random from $S$.   Rather embarrassingly, our understanding of cutoff for random walks on Cayley graphs is  far from complete even in the special case of  Abelian groups, of which our two toy models (the $n-$cycle and $n-$cube) are particular instances. In particular, the following reasonable prediction  was made at the aforementioned workshop \cite[Open Question 33]{peresamerican}, and reiterated  recently \cite[Conjecture 1]{hermon2021cutoff}. 
Its confirmation  would represent a substantial progress on Problems \ref{pb:PC} and \ref{pb:transitive}. 
\begin{problem}[Abelian graphs]\label{pb:abelian}Prove that cutoff is equivalent to the product condition for random walks on Cayley graphs of Abelian groups.
\end{problem} 
Random walks on Cayley graphs admit a fairly natural ``weighted'' generalization: instead of being drawn uniformly at random from a generating set $S$, the increments can be sampled according to any probability distribution $\nu\in\cP(\dX)$ whose support generates the group.  More formally, this corresponds to the transition matrix
\begin{eqnarray}
\label{def:RWgroup}
T(x,y) & := & \nu(yx^{-1}),
\end{eqnarray}
and the resulting Markov process (started from the identity element w.l.o.g.) is naturally called the \red[random walk on $\dX$ generated by $\nu$]. Understanding how the mixing time of this random walk depends on the group $\dX$ and on the  law $\nu$ is a rich and active research area, but our current understanding of cutoff remains limited to a handful of special cases. 
\subsection{Interacting particle systems} \label{sec:interacting}Broadly speaking, interactive particle systems are continuous-time Markov processes that model the collective behavior of stochastically interacting components \cite{Liggettbook1,Liggettbook2}. The most widely studied example is the \red[Exclusion Process], introduced by F. Spitzer \cite{Spitzer} in 1970. 
\paragraph{Symmetric Exclusion Processes.} Imagine particles performing independent random walks on the vertices of a network, except that they are not allowed to overlap: any jump that would violate this \emph{exclusion rule}  is simply canceled. More formally, the state space is 
\begin{eqnarray*}
\dX & := & \left\{x\in\{0,1\}^n\colon \sum_{i=1}^n x_i=m\right\},
\end{eqnarray*}
where $n>m$ are positive integers representing the number of sites and particles respectively, and where the boolean variable $x_i\in\{0,1\}$ indicates whether the $i-$th site is empty or occupied. The underlying geometry is specified by an irreducible symmetric stochastic matrix $G=(G_{ij})_{1\le i,j\le n}$, which is often chosen to be the transition matrix of simple random walk on a regular graph. The generator then acts on any function $f\colon\dX\to\dR$ as follows:
\begin{eqnarray*}
Lf(x) & := & \frac{1}{n}\sum_{1\le i,j\le n}G_{ij}x_i(1-x_j)\left(f(x+\ee_j-\ee_i)-f(x)\right),
\end{eqnarray*}
where $(\ee_1,\ldots,\ee_n)$ denotes the canonical basis. This dynamics is easily seen to be reversible w.r.t. the uniform measure on $\dX$. A particularly pleasant feature of this model is the explicitly solvable nature of its single-site marginals: if $(X_t)_{t\ge 0}$ is a Markov process with the above generator, then its mean $\mu_t:=\EE[X_t]$ evolves according to the \red[discrete heat equation]
\begin{eqnarray*}
\frac{\dd \mu_t}{\dd t} & = & (G-\mathrm{Id})\mu_t,
\end{eqnarray*}
independently of the number of particles. Another striking property, conjectured by D. Aldous and proved by P. Caputo, T. Liggett and T. Richthammer \cite{MR2629990}, is the highly non-trivial fact that  the spectral gap of the process does not depend on the number $m$ of particles. Those remarkable features suggest a deep and intimate relation between the convergence to equilibrium of the Exclusion Process and that of  independent random walkers. Formalizing this idea has been and continues to be the subject of active research \cite{MR3077529,MR4164461,MR4164852,MR4254474}. In particular, the occurrence of a cutoff was conjectured long ago by D. Wilson on various sequences of graphs  \cite[Section 9]{MR2023023}, along with precise candidates for its location, width and profile. Frustratingly, those predictions have been so far only  confirmed on the segment and the cycle, in a remarkable series of papers by H. Lacoin \cite{MR3474475,MR3551201,MR3689972} that inspired many closely related results \cite{MR3945753,MR4021252,MR4291457,arxiv.2003.03781,arxiv.2104.10478,arxiv.2110.06353,arxiv.2201.03463}. The proof relies on a very special monotonicity property of the model which, unfortunately, fails in higher dimensions. 
\begin{problem}[Exclusion]\label{pb:exclusion}Establish cutoff for the Symmetric Exclusion Process on growing lattices or tori in dimension $d\ge 2$. 
\end{problem}

\paragraph{Zero-Range Processes.} Another widely-studied conservative interacting particle system -- also introduced by F. Spitzer, in the very same paper -- is the \red[Zero-Range Process] (ZRP). Just like the Exclusion Process, it involves a fixed number $m$ of particles evolving on a network whose geometry is specified by an irreducible stochastic matrix $G=(G_{ij})_{1\le i,j\le n}$.  However, those particles are now allowed to overlap: the state space becomes
\begin{eqnarray*}
\dX & := & \left\{x\in\dZ_+^n\colon \sum_{i=1}^nx_i = m\right\},
 \end{eqnarray*} 
and the exclusion rule is replaced by an interaction function $r_i\colon\{1,2,\ldots\}\to (0,\infty)$ at each site $i$, indicating how the jump rate out of $i$ depends on the number of particles present on $i$. More formally, the generator of the process acts on functions $f\colon\dX\to\dR$ as follows:
\begin{eqnarray*}
Lf(x) & := & \frac{1}{n}\sum_{1\le i,j\le n}G_{ij}r_i(x_i)\left(f(x+\ee_j-\ee_i)-f(x)\right),
\end{eqnarray*}
with the interpretation that $r_i(0)=0$ (no jumps from empty sites). 
An interesting feature of the model is that the stationary distribution $\pi$ has an explicit product form,  whose single-site marginals can be freely adjusted by an appropriate choice of the parameters:
\begin{eqnarray*}
 \pi(x) & \propto & \prod_{i=1}^n\frac{p_i^{x_i}}{r_i(1)r_i(2)\cdots r_i(x_i)},
 \end{eqnarray*} 
where $p=(p_i)_{i=1}^n$ is the stationary probability vector of $G$.   Also, by playing with the choice of the functions $(r_i)_{1\le i\le n}$, a variety of interesting physical behaviors can be reproduced. For example,  super-linear interactions will make the particles accelerate upon meeting, inducing \emph{repulsion} and a tendency to spread evenly across the state space. On the contrary, sub-linear interactions will make the particles slow down upon meeting, inducing \emph{attraction} and a tendency to accumulate in macroscopic proportion on a single site (condensation). Quantifying the impact of those phenomena on the mixing time of the system is a natural and important question, which has received considerable attention; see, e.g., \cite{Cap,MR2184099,MorrisZRP,MR3704767,MR3984254,MR4332696} and the references therein. Contrastingly, the delicate problem of establishing cutoff has so far only been tackled in the \red[mean-field case], where $G_{ij}=\frac 1n$ for all $1\le i,j\le n$ \cite{MR4021248,MR4089493,arxiv.2104.10478}. 

\begin{problem}[Zero-Range Processes]\label{pb:zrp}Establish cutoff for the Zero-Range Process with unit rate on growing lattices/tori in dimension $d\ge 2$. 
\end{problem}
\subsection{Monte Carlo Markov Chains} The celebrated \emph{Markov chain Monte Carlo revolution} in computational statistics is fundamentally based on the simple but far-reaching idea -- attributed to Metropolis \cite{doi:10.1063/1.1699114} and Hastings \cite{MR3363437} -- that approximate samples from a target probability distribution $\pi$ can be efficiently produced by running an appropriate Markov chain on $\dX=\supp(\pi)$, which admits $\pi$ has its equilibrium law; see the survey paper by P. Diaconis \cite{MR2476411} and the references therein. To keep things simple and concrete, let us here assume that $\pi$ lives on $\{-1,1\}^n$ or $\{0,1\}^n$ and that the transitions are coordinate flips. In other words, the generator has the form
\begin{eqnarray}
\label{def:Glauber}
\forall x\in\dX,\quad Lf(x) & := & \sum_{i=1}^n c_i(x)\left(f(x^i)-f(x)\right),
\end{eqnarray}
where $x^i$ denotes the state obtained from $x$ by flipping the $i-$th coordinate, and where the rates  $c_i\colon\dX\to[0,\infty)$ satisfy the reversibility condition
\begin{eqnarray*}
\forall x\in\dX,\quad \pi(x)c_i(x) & = & \pi(x^i)c_i(x^i).
\end{eqnarray*}
Here are three widely-used implementations that share a big advantage: they  only involve the ratio $\pi(x^i)/\pi(x)$, whose computation is significantly less costly than that of $\pi$ itself.
\begin{eqnarray}
\label{def:rates}
c_i(x) & \propto & \left\{
\begin{array}{ll}
\displaystyle{\frac{\pi(x^i)}{\pi(x)+\pi(x^i)}} &  (\textrm{Gibbs sampler})\medskip\\
\displaystyle{1\wedge\frac{\pi(x^i)}{\pi(x)}} & (\textrm{Metropolis sampler})\medskip\\
\displaystyle{\sqrt{\frac{\pi(x^i)}{\pi(x)}}} & (???).
\end{array}
\right.
\end{eqnarray}
Note that, in the basic case where $\pi$ is uniform on $\{-1,+1\}^n$, those implementations all reduce to our favorite toy model: random walk on the $n-$cube. Since the latter exhibits cutoff, it is  natural to expect a similar phenomenon when sampling from more general high-dimensional measures with weak dependencies, such as high-temperature spin systems \cite{MR1746301}.

\begin{problem}[Cutoff under weak dependencies]\label{pb:glauber}Find a weak-dependency condition on the sequence of target measures $(\pi_n)_{n\ge 1}$, where $\pi_n\in\cP(\{-1,1\}^n)$, under which cutoff occurs for the three MCMC implementations described above. 
\end{problem}

For concreteness, let us now describe two emblematic special cases of this problem.

\paragraph{Ising model.} The best-known example of a spin system is certainly the \red[Ising model], which was introduced a century ago by the physicists Ernst Ising and Wilhelm Lenz to explain ferromagnetism. In the above language, it corresponds to the probability measure
\begin{eqnarray*}
\pi(x) & := & \frac{1}{Z}\exp\left(\sum_{1\le i,j\le n}G_{ij}x_ix_j\right),
\end{eqnarray*}
where $G$ is a $n\times n$ matrix describing the geometry of interactions -- typically, the adjacency matrix of a $n-$vertex graph --   and $Z$ a (highly non-trivial) normalizing constant. The \red[Curie-Weiss model] is the mean-field case where  $G_{ij}=\frac{\beta}{2n}$ for all $1\le i,j\le n$ and some \red[inverse temperature] $\beta\in\dR$. When $\beta=0$, the coordinates are simply independent and the dynamics (\ref{def:Glauber}) reduces to  random walk on the $n-$cube, which mixes in time $\Theta(n\log n)$ and  exhibits cutoff. Interestingly, the same conclusion actually persists throughout the low-interaction phase $\beta<1$, but when $\beta>1$ mixing becomes exponentially slow and cutoff is destroyed \cite{levin2010glauber}. Establishing a similar dynamical phase transition beyond the mean-field case is of course far more challenging.   In a remarkable series of works \cite{MR3020173,lubetzky2014cutoff,MR3434254,MR3486171,MR3729612}, E. Lubetzky and A. Sly developed a  framework called \emph{Information Percolation} which enabled them to establish cutoff for the Ising model on finite lattices, throughout the regime of \emph{strong spatial mixing}. 
 This ground-breaking approach also yields cutoff on other geometries, but under a strong decoupling assumption which only holds  when the temperature is high enough. Overcoming this intrinsic limitation is a challenging problem, with numerous practical motivations.
Very recently, the development of spectral independence and entropy factorization has resulted in the confirmation of the long-standing prediction that Gibbs sampling is rapidly mixing throughout the so-called \emph{tree uniqueness} regime  (see \cite{MR4232133,MR4398939,MR4415182} and the references therein). A natural but far-reaching conjecture, suggested by this breakthrough, is that cutoff should occur under the very same structural assumptions. 

\begin{problem}[Cutoff for the Ising Model]\label{pb:ising}Establish cutoff throughout the rapidly-mixing regime for the Gibbs sampler of the Ising Model on bounded-degree graphs. 
\end{problem}
Of course, the same question can be asked for other emblematic spin systems.
\paragraph{Hardcore model.} Given a finite graph $G=(V_G,E_G)$, an \red[independent set] (also known as a stable set, or an anticlique) is a subset of vertices in $G$, no two of which are adjacent. An example is illustrated on Figure \ref{fig:hardcore}. Finding a large independent set is one of the most fundamental examples of a computationally hard task. It has a broad range of applications, from scheduling problems to graph coloring, communications and coding theory. 
\begin{figure}[!h]
\begin{center}
\label{fig:hardcore}
\includegraphics[scale=.8]{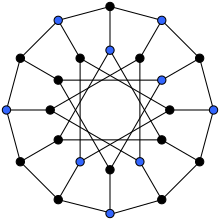}
\caption{The blue vertices form an independent set with the largest possible size.}
\end{center}
\end{figure}
The \red[Hardcore model] with fugacity $\zeta>0$ is the probability measure on $\{0,1\}^{V_G}$ defined by
\begin{eqnarray}
\pi(x) & \propto & \zeta^{\sum_{i\in V_G} x_i}\prod_{\{i,j\}\in E_G}(1-x_ix_j).
\end{eqnarray}
This distribution is supported on the independent sets of $G$, and the parameter $\zeta$  controls their typical size. In particular, $\pi$ concentrates on the empty set in the  $\zeta\to 0$ limit, and converges to the uniform distribution on maximum independent sets in the $\zeta\to\infty$ limit.  For moderate values of $\zeta$, running one of the samplers presented at (\ref{def:rates}) -- which are  equivalent up to scaling in this case -- appears to be one of the most  efficient ways to produce a large independent set in a graph. For this reason, understanding how the mixing time of this process depends on $G$ and $\zeta$ has been the subject of many works. When the degrees are bounded and the fugacity is low enough, the \emph{Information Percolation} machinery can be implemented to establish cutoff \cite{lubetzky2014cutoff}, but going beyond this perturbative regime seems hard.
\begin{problem}[Cutoff for the Hardcore Model]\label{pb:hardcore}Establish cutoff throughout the rapid-mixing regime for the Gibbs sampler of the Hardcore Model on bounded-degree graphs. 
\end{problem}
\subsection{Markov diffusions}
\label{sec:manifolds}
Although this course is primarily concerned with Markov chains on finite state spaces, we will occasionally allow ourselves to take a step aside and examine what happens in the infinitely smooth setup of Markov diffusions on Euclidean spaces, or more generally, weighted manifolds. Beyond their considerable practical interest, such processes happen to satisfy a convenient chain rule that makes their mathematical analysis particularly simple and elegant. They will therefore constitute excellent \emph{ideal models}, which we will use to develop a better intuition of what may be expected in our discrete setup. However, to keep the present lecture notes at an elementary level (or more honestly, because of the author's incompetence), we will \emph{not} attempt to  enter the technical details of their rigorous foundations, and refer the interested reader  to the beautiful expository textbooks \cite{MR889476,MR3155209} for a comprehensive treatment.  

\paragraph{Overdamped Langevin diffusions.} Consider the $d-$dimensional Euclidean space $\dX=\dR^d$, equipped with its Borel $\sigma-$algebra and a probability measure of the form
\begin{eqnarray}
\label{mu}
\pi(\dd x) & = & e^{-U(x)}\, \dd x,
\end{eqnarray}
where $U\colon\dX^d\to\dR$ is a smooth function called the \red[potential]. The associated \red[overdamped Langevin dynamics] is a celebrated model of molecular motion, as well as a widely-used algorithm for sampling from $\pi$. It is described by the stochastic differential equation
\begin{eqnarray}
\label{SDE}
\dd X_t & = & -\nabla U(X_t)\,\dd t+ \sqrt{2}\,\dd B_t,
\end{eqnarray}
where $(B_t)_{t\ge 0}$ is a standard $d-$dimensional Brownian motion. Under mild conditions on $U$, there is a unique solution $(X_t^x)_{t\ge 0}$ for from each initial state $x\in\dX$, and the formula
\begin{eqnarray}
\label{diff:Pt}
P_tf(x) & := & \EE\left[f(X_t^x)\right],
\end{eqnarray}
 defines a reversible Markov semi-group on $L^2(\dX,\pi)$ whose generator  acts  as follows:
\begin{eqnarray}
\label{def:L}
Lf & = & \Delta f - \nabla U \cdot \nabla f,
\end{eqnarray}
for any smooth function $f\colon\dX\to\dR$. An immediate consequence of this second-order differential expression is the following convenient \red[chain rule] for composed functions: 
\begin{eqnarray}
\label{def:CR}
L(\phi\circ f) & = & (\phi'\circ f)Lf + (\phi''\circ f)|\nabla f|^2,
\end{eqnarray}
for all smooth functions $f\colon\dX\to\dR$ and $\phi\colon\dR\to\dR$. As we will see, this identity makes the analysis of diffusions much more pleasant, in several ways,  than that of discrete models. 
\paragraph{Extension to manifolds.}
Much more generally, the  Euclidean space $\dR^d$  can be replaced with any weighted Riemannian manifold, provided  the formulae (\ref{mu}-\ref{def:CR}) are interpreted in an appropriate way. Specifically, consider a Riemannian manifold $\dX$ equipped with a metric tensor $g$, and let $\pi$ be a probability measure on of the form (\ref{mu}), where ${\rm d} x$ is the volume measure on $\dX$ and $U\colon \dX\to\dR$ a geodesically semi-convex function. Now, consider the Langevin equation (\ref{SDE}), with $\nabla$ denoting the gradient in  $(\dX,g)$ and $(B_t)_{t\ge 0}$ a Brownian motion thereon.  Then, the formula (\ref{diff:Pt})  defines a Markov semi-group of self-adjoint operators on $L^2(\dX,\pi)$ whose generator  acts as in (\ref{def:L}), with $\Delta$  denoting the Laplace-Beltrami operator on $\dX$. The chain rule (\ref{def:CR}) remains valid, and we again refer to \cite{MR889476,MR3155209} for details.

\newpage 
\section{Functional inequalities}
One of the most powerful techniques for bounding mixing times consists in establishing an appropriate \red[functional inequality] that quantifies the \emph{regularizing nature} of the semi-group. This can be done in several non-equivalent ways, leading to a jungle of functional-analytic constants that can be difficult to navigate through. To help the new comer, we have here opted for a unified presentation that starts with an abstract ``master theorem'' (Theorem \ref{th:analytic} below), from which the relevant functional inequalities can  then be easily extracted.
\subsection{A general recipe}
Throughout the notes, we  write $\cF$ for the set of real-valued functions on $\dX$, and 
\begin{eqnarray}
\bE[f] & : = & \sum_{x\in\dX}\pi(x)f(x),
\end{eqnarray}
for the stationary mean of $f\in\cF$. We shall add the superscript $+$ (resp. $++$) when we want to restrict  $\cF$  to non-negative (resp. positive) functions, and the subscript $m\in\dR$  to impose the constraint $\bE[f]=m$. In particular, $\cF_1^{+}$ denotes the set of \red[density functions], which can be identified with the set of probability measures on $\dX$ through the  correspondence 
\begin{eqnarray*}
f & = &  \frac{\dd \mu}{\dd \pi}.
\end{eqnarray*}
The functional-analytic approach to mixing times is based on the following trivial observation: rather than studying the convergence  of the marginal laws of our Markov process $(X_t)_{t\ge 0}$ towards  $\pi$, we may equivalently investigate the convergence of the corresponding densities $(f_t)_{t\ge 0}$ towards  $1$. Those densities evolve according to a  well-known   equation, which we recall in the following exercise.
\begin{exo}[Fokker-Planck equation]Prove that $(f_t)_{t\ge 0}$ solves the \red[Fokker-Planck equation]
\begin{eqnarray}
\label{FP}
\frac{\dd f_t}{\dd t} &  =  & L^\star f_t,
\end{eqnarray}
or equivalently, that $f_t=P_t^\star f_0$ for all $t\ge 0$. 
\end{exo}
The fact that our Markov process mixes (Theorem \ref{th:mixing}) can thus be formulated in functional-analytic terms by saying that  densities \emph{flatten} under the action of the  semi-group:
\begin{eqnarray*}
\forall f\in\cF_1^+,\qquad P_t^\star f  & \xrightarrow[t\to\infty]{} & 1.
\end{eqnarray*}
In order to quantify this regularizing effect, we need to agree on the choice of a  functional $\Phi\colon\cF_1^+\to[0,\infty)$ that measures the variations of functions. Concrete examples will be given soon enough, but for the time being, just think of $\Phi(f)$ as a non-negative quantity that vanishes precisely when $f$ is constant. In light of the semi-group property, it is  reasonable to hope that $\Phi$ will decay exponentially under our dynamics. We are thus naturally led to the task of finding a constant $\kappa\ge 0$, as large as possible, such that
\begin{eqnarray}
\label{generic:decay}
\forall f\in\cF_1^+,\qquad\forall  t\ge 0,\qquad \Phi(P_t^\star f) & \le & e^{-\kappa t}\Phi(f).
\end{eqnarray}
Fortunately, this strong dynamical property  admits a rather simple static characterization. In order to formulate it, we   introduce  the following important object.
\begin{definition}[Dirichlet form] The \red[Dirichlet form] is the quadratic form on $L^2$ defined by
\begin{eqnarray*}
\cE(f,g) & := & \langle f,-Lg\rangle \ = \ \sum_{x\in\dX}\pi(x)T(x,y)f(x)\left(g(x)-g(y)\right).
\end{eqnarray*}
When $T^\star=T$ or  $f=g$, we also have the useful symmetrized expression
\begin{eqnarray}
\label{rk:sym}
\cE(f,g) & = & 
\frac{1}{2}\sum_{(x,y)\in\dX^2}\pi(x)T(x,y)\left(f(x)-f(y)\right)\left(g(x)-g(y)\right).
\end{eqnarray}
\end{definition}
We henceforth assume that $\Phi$ is \emph{smooth}, in the sense that  it is continuous on $\cF_1^+$ and that it admits, at each point  $f_0\in\cF_1^{++}$, a first-order expansion of the form
\begin{eqnarray}
\label{assume:smooth}
\Phi(f) & = & \Phi(f_0)+\langle \nabla\Phi(f_0),(f-f_0)\rangle + o(\|f-f_0\|_2),
\end{eqnarray}
as $f\to f_0$ in $\cF_1^+$, for some continuous function $\nabla\Phi\colon\cF_1^{++}\to \cF$. The norm and scalar product appearing here are of course those of $L^2(\dX,\pi)$. Finally, we let $f^x:=\frac{{\bf 1}_{\{x\}}}{\pi(x)}$ denote the Dirac density function  at a state $x\in\dX$. The following result summarizes, in a compact and abstract form, how functional inequalities arise in the study of mixing times. 
\begin{theorem}[The functional-analytic approach to mixing times]\label{th:analytic}For each $\kappa\ge 0$, the exponential decay (\ref{generic:decay}) is equivalent to the static criterion
\begin{eqnarray}
\label{generic:ineq}
\forall f\in{\cF_1^{++}},\qquad \cE(f,\nabla\Phi(f)) & \ge &\kappa\Phi(f).
\end{eqnarray}
Moreover, when this holds, we have the mixing-time estimate
\begin{eqnarray*}
\forall\varepsilon\in(0,1),\quad \tmix(\varepsilon) & \le & \frac{1}{\kappa}\log\frac{\Delta}{\delta(\varepsilon)},
\end{eqnarray*}
where we have introduced the quantities
\begin{eqnarray*}
\Delta\ :=\ \max_{x\in\dX}\Phi\left(f^x\right), &  & \delta(\varepsilon) \ := \ \min\left\{\Phi\left(\frac{\dd \mu}{\dd \pi} \right)\colon \mu\in\cP(\dX),\dtv(\mu,\pi)\ge \varepsilon\right\}.
\end{eqnarray*}
\end{theorem}
\begin{proof}Fix $f\in\cF_1^+$. The smoothness of $\Phi$ together with the Fokker-Planck equation (\ref{FP}) ensure that  $u\colon t\mapsto \Phi(P_t^\star f)$ is continuously differentiable on $(0,\infty)$, with derivative
\begin{eqnarray*}
u'(t) & = & \langle \nabla\Phi(P_t^\star f),L^\star P_t^\star f\rangle \ = \ -\cE(P_t^\star f,\Phi(P_t^\star f)).
\end{eqnarray*}
In particular, the condition (\ref{generic:ineq}) gives the differential inequality $u'\le -\kappa u$, whose integration yields exactly (\ref{generic:decay}). Conversely, if  (\ref{generic:decay}) holds, then a first-order Taylor expansion at $t=0$ on both sides shows that (\ref{generic:ineq}) must also hold. Finally, the definition of $\delta$ implies that
\begin{eqnarray*}
\forall \mu\in\cP(\dX),\qquad \delta\left(\dtv(\mu,\pi)\right) & \le & \Phi\left(\frac{\dd \mu}{\dd \pi}\right).
\end{eqnarray*}
In particular, choosing $\mu=P_t(x,\cdot)$, applying 
 (\ref{generic:decay}), and then maximizing over $x\in\dX$ yields
 \begin{eqnarray*}
\forall t\ge 0,\qquad \delta\left(\dtv(t)\right) & \le & \Delta e^{-\kappa t}.
\end{eqnarray*}
Evaluating this at time $t=\tmix(\varepsilon)$ produces the desired mixing-time estimate.
\end{proof}
\subsection{Variance and entropy}
\label{sec:poincaré}
To implement the general strategy summarized in Theorem \ref{th:analytic}, we need to choose a concrete functional $\Phi$ that quantifies the variations of functions on our probability space $(\dX,\pi)$. The first choice that comes to mind is of course the \red[variance] $\Phi=\Var$, defined for all $f\in\cF$ by
\begin{eqnarray*}
\Var(f) & := & \bE[f^2]-\bE^2[f].
\end{eqnarray*}
Note that, by the Cauchy-Schwartz inequality, we have
\begin{eqnarray*}
4\dtv^2(\mu,\pi) & = & \left\|\frac{\dd\mu}{\dd\pi}-1\right\|_1^2 \ \le \  \left\|\frac{\dd\mu}{\dd\pi}-1\right\|_2^2 \ = \  \Var\left(\frac{\dd\mu}{\dd\pi}\right),
\end{eqnarray*}
showing that the function $\delta(\cdot)$ in Theorem \ref{th:analytic} satisfies $\delta(\varepsilon)\ge 4\varepsilon^2$. Moreover, we easily compute $\nabla\Var(f)  =  2(f-\bE[f])$ for any $f\in\cF$, and  $\Var(f^x)=\frac{1-\pi(x)}{\pi(x)}$ for any  $x\in\dX$. Thus, our meta-theorem  here takes the following particularly simple form.
\begin{coro}[Variance decay]\label{co:var}For each $\gamma\ge 0$, the following conditions are equivalent:
\begin{enumerate}[(i)]
\item The \red[Poincaré inequality] with constant $\gamma$ holds: 
\begin{eqnarray}
\label{def:poincaré}
\forall f\in\cF,\quad \cE(f,f) & \ge &  \gamma \Var(f).
 \end{eqnarray}
\item The variance decays  at rate $2\gamma$: 
\begin{eqnarray*}
\forall f\in\cF,\quad\forall t\ge 0,\quad \Var(P_t^\star f) & \le & e^{-2\gamma t}\Var(f).
\end{eqnarray*} 
 \end{enumerate}
Moreover, when this holds, we have the mixing-time estimate
 \begin{eqnarray}
 \label{L2mixbound}
 \tmix(\varepsilon) & \le &  \frac{1}{2\gamma}\log\left(\frac{1-\pi_\star}{4\varepsilon^2\pi_\star}\right).
 \end{eqnarray}
\end{coro}
\begin{definition}[Poincaré constant]The \red[Poincaré constant] of the chain, henceforth denoted by $\gamma$, is the largest constant such that the equivalent conditions (i)-(ii) hold.
\end{definition}
\begin{remark}[Spectral characterization]Because both sides of (\ref{def:poincaré}) are quadratic forms in $f$, the Poincaré constant $\gamma$ is is easily seen to coincide with the spectral gap of the symmetrized matrix $T_{\mathrm{sym}}:=(T+T^\star)/2$. In particular, when $T$ is reversible, we have
\begin{eqnarray*}
\gamma & = & \lambda.
\end{eqnarray*}
Thus, the mixing-time estimate (\ref{L2mixbound})  can be seen as a generalization of Lemma \ref{lm:L2bound}.
\end{remark}
For concreteness, let us see how the bound (\ref{L2mixbound})   performs on our toy models. 
\begin{example}[Random walk on the $n-$cycle]Consider the special case of simple random walk on the $n-$cycle. It is straightforward to check that the family $(\phi_k)_{1\le k\le n}$ defined by
\begin{eqnarray}
\phi_k(x) & := & e^{\frac{2i\pi kx}n},
\end{eqnarray}
forms an orthonormal basis of $L^2_{\mathbb C}(\dX,\pi)$ with $T\phi_k=\cos\left(\frac{2\pi k}{n}\right)\phi_k$. In particular, the Poincaré constant of this chain is $\gamma=1-\cos\left(\frac{2\pi}n\right)$, and (\ref{L2mixbound}) gives
\begin{eqnarray*}
\tmix^{(n)}(\varepsilon) & \le & \frac{n^2\log n}{4\pi^2}+O_\varepsilon(n^2),
\end{eqnarray*}
as $n\to\infty$, which is off by a logarithmic factor when compared with  (\ref{ex:cycle}).
\end{example}

\begin{example}[Random walk on the $n-$cube]Consider the special case of simple random walk on the $n-$cube. It is straightforward to check that the family $(\phi_S)_{S\subseteq[n]}$ defined by
\begin{eqnarray}
\phi_S(x) & := & (-1)^{\sum_{i\in S}x_i},
\end{eqnarray}
forms an orthonormal basis of $L^2(\dX,\pi)$ and that $T\phi_S=\left(1-\frac{2|S|}{n}\right)\phi_S$. Consequently, the Poincaré constant of this chain is $\gamma=2/n$, so that (\ref{L2mixbound}) gives
\begin{eqnarray*}
\tmix^{(n)}(\varepsilon) & \le & n^2\log 2+O_\varepsilon(n),
\end{eqnarray*}
as $n\to\infty$, which is rather disappointing when compared with the asymptotics   (\ref{ex:cube}).
\end{example}
This second example clearly demonstrates the need for a better functional inequality. Instead of bounding an $L^1$ norm (total-variation) with an $L^2$ norm (variance), we will now use an $L\log L$ norm (entropy). More precisely, we  define the \red[entropy] of  $f\in\cF_+$ as follows:
\begin{eqnarray}
\Ent(f) & := & \bE[f\log f]-\bE[f]\log\bE[f],
\end{eqnarray}
with the usual convention $0\log 0=0$. The strict convexity of $u\mapsto u\log u$ ensures that this quantity is non-negative and that it vanishes if and only if $f$ is constant. It thus constitutes a reasonable candidate for our functional $\Phi$, and it turns out to provide a much sharper control on the total-variation distance than the variance. Indeed, we have
\begin{eqnarray*}
\forall\mu\in\cP(\dX),\quad 2\dtv^2(\mu,\pi) \ \le \ \Ent\left(\frac{\dd \mu}{\dd \pi}\right) \ \le \ \log\left[1+\Var\left(\frac{\dd \mu}{\dd \pi}\right)\right].
\end{eqnarray*}
The first inequality is the celebrated \red[Csiszár-Kullback-Pinsker inequality] (see, e.g., \cite[Theorem 4.19]{MR3185193}), while the second simply uses the concavity of $\log$. Following our general recipe, we easily compute
$
\nabla\Ent(f)   =  \log\frac{f}{\bE[f]}
$
for all $f\in\cF^{++}$, and $\Ent(f^x)=\log\frac{1}{\pi(x)}$ for all $x\in\dX$. Thus, the entropic analogue of the previous result reads as follows.
\begin{coro}[Entropy decay]For each $\alpha\ge 0$, the following are equivalent:
\begin{enumerate}[(i)]
\item The \red[modified log-Sobolev inequality] (MLSI) holds with constant $\alpha$: 
\begin{eqnarray}
\label{MLSI}
\forall f\in\cF^{++},\quad  \cE(f,\log f) & \ge & \alpha \Ent(f).
 \end{eqnarray}
\item The entropy decays at rate at least $\alpha$: 
\begin{eqnarray}
\forall f\in\cF^+,\quad \forall t\ge 0,\quad \Ent(P_t^\star f) & \le & \Ent(f)e^{-\alpha t}.
\end{eqnarray} 
\end{enumerate}
Moreover, when this holds, we have the mixing-time estimate
\begin{eqnarray}
\label{Entmixbound}
\tmix(\varepsilon) & \le & \frac{1}{\alpha}\left(\log\frac{1}{2\varepsilon^2}+\log\log\frac{1}{\pi_\star}\right).
\end{eqnarray}
\end{coro}
\begin{definition}[MLS constant]The \red[modified log-Sobolev constant] of the chain, henceforth denoted by   $\alpha$, is the largest constant such that the equivalent conditions (i)-(ii) hold. 
\end{definition}
We refer the unfamiliar reader to the seminal paper \cite{MR2283379} for a more detailed account on this fundamental constant. Note that the dependency on  $1/\pi_{\star}$ in (\ref{Entmixbound}) is now doubly logarithmic, which constitutes a huge improvement over (\ref{L2mixbound}). The downside is that $\alpha$ is generally harder to estimate than $\gamma$, due to the lack of a spectral characterization. A few methods have been developed for doing so, and we will present a powerful one based on couplings in the next chapter. For the time being, let us  admit the following estimates, and admire the considerable gain. 
\begin{example}[Random walk on the $n-$cycle] In the case of random walk on the $n-$cycle, it can be shown that the modified log-Sobolev constant is of order $n^{-2}$, so that (\ref{Entmixbound}) gives
\begin{eqnarray*}
\tmix(\varepsilon) & = & O_\varepsilon(n^2\log \log n).
\end{eqnarray*}
This is much sharper than what we had obtained using the $L^2$ bound (\ref{L2mixbound}), albeit still off by a slowly diverging factor when compared to  the true asymptotics found at (\ref{ex:cube}).
\end{example}
\begin{example}[Random walk on the $n-$cube] In the case of random walk on the $n-$cube, it can be shown that the modified log-Sobolev constant is exactly $\alpha=4/n$, so that (\ref{Entmixbound}) gives
\begin{eqnarray*}
\tmix(\varepsilon) & \le & \frac{n\log n}{4}+O_\varepsilon(n).
\end{eqnarray*}
In view of (\ref{ex:cube}), this is remarkably sharp.
\end{example}
\begin{remark}[Linearization]\label{rk:linearization}The expansion $(1+u)\log(1+u)-u\sim\frac{u^2}{2}$ as $u\to 0$ gives
\begin{eqnarray*}
\forall f\in\cF,\quad \varepsilon^{-2}\Ent(1+\varepsilon f) & \xrightarrow[\varepsilon\to 0]{} & \Var(f).
\end{eqnarray*}
As a consequence, many functional inequalities that involve entropy automatically imply their variance counterparts. A simple example is exponential decay: by definition of $\alpha$ we have $\Ent(P_t^\star f)\le e^{-\alpha t}\Ent(f)$ for all $f\in\cF^+_1$, and by taking $f=1+\varepsilon g$ and letting $\varepsilon\to 0$, we deduce that $\Var(P_t^\star g)\le e^{-\alpha t}\Var(g)$ for any $g\in\cF_0$, hence for any $g\in\cF$. This establishes 
\begin{eqnarray}
\label{MLSIvsP}
\alpha & \le & 2\gamma,
\end{eqnarray}
and the example of the $n-$cube  shows that this is sharp.
\end{remark}
Beyond the fact that it provides very sharp bounds on mixing times, the MLSI happens to have a remarkable implication in terms of concentration, as we will see next.

\subsection{Sub-Gaussian concentration}
There is a well-established and beautiful connection  between the mixing properties of Markov semi-groups and the celebrated concentration-of-measure phenomenon, which can be informally summarized as follows: \emph{mixing is fast if and only if functions that do not vary much under a single transition are approximately constant under the stationary law}. To formalize this vague principle, let us consider a function $f\in\cF$, normalized so that $\Lip(f)\le 1$. Then clearly, $\cE(f,f)\le 1/2$ and therefore, the Poincaré inequality (\ref{def:poincaré}) implies
\begin{eqnarray*}
\Var(f) & \le & \frac{1}{2\gamma}.
\end{eqnarray*}
Combining this with Cantelli's one-sided improvement of Chebychev's inequality, we obtain 
\begin{eqnarray*}
\forall r\ge 0,\qquad \pi\left(f\ge \bE[f]+r\right) & \le & \frac{1}{1+2\gamma r^2}.
\end{eqnarray*}

This naive estimate already implies, for example, that any sequence of expanders exhibits the \red[thin shell phenomenon]: most vertices lie in an annulus whose inner radius grows logarithmically with the number of vertices, but whose thickness remains constant ! As we will now see, the modified log-Sobolev inequality leads to a considerably stronger deviation estimate, through what is now known as \red[Herbst's argument]. We refer the reader to the classical references \cite{MR1767995,MR3185193,MR2283379} for a general account, and to \cite{MR4203344} for the precise version used here. We note that this result is remarkably sharp: the example of random walk on the $n-$cube with $f(x)=\frac{1}{2}(x_1+\cdots+x_n)$ shows that the absolute  constant $2$ can not be improved. 
\begin{theorem}[Sub-Gaussian concentration]\label{th:herbst}If $T$ is reversible, then 
\begin{eqnarray}
\label{herbst:0}
\log \bE[e^{\theta f}] & \le & \theta \bE[f]+\frac{\theta^2}{2\alpha},
\end{eqnarray}
for any $f\in\cF$ with $\Lip(f)\le 1$ and any $\theta\in\dR$. In particular, 
\begin{eqnarray}
\forall r\ge 0,\qquad\pi\left(f \ge \bE[f]+r\right) & \le & e^{-\frac{\alpha r^2}{2}}.
\end{eqnarray}
\end{theorem}
\begin{proof}We focus on the first statement, since the second is a  classical consequence of Chernov's bound. Also, upon replacing $f$ with $-f$ if needed, we may assume that $\theta> 0$. Set
\begin{eqnarray*}
\Lambda(\theta) \ := \ \frac{1}{\theta}\log\bE\left[F_\theta \right],  & \textrm{ where } & F_\theta \ := \ \exp\left(\theta f-\frac{\theta^2}{2\alpha}\right),
\end{eqnarray*}
and note that $\Lambda(\theta)\to \bE[f]$ as $\theta\to 0$. Thus,  the claim (\ref{herbst:0}) simply asserts that $\Lambda(\theta)\le\Lambda(0)$, and the most natural way to prove it is to show that the derivative
\begin{eqnarray*}
\Lambda'(\theta) & = & \frac{\Ent\left[F_\theta \right]-\frac{\theta^2}{2\alpha}\bE[F_\theta]}{\theta^2\bE\left[F_\theta\right]},
\end{eqnarray*}
of $\Lambda$ is non-positive at every point $\theta> 0$. To do so, we will use the basic inequality
\begin{eqnarray}
\label{herbst:1}
\forall u\ge 0,\qquad 2(e^u - 1) & \le & u(e^u + 1),
\end{eqnarray}
which can be readily checked by a term-wise comparison of the corresponding  power series. Now, if  $x,y\in\dX$ satisfy $f(x)\ge f(y)$, then we can use (\ref{herbst:1}) with $u=\theta (f(x)-f(y))$ and multiply through by $(f(x)-f(y))F_\theta(y)$  to get 
\begin{eqnarray*}
2\left(F_\theta (x)-F_\theta (y)\right)\left(f(x)-f(y)\right)  & \le & \theta \left(f(x)-f(y)\right)^2\left(F_\theta (x)+F_\theta (y)\right). 
\end{eqnarray*}
By symmetry, the conclusion also holds when  $f(x)\le f(y)$. Thus,
\begin{eqnarray*}
\cE\left(F_\theta ,\log F_\theta \right) 
& = &  \frac{\theta }{2} \sum_{(x,y)\in \dX^2}\pi(x)T(x,y)\left(F_\theta (x)-F_\theta (y)\right)\left(f(x)-f(y)\right)\\
& \le & \frac{\theta^2}{4} \sum_{(x,y)\in \dX^2}\pi(x)T(x,y)\left(F_\theta (x)+F_\theta (y)\right)\left(f(x)-f(y)\right)^2\\
& = & \frac{\theta^2}{2} \sum_{(x,y)\in \dX^2}\pi(x)T(x,y)F_\theta (x) \left(f(x)-f(y)\right)^2\\
& \le & \frac{\theta^2  \bE[F_\theta]}{2},
\end{eqnarray*}
since $\Lip(f)\le 1$. Applying the MLSI (\ref{MLSI}) to the function $F_\theta$ completes the proof.
\end{proof}
\begin{remark}[Terminology] The inequality (\ref{herbst:0})  is often referred to as \red[sub-Gaussian concentration], because it would be an equality if $f$ were Gaussian with variance $1/\alpha$.
\end{remark}
\begin{remark}[Refinement] An inspection of the proof shows that the weaker condition
\begin{eqnarray*}
\forall x\in\dX,\qquad \sum_{y\in\dX}T(x,y)\left(f(x)-f(y)\right)^2 & \le & 1,
 \end{eqnarray*} 
is all we actually need from the function $f$, instead of $\Lip(f)\le 1$.
\end{remark}
As an application, let us use $\alpha$ to bound the diameter of the chain, as defined at (\ref{def:diam}).  
\begin{coro}[Diameter estimate]\label{co:diam}If $T$ is reversible, then
\begin{eqnarray}
\diam(\dX) & \le & \sqrt{\frac{8}{\alpha}\log \frac{1}{\pi_\star}}.
\end{eqnarray}\end{coro}
\begin{proof}
For any fixed $f\in\cF_0$ with  $\Lip(f)=1$, and any $r \ge 0$, Theorem \ref{th:herbst} yields
\begin{eqnarray*}
\pi(f\ge r) & \le & e^{-\frac{\alpha r^2}{2}}.
\end{eqnarray*}
Choosing $r=\max f$ and noting that the left-hand side is  at least $\pi_\star$, we obtain
\begin{eqnarray*}
(\max f)^2 & \le & \frac{2}{\alpha}\log\frac{1}{\pi_\star}.
\end{eqnarray*}
Replacing $f$ with $-f$ gives the same estimate for $\min f$, and combining the two leads to
\begin{eqnarray*}
\max f-\min f & \le &  \sqrt{\frac{8}{\alpha}\log\frac{1}{\pi_\star}}.
\end{eqnarray*}
At this point, our earlier requirement that $\bE[f]=0$ can simply be dropped, because the conclusion is invariant under shifting $f$ by a constant. In particular, we may take $f(x)=\dist(o,x)$ for any $o\in\dX$, and the claim follows by taking a maximum over all $o\in\dX$.
\end{proof}
\begin{example}[The $n-$cube]For our favorite toy model (random walk on the $n-$cube), we have $\diam(\dX)=n$ and $\alpha=4/n$, so that Corollary \ref{co:diam} is  off by a factor $\sqrt{2\log 2}\approx 1.18$ only. 
\end{example}
\subsection{Hyper-contractivity} In the smooth setting of Markov diffusions, the Dirichlet form takes the pleasant form 
\begin{eqnarray*}
\cE(f,g) & = & \int_\dX \nabla f\nabla g\dd \pi,
\end{eqnarray*}
for all smooth functions $f,g$ on $\dX$. In particular, the classical chain rule for differentiating composed functions allows us to rewrite the asymmetric term $\cE(f,\log f)$  as follows:
\begin{eqnarray}
\label{MLSI=LSI}
\cE(f,\log f) & = & 4\cE(\sqrt{f},\sqrt{f}).
\end{eqnarray}
The  functional inequality obtained by replacing $\cE(f,\log f)$ with $\cE(\sqrt{f},\sqrt{f})$ in (\ref{MLSI}) is famously known as a \red[log-Sobolev inequality] (LSI); see \cite{MR420249,MR1410112}. Of course, the latter makes perfect sense on discrete state spaces as well, but its equivalence with the MLSI is lost because the chain rule used at (\ref{MLSI=LSI}) no longer applies. It is thus natural to wonder whether a discrete LSI has any interesting dynamical interpretation, and the answer turns out to be yes. To describe it, we first recall that our Markov semi-group is a contraction on $L^2$ (and in fact, on $L^p$ for any $1\le p\le\infty$): for every $f\in\cF$ and $t\ge 0$,  
\begin{eqnarray*}
\|P_tf\|_2^2 & = & \sum_{x\in\dX}\pi(x)\EE_x^2\left[f(X_t)\right]\ 
 \le \ \sum_{x\in\dX}\pi(x)\EE_x[f^2(X_t)] \ = \ \|f\|_2^2,
\end{eqnarray*}
with equality when $f$ is constant. \red[Hypercontractivity] is the considerably stronger statement obtained by replacing the $L^2$ norm on the left-hand side by an $L^q$ norm, with an exponent $q=q_t\ge 2$ which is independent of  $f$ and which grows exponentially with $t$. 
While non-reversible versions of the following result exist \cite{MR1410112}, they are slightly less pretty.
\begin{theorem}[Hypercontractivity]For reversible chains, the following  are equivalent:
\begin{enumerate}[(i)]
\item The \red[log-Sobolev inequality] holds with constant $\beta$:
\begin{eqnarray}
\label{LSI}
\forall f\in\cF^{++},\quad \cE(\sqrt{f},\sqrt{f}) & \ge & \beta\Ent(f).
\end{eqnarray}
\item  The semi-group is \red[hypercontractive] at rate $4\beta$, i.e. 
\begin{eqnarray}
\forall f\in\cF,\quad\forall t\ge 0,\quad \|P_tf\|_{1+e^{4\beta t}} & \le & \|f\|_2.
 \end{eqnarray} 
 \end{enumerate}
\end{theorem}
\begin{proof}It is easy to see that both conditions can be restricted to positive functions without loss of generality. Now, fix $f\in\cF^{++}$ and for $t\ge 0$, introduce the shorthands $f_t:=P_tf$ and $q_t := 1+e^{4\beta t}$. Then, an easy (though tedious) computation gives
\begin{eqnarray}
\label{hypercontractivity:diff}
\frac{\dd \|f_t\|_{q_t}}{\dd t} & = & \|f_t\|_{q_t}^{1-q_t}\left[\frac{q_t'}{q_t^2}\Ent\left(f_t^{q_t}\right)-\cE\left(f_t,f_t^{q_t-1}\right)\right].
\end{eqnarray}
In particular, if (ii) holds, then the left-hand side is non-positive at time $t=0$, hence so is the right-hand side. This reads  $\cE(f,f) \ge \beta \Ent(f^2)$, establishing (i). We now turn to the converse. For $q\ge 2$ and $u,v>0$, it is easy to check that
\begin{eqnarray*}
(u-v)\left(u^{q-1}-v^{q-1}\right) & \ge & \frac{4(q-1)}{q^2}\left(u^{q/2}-v^{q/2}\right)^2.
\end{eqnarray*}
If we apply this inequality to $q=q_t$, $u=f_t(x)$, $v=f_t(y)$, and multiply both sides by $\frac{1}{2}\pi(x)T(x,y)$ before summing over all $(x,y)\in\dX^2$, we obtain
\begin{eqnarray*}
\cE\left(f_t,f_t^{q_t-1}\right) & \ge &\frac{4(q_t-1)}{q_t^2}\cE(f_t^{q_t/2},f_t^{q_t/2})\\
& \ge & \frac{4\beta(q_t-1)}{q_t^2}\Ent(f_t^{q_t})\\
& = & \frac{q_t'}{q_t^2}\Ent(f_t^{q_t})
\end{eqnarray*}
thanks to (\ref{LSI}) and our specific choice for $q_t$.  Thus, the right-hand side of  (\ref{hypercontractivity:diff}) is non-positive, hence $t\mapsto \|f_t\|_{q_t}$ is non-increasing. In particular, $\|f_t\|_{q_t}\le \|f_0\|_{q_0}$, which is (ii).
\end{proof}
\begin{definition}[log-Sobolev constant]The \red[log-Sobolev constant] of the chain, henceforth denoted by $\beta$, is the largest constant for which the equivalent conditions (i)-(ii) hold.
\end{definition}
Note that for diffusions,  the identity   (\ref{MLSI=LSI}) implies $\alpha=4\beta$. For reversible chains on discrete spaces, ``half'' of this relation survives. Indeed, it will be shown below that
\begin{eqnarray}
\label{half:LSIvsMLSI}
\forall f\in\cF^{++},\quad 4\cE(\sqrt{f},\sqrt{f}) & \le & \cE(f,\log f),
\end{eqnarray}
which readily implies that $4\beta\le \alpha$. It is important to realize, however, that the inequality (\ref{half:LSIvsMLSI}) can \emph{not} be reversed at a uniform cost, because  $\cE(f,\log f)$ can be made arbitrary larger than $\cE(\sqrt{f},\sqrt{f})$ by just increasing the value of $f$ at a single point. Summarizing the relations established so far between our functional-analytic constants, we have
\begin{eqnarray}
\label{hierarchy}
4\beta & \le \ \alpha\ \le & 2\gamma,
\end{eqnarray} 
and this chain of inequalities reflects a genuine hierarchy between three different regularizing properties of Markov semi-groups: hyper-contractivity is stronger than entropy contraction, which is in turn stronger than variance contraction. Nevertheless, those inequalities happen to be equalities (or nearly so) in various examples, including  our two toy models.

\begin{example}[Random walk on the $n-$cube]An ingenious induction on $n$ shows that  simple random walk on the $n-$cube satisfies $\beta\ge 1/n$. Discovered independently by A. Bonami and W. Beckner \cite{MR283496,MR385456}, this celebrated hypercontractivity estimate has become a cornerstone result in the modern analysis of Boolean functions, with spectacular implications in combinatorics, probability, and theoretical computer science \cite{MR3443800,MR4732718}. Recalling that $\gamma\le 2/n$ in this case, we deduce that (\ref{hierarchy}) is in fact saturated here, i.e.
\begin{eqnarray*}
4\beta \ = \ \alpha\ =\ 2\gamma \ =\ 4/n.
 \end{eqnarray*}  
\end{example}
\begin{example}[Random walk on the $n-$cycle]For random walk on the $n-$cycle with $n$ even, it can be shown that  (\ref{hierarchy}) is again saturated (see \cite{MR1990534} for details):
  \begin{eqnarray*}
4\beta \ = \ \alpha\ =\ 2\gamma \ =\ 2-2\cos\frac{2\pi}{n} \ \sim \ \frac{4\pi^2}{n^2}.
 \end{eqnarray*}
When $n$ is odd, however, the precise value of $\beta$ is embarrassingly still unknown. Nevertheless, the above equalities remain valid up to constant multiplicative errors.
\end{example}
We end this section with a simple  relation between the log-Sobolev constant and the diameter of the chain, which we will use later. Here again, the case of simple random walk on the $n-$cube shows that it  is sharp, except perhaps for the $\sqrt{2}$ pre-factor.  
\begin{lemma}[Diameter bound]\label{lm:diamLSI}
For any reversible chain,  $\diam(\dX)  \leq \sqrt{2}\beta^{-1}$. 
\end{lemma}
\begin{proof}
For any $x\in\dX$, we have 
$
\Ent(f^x) =  \log\frac{1}{\pi(x)}$ and  $\cE(\sqrt{f^x},\sqrt{f^x}) = 1$, so $\log\frac{1}{\pi(x)}\le \beta^{-1}$. Optimizing this bound over $x\in\dX$, we obtain
\begin{eqnarray}
\label{LSvsDimension}
\log \frac{1}{\pi_\star} & \le & \frac 1\beta.
\end{eqnarray} 
On the other hand, in view of Lemma \ref{lm:LSIvsMLSI}, we also have
\begin{eqnarray*}
\frac{4}{\alpha} & \le & \frac{1}{\beta}.
\end{eqnarray*}
Multiplying those two lines together concludes the proof, thanks to 
 Corollary \ref{co:diam}. 
\end{proof}

\subsection{Approximate chain rule}
\label{sec:MLSIvsLSI}
Here again, we restrict ourselves to reversible chains for simplicity.
While hyper-contractivity is genuinely stronger than entropy decay on discrete spaces, we will see that the inequality $4\beta\le \alpha$ can be reversed by paying a price of order $\log d$ only, where
\begin{eqnarray}
\label{def:d}
d & := & \max_{x\sim y}\left\{\frac{1}{T(x,y)}\right\}. 
\end{eqnarray}
Note that this simple parameter controls the \red[sparsity] of the transition matrix $T$: indeed, no row or column can have more than $d$ non-zero entries, and $d$ is exactly the maximum degree of the graph in the special case of simple random walks.
\begin{theorem}[Upgrading MLSI to LSI]\label{th:LSIvsMLSI}For any reversible chain, $4\beta\le \alpha \le  15\beta \log d$.
\end{theorem}
There is in fact an emerging principle hidden behind this result:  several classical estimates that hold in the smooth setting of diffusions due to the chain rule were recently extended to discrete state spaces, by paying the very same cost $\log d$. The key underlying observation is that an approximate version of the chain rule always holds, with a price that depends on the \red[roughness] $r:=\Lip\log f$ of the considered observable. This is formalized in the following result, of which (\ref{MLSI=LSI}) can be seen as the infinitesimally smooth limit $r=0$. 
\begin{lemma}[Approximate chain rule]\label{lm:LSIvsMLSI}Fix  $f\in\cF^{++}$ and set $r:=\Lip\log f$. Then, 
\begin{eqnarray*}
4\cE(\sqrt{f},\sqrt{f}) \ \le \ \cE(f,\log f) \ \le \ \phi(r)\,\cE(\sqrt{f},\sqrt{f}),
\end{eqnarray*}
where we have introduced the  cost function $\phi(r)=r\frac{e^{\frac r2}+1}{e^{\frac r2}-1}$. 
\end{lemma}
\begin{proof}
Fix a function $f\in\cF^{++}$ and two states $x,y\in\dX$, and simply observe that
\begin{eqnarray}
\label{corrector}
\left(f(x)-f(y)\right)\log \frac{f(x)}{f(y)} & = & \left(\sqrt{f(x)}-\sqrt{f(y)}\right)^2 \phi\left(\log \frac{f(x)}{f(y)}\right),
\end{eqnarray}
thanks to our definition of $\phi$. Now, since $\phi\colon\dR\to(0,\infty)$ is even, and increasing on $(0,\infty)$, its value in the above identity is between $\phi(0+)=4$ and $\phi(r)$. Multiplying through by $\frac{1}{2}\pi(x)T(x,y)$ and summing over all $(x,y)\in\dX^2$ finishes the proof. 
\end{proof}
The first inequality in Lemma \ref{lm:LSIvsMLSI} can be directly inserted into the definitions (\ref{MLSI}) and (\ref{LSI}) to yield $4\beta\le\alpha$. However, the second inequality seems unexploitable, because the comparison cost $\phi(r)$ is unbounded as $f$ varies in  $\cF^{++}$. The key observation is that we do not actually need to verify the LSI (\ref{LSI}) for \emph{all}  $f\in\cF^{++}$, but only for \emph{one} very particular choice: an extremizer of the LSI. Fortunately, the latter happens to solve a pointwise equation, just like extremizers in the Poincaré inequality solve an eigenvalue problem. Write
\begin{eqnarray}
\label{def:beta}
\beta & = & \inf_{\cF^+_1\setminus\{1\}}\Phi,\quad \textrm{ where }\quad\Phi(f) \ := \ \frac{\cE(\sqrt{f},\sqrt{f})}{\Ent(f)}.
\end{eqnarray}
\begin{lemma}[Local characterization of log-Sobolev extremizers]\label{lm:LSI}Either  the infimum in (\ref{def:beta}) is achieved by some $f\in\cF^{++}_1\setminus\{1\}$, and the latter must then solve
\begin{eqnarray}
\label{local}
 L \sqrt{f} + \beta \sqrt{f}\log f & = & 0,
\end{eqnarray}
or the infimum in (\ref{def:beta}) is not achieved at all, and   $4\beta=\alpha=2\gamma$.
\end{lemma}
\begin{proof}
The definition of $\Phi(f)$ actually makes sense for any non-constant non-negative function $f$, and it is invariant under scaling, so we may work on this enlarged space. Now, $\Phi$ admits a smooth gradient at every  non-constant positive function $f$, given by 
\begin{eqnarray}
\label{nablaphi}
\nabla\Phi(f) & = & -\frac{1}{2\sqrt{f}\Ent(f)}\left(L\sqrt{f}+\Phi(f)\sqrt{f}\log \frac{f}{\bE[f]}\right).
\end{eqnarray}
Thus, the first part of the claim is just a particular instance of Fermat's theorem, which asserts that the differential of a smooth function must vanish at any interior minimizer. Next, observe that $\Phi$ can not be minimized at a non-interior point. Indeed, such a point is, by definition, a function $f\in\cF^+$ that vanishes somewhere but not everywhere, so we can find an edge  $(x,y)\in E$ with $f(x)=0$ and $f(y)>0$; but then  (\ref{nablaphi}) shows that $\nabla\Phi(f)(x)<0$, hence  $\Phi(f+\varepsilon{\bf 1}_x)<\Phi(f)$ when $\varepsilon>0$ is small enough. Finally, suppose that $\inf\Phi$ is not achieved at all, and consider a sequence $(f_n)_{n\ge 1}$ in $\cF^+_1\setminus\{1\}$ such that
\begin{eqnarray*}
\Phi(f_n) & \xrightarrow[n\to\infty]{} & \beta.
\end{eqnarray*} 
Upon extracting a subsequence if needed, we may assume that $(f_n)_{n\ge 1}$ admits a point-wise limit, which must be in $\cF_1^+$ because $\dX$ is finite. If it were not identically $1$, we would have found a minimizer of $\Phi$, contradicting our assumption. Thus, $f_n\to 1$ as $n\to\infty$, and this easily implies that $2\Ent(f_n) \sim \Var(f_n) \sim  4\Var(\sqrt{f_n})$ as $n\to\infty$, so that 
\begin{eqnarray*}
\Phi(f_n) & \sim & \frac{\cE(\sqrt{f_n},\sqrt{f_n})}{2\Var(\sqrt{f_n})}.
\end{eqnarray*}
Since the right-hand side is at least $\gamma/2$, we obtain $\beta\ge\gamma/2$, forcing equalities in (\ref{hierarchy}).
\end{proof}
As anticipated, the local equation (\ref{local}) inherently imposes some spatial regularity. 
\begin{lemma}[Roughness estimate]\label{lm:regularity}If $f\in\cF^{++}_1$ solves (\ref{local}), then 
$
\Lip \log f  \le  14\log d.
$
\end{lemma}
\begin{proof}
Set $r:=\Lip \log f$ and consider an edge $(x,y)\in E$ where $r$ is attained, i.e.
$
r  = \log\frac{f(y)}{f(x)}.
$
Recalling the definition of $d$ at (\ref{def:d}), we have in particular
\begin{eqnarray*}
\sum_{z\in\dX}T(x,z)\sqrt{\frac{f(z)}{f(x)}} & \ge & d^{-1}e^{\frac{r}2}.
\end{eqnarray*}
On the other hand, evaluating (\ref{local}) at  $x$ and dividing through by $\sqrt{f(x)}$, we know that
\begin{eqnarray*}
\sum_{z\in\dX}T(x,z)\sqrt{\frac{f(z)}{f(x)}} & = & 1-\beta\log f(x).
\end{eqnarray*}
But $\bE[f]=1$, so $\max f\geq 1$, and we may thus write
\begin{eqnarray*}
-\log f(x) & \le & \max \log f-\min \log f \\
& \le & r\diam(\dX) \\
& \le & \sqrt{2}r\beta^{-1},
\end{eqnarray*}
thanks to Lemma \ref{lm:diamLSI}. Combining the last three displays and taking logarithms, we obtain
\begin{eqnarray*}
\frac{r}{2} & \le & \log d+\log \left(1+\sqrt{2}r\right)\\
& \le & \log d +\log\left(1+\frac{r}4\right)+\frac{5}{2}\log 2\\
& \le & \frac{r+14\log d}4,
\end{eqnarray*}
where the last line uses $\log (1+u)\le u$ and the fact that $d\ge 2$. The claim follows.
\end{proof}

\begin{proof}[Proof of Theorem \ref{th:LSIvsMLSI}]According to Lemma \ref{lm:LSI}, either the infimum in (\ref{def:beta}) is not attained, in which case $\alpha=4\beta$ and the claim is trivial, or it is attained at some $f\in\cF^{++}_1\setminus \{1\}$, in which case we may use the definition of $\alpha$ and  Lemmas \ref{lm:LSIvsMLSI} to write  
\begin{eqnarray*}
\alpha\Ent(f) & \le & \cE(f,\log f)\\
& \le & \cE(\sqrt{f},\sqrt{f})\phi(14\log d)\\
& = &  \beta\Ent(f)\phi(14\log d)\\
& \le & 15\beta  \Ent(f)\log d. 
\end{eqnarray*}
In the last line,  we have used the fact that $\frac{\phi(r)}{r}=\frac{e^{r/2+1}}{e^{r/2}-1}$ is a decreasing function of $r$ which equals $\frac{129}{127}\le\frac{15}{14}$ when $r=14\log 2$. Since $f$ is not constant, we may safely simplify through by $\Ent(f)$. 
\end{proof}
\begin{remark}[Diameter estimate]\label{rk:diammls}It follows from Theorem \ref{th:LSIvsMLSI} and Lemma \ref{lm:diamLSI} that
\begin{eqnarray*}
\diam(\dX) & \le & \frac{c}{\alpha}\log d,
\end{eqnarray*} where $c$ is a universal constant. However, this can also be deduced directly from Corollary \ref{co:diam} (which only requires weak reversibility) and  the easy observation that  $\pi_\star\ge (2d)^{-\diam(\dX)}$.
\end{remark}
\newpage
\section{Curvature of Markov chains}
\label{sec:curvature}
In Riemannian geometry, a lower bound  on the Ricci curvature classically  implies an array of powerful estimates on the underlying manifold, including diameter bounds, volume growth, comparison principles, splitting theorems, spectral estimates, and concentration inequalities; see, e.g., the classical textbook \cite{MR3726907}. Over the past decade, those remarkable implications have motivated the development of non-smooth analogues of curvature that can be applied to discrete state spaces, such as graphs and Markov chains. In this chapter, we present two very successful such theories: the Ollivier-Ricci curvature, and the Bakry-\'Emery curvature. 
\subsection{Ollivier-Ricci curvature}
 In an inspiring series of papers \cite{MR2371483,MR2484937,MR2648269}, Y. Ollivier   developed a Wasserstein-based definition  of curvature that makes sense on arbitrary metric spaces. For simplicity, we will here assume that the chain is weakly reversible and equip $\dX$  with the hop-count distance (\ref{def:hopcount}). Focusing on this concrete and intuitive metric will help the unfamiliar reader without too much loss in terms of generality. We emphasize, however, that the theory presented here in fact applies to \emph{any} distance on $\dX$, and that this degree of freedom can actually be   useful, especially in Theorem \ref{th:PT} below. Once we have agreed on a way to measure distances between \emph{points}, we can measure distances between \emph{probability measures}  using optimal transport: specifically, we define the \red[Wasserstein distance] between  $\mu,\nu\in\cP(\dX)$ as
\begin{eqnarray}
\label{def:W}
W(\mu,\nu) & := & \inf_{X\sim\mu,Y\sim\nu}\EE[\dist(X,Y)],
\end{eqnarray}
where the infimum runs over all random pairs $(X,Y)$ with respective marginals $\mu$ and $\nu$. Let us make a number of  important comments about this definition.
First, by compactness, the infimum is always achieved, and  minimizers will henceforth be called \red[optimal couplings]. 
Second, it is not hard to show that $W(\cdot,\cdot)$ inherits the symmetry, separation and triangle inequality from $\dist(\cdot,\cdot)$, and is therefore a distance on $\cP(\dX)$.
Third, the Wasserstein distance is known to admit the \red[Kantorovich dual formulation]
\begin{eqnarray}
\label{kantorovich}
W(\mu,\nu) &  =  & \sup\left\{\left|\mu f-\nu f\right|\colon f\in\cF,
\Lip(f)\le 1\right\},
\end{eqnarray}
as a special case of the celebrated duality theory in convex optimization  \cite{MR1451876}. Fourth, $W(\cdot,\cdot)$ can be seen as an extension of $\dist(\cdot,\cdot)$ to distributions, since
\begin{eqnarray}
\forall x,y\in\dX,\qquad W(\delta_x,\delta_y) & = & \dist(x,y).
\end{eqnarray}
Finally, note that replacing the hop-count distance by the  \red[trivial metric] $(x,y)\mapsto  {\bf 1}_{x\ne y}$ in (\ref{def:W}) produces exactly the total-variation distance, which is totally  \emph{blind} to the underlying geometry: we have $\dtv(\delta_x,\delta_y)=1$ for all $x\ne y$, regardless of the number of transitions needed to move from $x$ to $y$. In particular, the total variation distance can \emph{not} exhibit uniform exponential decay along the semi-group, except in the unrealistic case where $E=\dX^2$. To overcome this limitation, it is natural to consider a metric that really \emph{feels} the intrinsic geometry of the chain, and this is exactly the idea behind the Ollivier-Ricci curvature.  Recall that a transition matrix $T$ is called \red[lazy] if its diagonal entries are at least $1/2$, and that this condition can be enforced at essentially no cost by  replacing $T$ with its \red[lazy version] 
\begin{eqnarray}
\label{def:lazy}
\widehat{T} & := & \frac{T+\mathrm{Id}}{2},
\end{eqnarray}
which simply amounts to running the semi-group at half speed (i.e. $\widehat{P}_t=P_{t/2}$ for all $t\ge 0$). 
\begin{theorem}[Wasserstein contraction]\label{th:OR}For each $\kappa\in\dR$, the following  are equivalent:
\begin{enumerate}[(i)]
\item A lazy transition brings adjacent states closer by an amount at least $\kappa/2$,  on average:
\begin{eqnarray}
\label{OR:local}
\forall \{x,y\}\in E, \qquad W\left(\widehat{T}(x,\cdot),\widehat{T}(y,\cdot)\right) & \le & 1-\frac{\kappa}{2}.
\end{eqnarray}
\item Lipschitz norms decay at rate at least $\kappa$ along the semi-group: 
\begin{eqnarray}
\label{def:ORdual}
\forall f\in\cF,\quad \forall t\ge 0,\quad \Lip(P_tf) & \le & e^{-\kappa t}\Lip(f).
\end{eqnarray}
\item Wasserstein distances decay at rate at least $\kappa$ along the semi-group:
\begin{eqnarray}
\label{def:OR}
\forall \mu,\nu\in\cP(\dX),\quad\forall t\ge 0,\quad W(\mu P_t,\nu P_t)  & \le &  e^{-\kappa t}W(\mu,\nu).
\end{eqnarray}
\end{enumerate}
Moreover, when those equivalent properties hold with  $\kappa>0$, we have the mixing-time estimate 
\begin{eqnarray}
\label{Wmix}
\forall \varepsilon\in(0,1),\qquad \tmix(\varepsilon) & \le & \frac{1}{\kappa}\log\left( \frac{\diam(\dX)}{\varepsilon}\right).
\end{eqnarray}
\end{theorem}
\begin{proof}[Proof of $(i)\Longrightarrow(ii)$]When applied to $\mu=\widehat{T}(x,\cdot)$ and $\nu=\widehat{T}(y,\cdot)$, the duality (\ref{kantorovich}) reads
\begin{eqnarray*}
W\left(\widehat{T}(x,\cdot),\widehat{T}(y,\cdot)\right) & = & \sup_{f\in\cF,\Lip(f)\le 1} \left|\widehat{T}f(x)-\widehat{T}f(y)\right|.
\end{eqnarray*}
Consequently, the condition (i) is equivalent to the Lipshitz contraction property
\begin{eqnarray*}
\forall f\in\cF,\quad \Lip\left(\widehat{T}f\right) & \le & \left(1-\frac{\kappa}{2}\right)\Lip(f).
\end{eqnarray*}
If we iterate this inequality $k$ times, multiply the result through by $e^{-2t}{(2t)}^k/k!$, and sum over $k\ge 0$, we easily arrive at (ii), because $f\mapsto\Lip(f)$ is convex and $P_t=\widehat{P}_{2t}$. 
\end{proof}
\begin{proof}[Proof of $(ii)\Longrightarrow(iii)$]Consider an optimal coupling $(X,Y)$ of two measures $\mu,\nu\in\cP(\dX)$. As is clear from its dual formulation, $W(\cdot,\cdot)$ is convex, so for any $t\ge 0$, we can write 
\begin{eqnarray*}
W\left(\mu P_t,\nu P_t\right) & \le & \sum_{x,y\in\dX}\PP(X=x,Y=y)W\left(P_t(x,\cdot),P_t(y,\cdot)\right)\\
& = & \sum_{x,y\in\dX}\PP(X=x,Y=y)\sup_{f\in\cF,\Lip(f)\le 1}\left|(P_tf)(x)-(P_tf)(y)\right|\\
& \le & e^{-\kappa t}\sum_{x,y\in\dX}\PP(X=x,Y=y)\dist(x,y)\\
& = & e^{-\kappa t}W(\mu,\nu),
\end{eqnarray*}
where we have successively used  (\ref{kantorovich}),  (ii), and the optimality of $(X,Y)$. 
\end{proof}
\begin{proof}[Proof of $(iii)\Longrightarrow(i)$]Fix an edge $\{x,y\}\in E$ and a time $t\in (0,1/2]$. Assumption (iii) ensures the existence of a coupling $(X_t,Y_t)$ of $P_t(x,\cdot)$ and $P_t(y,\cdot)$ such that
\begin{eqnarray}
\label{optcoup}
\EE[\dist(X_t,Y_t)] & \le & e^{-\kappa t}. 
\end{eqnarray}
Now, because transitions occur at rate $1$ under our semi-group, we know that
\begin{eqnarray*}
\PP(X_t=x,Y_t=y) & = & \PP(X_t=x)+\PP(Y_t=y)-\PP(X_t=x\textrm{ or }Y_t=y)\\
& \ge & 2e^{-t}-1\\
& \ge & 1-2t.
\end{eqnarray*}
This means that we can decompose the law of the random pair $(X_t,Y_t)$ as
\begin{eqnarray}
\label{lazydec}
\mathrm{Law}(X_t,Y_t) & = & (1-2t)\delta_{(x,y)}+2t\mathrm{Law}(X_t^\star,Y_t^\star),
\end{eqnarray}
for some random pair $(X^\star_t,Y^\star_t)$. By tightness ($\dX^2$ is finite), the latter admits a sub-sequential limit $(X_0^\star,Y_0^\star)$ as $t\to 0$. To determine its marginals, observe that by definition of the lazy matrix $\widehat{T}$, we have $P_t=\widehat{P}_{2t}=(1-2t)\mathrm{Id}+2t\widehat{T}+o(t)$ as $t\to 0$, so that
\begin{eqnarray*}
\mathrm{Law}(X_t) & = & (1-2t)\delta_{x}+2t\widehat{T}(x,\cdot)+o(t),\\
\mathrm{Law}(Y_t) & = & (1-2t)\delta_{y}+2t\widehat{T}(y,\cdot)+o(t).
\end{eqnarray*}
Comparing this with the decomposition (\ref{lazydec}), we readily deduce that 
\begin{eqnarray*}
\mathrm{Law}(X_0^\star) \ = \ \widehat{T}(x,\cdot), & \textrm{ and } & 
\mathrm{Law}(Y_0^\star) \ = \ \widehat{T}(y,\cdot).
\end{eqnarray*}
In other words, $(X_0^\star,Y_0^\star)$ is a coupling of $\widehat{T}(x,\cdot)$ and $\widehat{T}(y,\cdot)$, and this ensures that 
\begin{eqnarray*}
W\left(\widehat{T}(x,\cdot),\widehat{T}(y,\cdot)\right) & \le & \EE[\dist(X_0^\star,Y_0^\star)].
\end{eqnarray*}
But our assumption (\ref{optcoup}), once decomposed according to (\ref{lazydec}), yields
\begin{eqnarray*}
\EE\left[\dist\left({X}_t^\star,{Y}_t^\star\right)\right]  & \le &  1-\frac{1-e^{-\kappa t}}{2t},
\end{eqnarray*}
and the right-hand side tends to $1-\frac{\kappa}{2}$ as $t\to 0$, establishing the claim.
\end{proof}
\begin{proof}[Proof of (\ref{Wmix})]
The trivial metric is always a lower-bound on the hop-count distance, and this  pointwise comparison is clearly preserved under the Wasserstein lift (\ref{def:W}). Thus, 
\begin{eqnarray*}
\dtv(\mu P_t,\nu P_t) & \le & W(\mu P_t,\nu P_t)\
\ \le \ W(\mu,\nu) e^{-\kappa t},
\end{eqnarray*}
for any $\mu,\nu\in\cP(\dX)$ and  $t\ge 0$. Choosing $\nu=\pi$ and maximizing over all $\mu\in\cP(\dX)$ yields
\begin{eqnarray}
\label{Wexpdecay}
\dtv(t) & \le & \diam(\dX)e^{-\kappa t},
\end{eqnarray}
from which the claim follows by taking $t=\tmix(\varepsilon)$.  
\end{proof}
\begin{definition}[Ollivier-Ricci curvature]The  \red[Ollivier-Ricci curvature],  henceforth denoted by $\kappa$, is the largest constant for which the equivalent conditions (i)-(ii)-(iii) hold
\end{definition} It is perhaps worth emphasizing that this constant quantifies the regularizing nature of the semi-group, just like the Poincaré and modified log-Sobolev constants. Indeed, the Lipschitz norm $\Lip(f)$ is a particular way of measuring how far a function $f$ is from being flat, just like the variance $\Var(f)$ and the entropy $\Ent(f)$, and substituting those for it in (\ref{def:ORdual}) leads to the  Poincaré and modified log-Sobolev constants, respectively. There is, however, an important conceptual difference: the Lipschitz norm  measures the variations of $f$ in a  purely \emph{geometric} way (based on the hop-count distances), whereas the variance and the entropy were purely \emph{statistical} (measuring fluctuations under the equilibrium law). Note also that $\Lip(f)$ is not a smooth function of $f$ in the sense of (\ref{assume:smooth}), so that Theorem \ref{th:OR} can not be obtained by a direct application of our meta theorem (Theorem \ref{th:analytic}).
\begin{remark}[Lichnerowicz estimate]\label{rk:Lichnerowicz}The exponential decay (\ref{Wexpdecay}) shows that 
\begin{eqnarray}
\label{Lichnerowicz}
\lambda & \ge & \kappa.
\end{eqnarray}
This can be seen as a discrete analogue of a celebrated result in differential geometry, due to Lichnerowicz, which asserts that the Ricci curvature of a manifold is a lower-bound on the spectral gap of the Laplace-Beltrami operator. 
\end{remark}
\begin{remark}[Diameter estimate]\label{rk:ORdiam}As many other mixing parameters, the Ollivier-Ricci curvature provides a simple bound on the diameter, namely
\begin{eqnarray}
\label{ORdiam}
\kappa\, \diam(\dX) & \le & 2.
\end{eqnarray}
Indeed, for any $x,y\in\dX$, we can consider an coupling $(X,Y)$ of $\widehat{T}(x,\cdot)$ and $\widehat{T}(y,\cdot)$  to write
\begin{eqnarray*}
\dist(x,y) & \le & \EE\left[\dist(x,X)\right]+\EE[\dist(X,Y)]+\EE[\dist(Y,y)]\\
& \le & \frac 12+W\left(\widehat{T}(x,\cdot),\widehat{T}(y,\cdot)\right)+\frac 12\\
& \le & 1+\left(1-\frac{\kappa}{2}\right)\dist(x,y), 
\end{eqnarray*}
thanks to (\ref{OR:local}). Choosing 
 $x,y$ at maximal distance and simplifying gives exactly (\ref{ORdiam}). 
\end{remark}
\begin{example}[The $n-$cube]\label{exOR:cube}On the $n-$cube, a lazy transition  consists in choosing a  coordinate uniformly at random in $\{1,\ldots,n\}$ and resampling its value uniformly at random from the set $\{-1,+1\}$. If we couple the  transitions from two adjacent states $x$ and $y$ by using the same randomness for both updates, we generate a random pair $(X,Y)$ that satisfies
\begin{eqnarray*}
\dist(X,Y) & = & \left\{
\begin{array}{ll}
1 & \textrm{with probability }1-1/n\\
0 & \textrm{with probability }1/n.
\end{array}
\right.
\end{eqnarray*}
By the local criterion (\ref{OR:local}), we deduce that $\kappa\ge 2/n$, so that  (\ref{Lichnerowicz})-(\ref{ORdiam}) are saturated: 
\begin{eqnarray}
\kappa & = & \gamma \ = \  \frac{2}{\diam(\dX)} \ = \ \frac{2}{n}.
\end{eqnarray}
Note  that the mixing-time estimate (\ref{Wmix})  reads
$
\tmix(\varepsilon)  \le  \frac{n\log n}{2}+ O_\varepsilon(n),
$
as $n\to\infty$, which is only off by a factor $2$ when compared with  (\ref{ex:cube}).
\end{example}
\subsection{The Peres-Tetali conjecture}

At a high level, the Ollivier-Ricci curvature shares many similarities with the modified log-Sobolev constant: both quantify how functions \emph{flatten} along the semi-group, both constitute lower bounds on the spectral gap, both provide sharp estimates on mixing times, and both are known to imply sub-Gaussian concentration under the stationary measure \cite{MR3726607}. This observation, along with explicit computations on several examples, has led the community to hope for a systematic relation of the following form, possibly up to universal pre-factors.
\begin{question}[The Peres-Tetali conjecture]Do we always have $\alpha \ge \kappa$?
\end{question}
Note that this would constitute a substantial strengthening   of the Lichnerowicz estimate (\ref{Lichnerowicz}). Beyond the obvious conceptual interest of connecting a geometric notion to an information-theoretic one, such a statement would have considerable practical value: indeed, the  local, coupling-based characterization of curvature makes it very easy to evaluate on concrete models, and the inequality $\alpha\ge \kappa$ would thus provide an effective way to estimate the much more delicate modified log-Sobolev constant. Unfortunately, counter-examples to the Peres-Tetali conjecture have been constructed \cite{munch2023ollivier}. Nevertheless, we will show that the conjecture can be repaired at a reasonable price, as discovered in \cite{munch2023ollivier} and refined in \cite{MR4891383}. To this aim, we generalize the Wasserstein distance by setting
\begin{eqnarray}
W_p(\mu,\nu) & := & \inf_{X\sim\mu,Y\sim\nu}\EE\left[\dist^p(X,Y)\right]^{\frac 1p},
\end{eqnarray}
for any $p\in[1,+\infty]$, with the usual understanding that $\EE[Z^\infty]^{1/\infty}$ is the essential supremum of $Z$. We then naturally define the \red[Ollivier-Ricci $p-$curvature]  as the largest $\kappa_p$ such that
\begin{eqnarray*}
\forall \mu,\nu\in\cP(\dX),\quad \forall t\ge 0,\quad W_p(\mu P_t,\nu P_t) & \le & e^{-\kappa_p t}W_p(\mu,\nu).
\end{eqnarray*} 
As before, it is in fact enough to check this when $\mu$ and $\nu$ are Dirac masses at adjacent states:
\begin{eqnarray*}
\forall \{x,y\}\in E,\quad \forall t\ge 0,\quad W_p(P_t(x,\cdot),P_t(y,\cdot)) & \le & e^{-\kappa_p t}\dist(x,y).
\end{eqnarray*} 
From this formulation, it is clear that $\kappa_p$ is non-increasing in $p$, so the weakest notion is  the Ollivier-Ricci curvature $\kappa_1$ defined earlier, whereas the strongest is the so-called \red[sectional curvature] $\kappa_\infty$.  We naturally write $\kappa_p^\star$ for the $p-$curvature of the adjoint chain $T^\star$. With this notation at hand, we can state our ``repaired'' Peres-Tetali conjecture.
\begin{theorem}[Curvature and entropy]\label{th:PT}We always have
$
\alpha \ge \kappa_1+\kappa_\infty^\star.
$
\end{theorem}
Note that the sectional curvature $\kappa_\infty^\star$ can not be positive in our discrete setting, since in a small enough time interval there is always a non-zero chance that no transition occur, regardless of the considered coupling. However, many interesting Markov chains satisfy $\kappa_\infty^\star=0$, and verifying this is fairly easy on concrete models. Indeed, all we have to do is to find, for each edge $\{x,y\}\in E$, a coupling $(X,Y)$ of $\widehat{T}^\star(x,\cdot)$ and $\widehat{T}^\star(y,\cdot)$ satisfying
\begin{eqnarray}
\label{assume:NNSC}
\PP\left(\dist(X,Y) \le  1\right) & = & 1.
\end{eqnarray}
Examples include random walks on Abelian groups, conjugacy-invariant random walks, high-temperature Glauber dynamics, or zero-range processes with monotone jump rates. For concreteness, let us illustrate the above theorem on  a couple of simple models.
\begin{example}[Random walk on the $n-$cube]The coupling defined in Example \ref{exOR:cube} clearly satisfies (\ref{assume:NNSC}). Thus, Theorem \ref{th:PT} guarantees that $\alpha\ge 2/n$, which is half the true value!
\end{example}
\begin{example}[Rank-one chains]\label{ex:rankone}Given a target measure $\pi\in\cP(\dX)$, consider the rank-one chain $T(x,y)  :=  \pi(y)$ that mixes in a single transition. Since trajectories can be coupled so as to coincide after the first transition, we have $\kappa_1=1$ and $\kappa_\infty=0$. Thus, Theorem \ref{th:PT} guarantees that $\alpha\ge 1$, and since $\gamma=1$ we conclude that $\alpha\in[1,2]$. Interestingly, we have $d=\frac{1}{\pi_\star}$ here, so we can invoke Theorem \ref{th:LSIvsMLSI} and (\ref{LSvsDimension})  to obtain
\begin{eqnarray*}
\beta & \in & \left[\frac{1}{30\log\frac{1}{\pi_\star}}, \frac{1}{\log\frac{1}{\pi_\star}}\right].
\end{eqnarray*}
This shows that the $\log d$ dependency in  Theorem \ref{th:LSIvsMLSI} is actually sharp. 
\end{example}

In order to establish Theorem \ref{th:PT}, we fix $t\ge 0$ and analyze the optimization problem
\begin{eqnarray}
\label{def:rho}
\chi_t & = & \sup_{f\in\cF^+_1\setminus\{1\}}\frac{\Ent(P_t^\star f)}{\Ent(f)}.
\end{eqnarray}
Our starting point is the following characterization of maximizers,  similar to  Lemma \ref{lm:LSI}.
\begin{lemma}[Maximizers]\label{lm:optimizer}Either (\ref{def:rho}) is attained by some $f\in\cF^{++}_1\setminus\{1\}$, and
\begin{eqnarray}
\label{chi:local}
P_t \log {P_t^\star f} & = & \chi_t\log {f},
\end{eqnarray}
or it is not attained at all, and $\chi_t$ is then exactly the second largest eigenvalue of $P_tP_t^\star$.
\begin{proof}
We may assume that $t>0$, otherwise the claim is trivial.  The functional $\Phi\colon f\mapsto \frac{\Ent(P_t^\star f)}{\Ent(f)}$ is continuous on the space of non-constant non-negative functions and invariant under scaling. Moreover, it is continuously differentiable in the interior, with gradient
\begin{eqnarray}
\label{nablaphi2}
\nabla \Phi(f) & = & \frac{1}{\Ent(f)}\left[P_t\log \frac{P_t^\star f}{\bE[f]}-\Phi(f)\log \frac{f}{\bE[f]}\right].
\end{eqnarray}
Thus, the first part of the claim is just a particular instance of Fermat's theorem, which asserts that the differential of a smooth function must vanish at any interior maximizer. Next, consider a non-interior point in the above space, i.e. a function $f\in\cF^+$ that vanishes somewhere, but not everywhere. Since all entries of $P_t$ are positive,  $P_t^\star f$ is then positive everywhere, so  the right-hand side of (\ref{nablaphi2}) extends continuously to $f$, with limits being understood in $\dR\cup\{+\infty\}$. But $f$ vanishes at some $x\in\dX$, so $\nabla \Phi(f)(x)=+\infty$, implying that $\Phi(f+\eta\delta_x)>\Phi(f)$ for all small enough $\eta>0$: thus, $\Phi$ cannot attain its supremum at a non-interior point. Finally, suppose that $\Phi$ does not attain its supremum at all. Then, we can at least find a sequence $(f_n)_{n\ge 1}$ in $\cF_1^{+}\setminus\{1\}$ such that
\begin{eqnarray*}
\Phi(f_n) & \xrightarrow[n\to\infty]{} & \chi_t.
\end{eqnarray*} 
Upon extracting a subsequence if needed, we may assume that $(f_n)_{n\ge 1}$ has a limit in $\cF_1^+$. If the latter were not identically $1$, we would have found a maximizer of $\Phi$, contradicting our assumption. Thus, $f_n\to 1$ as $n\to\infty$ and by linearization (Remark \ref{rk:linearization}), this implies  
\begin{eqnarray*}
\Phi(f_n) & \sim & \frac{\Var(P_t^\star f_n)}{\Var(f_n)}.
\end{eqnarray*}
Clearly, the expression on the right-hand side is maximized when $f_n$ is an eigenvector corresponding to the second largest eigenvalue of $P_tP_t^\star$ (shifted  and normalized to be in $\cF_1^{++}$). 
\end{proof}
\end{lemma}
We will also need the following dual formulation of sectional curvature. 
\begin{lemma}[Dual formulation of $\kappa_\infty$]For any $f\in\cF^{++}$ and  $t\ge 0$, we have
\begin{eqnarray*}
\Lip \log P_tf  & \le & e^{-t\kappa_\infty }\Lip \log f.
\end{eqnarray*}
\end{lemma}
\begin{proof}
 Fix a function $f\in\cF^{++}$, a time $t\ge 0$, and two states $x,y\in\dX$, and let $(X_t,Y_t)$ realize a  $W_\infty-$optimal coupling of $P_t(x,\cdot),P_t(y,\cdot)$. We can then write, almost-surely,
\begin{eqnarray*}
\log f(X_t) & \le & \log f(Y_t)+\Lip(\log f)\dist(X_t,Y_t)\\
& \le & \log f(Y_t)+\Lip(\log f)W_\infty\left(P_t(x,\cdot),P_t(y,\cdot)\right)\\
& \le & \log f(Y_t)+\Lip(\log f)e^{-t\kappa_\infty}\dist(x,y).
\end{eqnarray*}
Taking exponentials, then expectations, and then logarithms again yields
\begin{eqnarray*}
\log \EE[f(X_t)] & \le & \log \EE[f(Y_t)]+\Lip(\log f)e^{-t\kappa_\infty}\dist(x,y).
\end{eqnarray*}
The claim follows, since by construction $\EE[f(X_t)]=P_tf(x)$ and $\EE[f(Y_t)]=P_tf(y)$.
\end{proof}
We now have everything we need to prove our main result.
\begin{proof}[Proof of Theorem \ref{th:PT}]In view of the definitions of $\alpha$ and $\chi_t$, our task is to prove that 
\begin{eqnarray*}
\forall t\ge 0,\quad \chi_t & \le & e^{-(\kappa_1+\kappa_\infty^\star)t}.
\end{eqnarray*}
To do so, we fix $t\ge 0$ and we first assume that the supremum  (\ref{def:rho}) is not attained. Then, Lemma \ref{lm:optimizer} ensures that $\chi_t$ is the second largest eigenvalue of $P_tP_t^\star$. Denoting by $f$ a corresponding eigenvector, and using the functional characterization of $\kappa_1$, we then have 
\begin{eqnarray*}
\chi_t\Lip f & = & \Lip(P_tP_t^\star f) \ \le \ e^{-(\kappa_1+\kappa_1^\star)t}\Lip f,
\end{eqnarray*}  
and simplifying through by $\Lip(f)$ allows us to conclude, since $\kappa_1^\star\ge\kappa_\infty^\star$. On the other hand, if the supremum (\ref{def:rho}) is attained, it must be  at some $f\in\cF^{++}_1\setminus\{1\}$ satisfying
(\ref{chi:local}). Taking Lipschitz norms on both sides, and using our dual formulations of $\kappa_1,\kappa_\infty^\star$, we get
\begin{eqnarray*}
\chi_t\Lip \log f & = & \Lip\left(P_t\log P_t^\star f\right)
\ \le \ e^{-(\kappa_1 +\kappa_\infty^\star)t}\Lip \log f, 
\end{eqnarray*}
which produces the desired conclusion after simplifying through by $\Lip \log f$.
\end{proof}

\subsection{Bakry-\'Emery curvature}
Introduced four decades ago in the context of diffusions on manifolds \cite{MR889476}, the Bakry-\'Emery theory of curvature  has quickly become one of the most powerful tools in the quantitative study of geometric, probabilistic and functional-analytical properties of Markov semi-groups. We refer the reader to the  textbook \cite{MR3155209} for a comprehensive introduction, and to the seminal papers \cite{MR1168070, MR1665591, MR2644381} for the discrete version considered here. If $(X_t)_{t\ge 0}$ is a realization of our continuous-time Markov chain, started from a deterministic state $X_0=x\in\dX$, then it follows from the very definition of the generator that for any $f\in \cF$,
\begin{eqnarray}
\label{expectation}
\EE[f(X_t)] & = & f(x)+t Lf(x)+o(t),
\end{eqnarray}
as $t\to 0$. We claim that we have the following second-order version: for any  $f,g\in\cF$, 
\begin{eqnarray}
\label{covariation}
\Cov\left(f(X_t),g(X_t)\right) & = & 2t\,\Gamma(f,g)(x)+o(t),
\end{eqnarray}
where we have introduced the \red[carré du champ operator]
\begin{eqnarray}
\label{def:Gamma}
\Gamma(f,g) & := & \frac{1}{2}\left(L(fg)-fLg-gLf\right).
\end{eqnarray}
Indeed, by polarization, it is enough to prove (\ref{covariation}) when $f=g$, which is readily done by applying (\ref{expectation}) to $f^2$ and substracting the square of (\ref{expectation}). Recalling that $L=T-{\rm Id}$, we easily arrive at the following concrete expression:
\begin{eqnarray*}
\Gamma(f,g)(x) & = & \frac{1}{2}\sum_{y\in\dX}T(x,y)\left(f(x)-f(y)\right)\left(g(x)-g(y)\right).
\end{eqnarray*}
In particular, integrating this against the stationary law gives rise to the symmetrized Dirichlet form (\ref{rk:sym}), which is of course just $\cE(f,g)$ when $T$ is reversible or when $f=g$. In what follows, it will be convenient to introduce the \red[iterated carré du champ operator], defined by replacing the multiplication $\cdot\times\cdot$ with $\Gamma(\cdot,\cdot)$ in the definition (\ref{def:Gamma}):
\begin{eqnarray}
\Gamma_2(f,g) & := & \frac{1}{2}\left(L\Gamma(f,g)-\Gamma(f,Lg)-\Gamma(g,Lf)\right).
\end{eqnarray}
The interest of this object is contained in the following fundamental result. 
Here and throughout the notes, we will simply write $\Gamma f$ and $\Gamma_2 f$ when $f=g$. 
\begin{theorem}[The Bakry-\'Emery criterion]\label{th:BE}For any $\rho\in\dR$, the following are equivalent:
\begin{enumerate}[(i)]
\item \red[Bakry-\'Emery $\mathrm{CD}(\rho,\infty)$ criterion]:
\begin{eqnarray}
\label{criterion:CD}
\forall f\in\cF,\qquad \Gamma_2 f & \ge & \rho\,\Gamma f.
\end{eqnarray}
\item \red[Exponential sub-commutation]: 
\begin{eqnarray}
\label{subcommute}
\forall f\in\cF,\quad\forall t\ge 0,\quad \Gamma P_t f & \le & e^{-2\rho t}P_t\Gamma f.
\end{eqnarray} 
\item \red[Local Poincaré inequality]: from any non-random initial condition $X_0$, we have
\begin{eqnarray}
\label{def:localPoincaré}
\forall f\in\cF,\quad\forall t\ge 0,\quad
\Var(f(X_t)) & \le & \left(\int_0^t 2e^{-2\rho u}\dd u\right)\EE[\Gamma f(X_t)].
\end{eqnarray}
\end{enumerate}
\end{theorem}
The local Poincaré inequality (\ref{def:localPoincaré}) controls the variance of any observable in terms of its local quadratic variations, exactly like our classical Poincaré inequality (\ref{def:poincaré}). The key difference is that the underlying measure here is the law of the chain at an arbitrary time $t\ge 0$, rather than the equilibrium distribution. The ability to estimate non-equilibrium variances is of course very useful for the practical analysis of Markov processes, and it will play a crucial in our understanding of the cutoff phenomenon in the next chapter. The optimal parameter $\rho$ appearing in (i)-(ii)-(iii) certainly deserves a fancy name. 
\begin{definition}[Bakry-\'Emery curvature]The \red[Bakry-\'Emery curvature] of the chain, henceforth denoted by $\rho$, is the largest number for which the  properties (i)-(ii)-(iii) hold. 
\end{definition}
\begin{proof}[Proof of Theorem \ref{th:BE}]For any $f\in\cF$ and $0\le s\le t$, we easily compute
\begin{eqnarray}
\label{BEdiff}
\frac{\dd P_{t-s} (P_s f)^2}{\dd s} & = & -2P_{t-s}\Gamma P_s f,
\end{eqnarray} 
and integrating yields the following useful integral expression for variances:
\begin{eqnarray}
\label{variance}
P_t (f^2)-(P_t f)^2 & = & 2\int_0^tP_{t-s}\Gamma P_s f\dd s.
\end{eqnarray}
In particular, using (ii) inside the integral readily yields
\begin{eqnarray}
\label{BEiii}
P_t (f^2)-(P_t f)^2 & \le & \left(\int_0^t 2e^{-2\rho s}\dd s\right)P_t \Gamma f,
\end{eqnarray}
which is exactly (iii). 
Also, replacing multiplication with $\Gamma(\cdot,\cdot)$ in (\ref{BEdiff}) gives 
\begin{eqnarray}
\label{BEdiff2}
\frac{\dd P_{t-s} \Gamma P_s f}{\dd s} & = & -2P_{t-s}\Gamma_2 P_s f.
\end{eqnarray}
Thus,  (i) implies that the function $F\colon s\mapsto P_{t-s} \Gamma P_s f$ satisfies the differential inequality $F'\le -2\rho F$ on $[0,t]$, which we can integrate to obtain  exactly (ii). Finally, we have the following (easy, but tedious) Taylor expansion:
\begin{eqnarray*}
P_t (f^2)-(P_t f)^2 & = & 2t \Gamma f+2t^2\left[L\Gamma f-\Gamma_2f\right] +o(t^2);\\
 \left(\int_0^t 2e^{-2\rho s}\dd s\right)P_t \Gamma f & = & 2t\Gamma f+2t^2\left[L\Gamma f-\rho \Gamma f\right]+o(t^2),
\end{eqnarray*}
as $t\to 0$, so that the local Poincaré inequality (iii) forces $\Gamma_2 f\ge \rho \Gamma f$, as desired. 
\end{proof}
\begin{remark}[Lichnerowicz estimate]\label{rk:Lichnerowicz2}
When the Bakry-\'Emery curvature $\rho$ is positive, we may simply send $t\to\infty$ in (\ref{def:localPoincaré}) to recover exactly the Poincaré inequality (\ref{def:poincaré}). Thus,
\begin{eqnarray*}
\gamma & \ge & \rho,
\end{eqnarray*}
which constitutes yet another discrete analogue of the celebrated Lichnerowicz estimate.
\end{remark}
\begin{remark}[Comparison with $\kappa_1$]\label{rk:ORvsBE}As all our  functional-analytic constants $(\gamma,\alpha,\beta,\kappa,\cdots)$, the Bakry-\'Emery curvature $\rho$ seeks to capture the regularizing nature of the semi-group. Indeed, the quantity $\|\Gamma f\|_\infty$ measures how far a function $f$ is from being flat, just like $\Var(f),\Ent(f)$ or $\Lip(f)$, and the sub-commutation property (\ref{subcommute}) readily implies that
\begin{eqnarray}
\label{Gammadecay}
\forall t\ge 0,\forall f\in\cF,\quad \|\Gamma P_tf\|_\infty & \le & e^{-2\rho t} \|\Gamma f\|_\infty. 
\end{eqnarray}
The resemblance with  $\kappa$ is particularly strong, since we have the general comparison 
\begin{eqnarray}
\label{LipvsGamma}
2\|\Gamma f\|_\infty & \le \ \Lip^2(f) \ \le &  2d\|\Gamma f\|_\infty.
\end{eqnarray}
In particular, $\kappa$ also implies a local Poincaré inequality, albeit in the  weaker form
\begin{eqnarray}
\Var(f(X_t)) & \le & \left(\int_0^t e^{-2\kappa s}\dd s\right)\Lip^2(f),
\end{eqnarray}
as can be easily derived by combining (\ref{def:ORdual}), (\ref{variance}) and  (\ref{LipvsGamma}).
\end{remark}
\subsection{A discrete Bakry-\'Emery theorem}

A cornerstone result in the Bakry-\'Emery theory (see, e.g.,  \cite{MR889476,MR1767995,MR3155209}) is the fact that the curvature of a Markov diffusion provides a lower-bound on the log-Sobolev constant:
\begin{eqnarray}
\label{BELSI}
\rho & \leq & 2\beta. 
\end{eqnarray}
This beautiful connection between geometry and functional analysis constitutes a remarkable strengthening of the Lichnerowicz estimate (Remark \ref{rk:Lichnerowicz2}). Unfortunately, its proof crucially relies on the following chain rule, which is specific to second-order differential generators:
\begin{eqnarray}
\label{chain}
\Gamma \log f & = & \frac{Lf}{f}-L\log f.
\end{eqnarray}
 Nevertheless, an approximate version of (\ref{BELSI}) turns out to hold for reversible chains on discrete state spaces, with a cost that is logarithmic in the sparsity parameter (\ref{def:d}). In particular, for models in which $d$ is uniformly bounded (such as  random walks on bounded-degree graphs), the remarkable relation (\ref{BELSI}) remains essentially valid.  
\begin{theorem}[Curvature and LSI]\label{th:BELSI}Any reversible chain satisfies $\rho  \leq 33 \beta\log d$. 
\end{theorem}
\begin{exo}[Sharpness]Consider again the rank-one chain $T(x,y)=\pi(y)$ of Example \ref{ex:rankone}. Check that for any $f\in\cF$ with mean zero and variance $1$, we have 
\begin{eqnarray*}
\Gamma f \ = \ \frac{1+f^2}{2} & \textrm{ and } & \Gamma_2f \ = \ \frac{3+f^2}{4},
\end{eqnarray*}
and deduce that $\rho=\frac{1}{2}+\frac{\pi_\star}{1+\pi_\star}$. Thus, the two sides of the inequality in Theorem \ref{th:BELSI} have the same order of magnitude in this generic class, confirming that $\log d$ is the right price to pay.
\end{exo}

To prove Theorem \ref{th:BELSI}, we follow the strategy developed in \cite{salez2025intrinsicregularitydiscretelogsobolev}, which is similar to the one already used for Theorem \ref{th:LSIvsMLSI}: we will establish a discrete approximate version of the chain rule (\ref{chain}), with an extra cost that depends on the roughness $r:=\Lip\log f$ of the considered observable. This is the content of the following lemma, which does not require reversibility, and of which (\ref{chain}) can be seen as the infinitely-smooth limit $r\to 0$. 
\begin{lemma}[Approximate chain rule]\label{lm:chain}Fix $f\in\cF^{++}$ and set $r:=\Lip \log f$. Then,
\begin{eqnarray*}
\phi(-r)\,\left(\frac{Lf}{f}-L\log f\right)\ \le & \Gamma \log f & \le \ \phi(r)\,\left(\frac{Lf}{f}-L\log f\right),
\end{eqnarray*}
where we have introduced the cost function $\phi(r)  = \frac{r^2}{2\left(r +e^{-r}-1\right)}$. In particular, 
\begin{eqnarray*}
\EE\left[f\Gamma \log f\right] & \le & \left(1+r\right)\cE(f,\log f).
\end{eqnarray*}
\end{lemma}
\begin{proof}
Since the above function $\phi\colon\dR\to[0,\infty)$  is  increasing, we have $\phi(-r)\le \phi(u)\le\phi(r)$ for all $u\in[-r,r]$. Choosing $u=\log\frac{f(x)}{f(y)}$ and $r=\Lip \log f $, we obtain
\begin{equation*}
\phi(-r)\left(\frac{f(y)}{f(x)}-1+\log\frac{f(x)}{f(y)}\right) \ \le \ \frac{1}{2}\log^2\frac{f(x)}{f(y)} \ \le \ \phi(r)\left(\frac{f(y)}{f(x)}-1+\log\frac{f(x)}{f(y)}\right),
\end{equation*}
for all $(x,y)\in E$. Multiplying through by $T(x,y)$ and summing over all $y\in\dX$ yields exactly the first claim. To deduce the second, we multiply through by $f$, take expectations, and observe that $\bE[Lf]=0$ (by stationarity), and that $\phi(r)\le 1+r$ for all $r\ge 0$. 
\end{proof}
Our plan is to  apply this chain rule to a function $f$ that saturates the LSI. Recall that the roughness of such a function is guaranteed to be $O(\log d)$, by Lemma \ref{lm:optimizer}. However, we will also have to  work with the time-evolved functions $P_tf,t\ge 0$, hence the following lemma. 
\begin{lemma}[Roughness does not blow up under the semi-group]\label{lm:timeregularity}
\begin{eqnarray*}
\forall f\in\cF^{++},\quad\forall t\ge 0,\quad \Lip \log P_tf  &  \le & \log d+\Lip \log f.
\end{eqnarray*}
\end{lemma}
\begin{proof}
Fix $f\in\cF^{++}$ and write $r:=\Lip \log f$. For any $x\in\dX$ and $k\in\dN$, we have
\begin{eqnarray*}
T^{k}f(x) &  = & \sum_{z,w\in\dX}T^{k-1}(x,z)T(z,w)f(w)\\
& \le & \sum_{z,w\in\dX}T^{k-1}(x,z)T(z,w)f(z)e^{r}\ = \ e^{r} T^{k-1}f(x),
\end{eqnarray*}
by definition of $r$. On the other hand, for any neighbor $y$ of $x$, we also have
\begin{eqnarray*}
T^{k}f(y) &  = & \sum_{z\in\dX}T(y,z)T^{k-1}f(z)\\
& \ge & T(y,x)T^{k-1}f(x)\\
& \ge & d^{-1}T^{k-1}f(x).
\end{eqnarray*}
Combining those two bounds, we obtain 
$
T^{k}f(x)  \le  de^{r}T^{k}f(y),
$
and this also holds  when $k=0$. Finally, multiplying through by $e^{-t}t^k/k!$ and summing over $k\in\mathbb N$ yields
\begin{eqnarray*}
P_tf(x) & \le & de^{r}P_tf(y),
\end{eqnarray*}
for any adjacent states $x,y\in\dX$. Taking logarithms concludes the proof. 
\end{proof}

\begin{proof}[Proof of Theorem \ref{th:BELSI}]
First note that by reversibility, we have for any $g,h\in\cF$,
\begin{eqnarray*}
   \cE(g^2,h )
   & = &\frac12 \sum_{x,y\in \dX} \pi(x)T(x,y)(g(x)+g(y))(g(x)-g(y))(h(x)-h(y)) \\
   & = & \sum_{x,y\in \dX} \pi(x)T(x,y)g(x)(g(x)-g(y))(h(x)-h(y))\\
   & = & 2\bE \left[g\Gamma(g,h)\right]\\
   & \le & 2\bE \left[g\sqrt{\Gamma g\Gamma h}\right]\\
   & \le & 2\sqrt{\cE(g,g){\bE\left[g^2\Gamma h\right]}},
\end{eqnarray*}
where we have used the Cauchy-Schwarz inequality for $\Gamma(\cdot,\cdot)(x)$ and for the scalar product in $L^2$.  
We now fix $f\in\cF^{++}$ and $t\ge 0$, and apply the above inequality to $g=\sqrt{f}$ and $h=P_t\log P_t f$: using reversibility, we obtain
\begin{eqnarray*}
\cE( P_t f,\log P_tf) & \le & 2\sqrt{\cE(\sqrt{f},\sqrt{f})\bE\left[f\Gamma P_t\log P_t f\right]}.\end{eqnarray*}
On the other hand, 
\begin{eqnarray*}
\bE\left[f\Gamma P_t\log P_t f\right] & \le & e^{-2\rho t}\bE\left[fP_t\Gamma\log P_t f\right]\\
& = & e^{-2\rho t}\bE\left[(P_tf)\Gamma \log P_t f\right]\\
& \le & e^{-2\rho t}\left(1+\Lip \log P_tf\right)\cE(P_tf,\log P_tf)\\
& \le & e^{-2\rho t}\left(1+\log d+\Lip \log f\right)\cE(P_tf,\log P_tf),
\end{eqnarray*}
where we have successively used the sub-commutation relation (\ref{subcommute}), reversibility, and   Lemmas \ref{lm:chain}-\ref{lm:timeregularity}. Combining those two observations, we deduce that
\begin{eqnarray*}
\cE( P_t f,\log P_tf) & \le & 4e^{-2\rho t}\left(1+\log d+\Lip\log f\right)\cE(\sqrt{f},\sqrt{f}).
\end{eqnarray*}
But the left-hand side is exactly $-\frac{\dd\Ent(P_t f)}{\dd t}$, so we may integrate on $[0,\infty)$ to obtain
\begin{eqnarray*}
       \Ent(f) & \le & \frac{2}{\rho}\left(1+\log d+\Lip\log f\right)\cE(\sqrt{f},\sqrt{f}).
\end{eqnarray*}
In particular, if $f$ is a non-constant function realizing the definition  of the log-Sobolev constant $\beta$, then by Lemma \ref{lm:regularity} and the fact that $d\ge 2$, we obtain
\begin{eqnarray*}
\frac{1}{\beta}\cE(\sqrt{f},\sqrt{f}) & \le & \frac{33\log d}{\rho}\cE(\sqrt{f},\sqrt{f}),
\end{eqnarray*}
which readily yields the desired conclusion. On the other hand, if no non-constant function achieves equality in (\ref{LSI}), then Lemma \ref{lm:LSI} guarantees that $\beta=\gamma/2$, which  is more than enough to conclude in view of Remark \ref{rk:Lichnerowicz2}.
\end{proof}

\subsection{Examples of non-negatively curved chains}
In this illustrative section, we investigate the Ollivier-Ricci and Bakry-\'Emery curvatures of some of the important models introduced in the second chapter. Let us start with the smooth setup of Langevin diffusions, for which the answer is particularly simple and elegant. 
\begin{theorem}[Langevin dynamics]The Bakry-\'Emery curvature of the Langevin diffusion  (\ref{SDE})  is exactly the largest number $\rho\in[-\infty,\infty)$ such that
\begin{eqnarray}
\label{assume:convexity}
\mathrm{Hess}(U) & \ge & \rho\textrm{Id},
\end{eqnarray}
and the same holds for the Ollivier-Ricci $p-$curvature, for any $p\in[1,\infty]$.
\end{theorem}
\begin{proof}The carré du champ and its iterated version are easily computed to be
\begin{eqnarray*}
\Gamma(f,g) \ = \ \nabla f \cdot \nabla g, & \textrm{ and } & 
\Gamma_2(f,g) \ = \ \|\mathrm{Hess\,}(f)\|_{HS}^2+\nabla f\cdot \mathrm{Hess\,}(U) \nabla g,
\end{eqnarray*}
Thus, the Bakry-\'Emery curvature is exactly the optimal constant in (\ref{assume:convexity}). On the other hand, an obvious way to couple two solutions $(X_t)_{t\ge 0}$  and $(Y_t)_{t\ge 0}$ of the Langevin equation (\ref{SDE}) starting from different points $x,y\in\dR^d$ consists in using the same Brownian motion to drive both trajectories. Under this \red[parallel coupling], we readily find
\begin{eqnarray*}
\frac{\dd |X_t-Y_t|^2}{\dd t} & = & -2\left[\nabla U(X_t)-\nabla U(Y_t)\right]\cdot\left[X_t-Y_t\right] \ 
 \le \ -2\rho |X_t-Y_t|^2,
\end{eqnarray*}
 thanks to (\ref{assume:convexity}), and this shows that $\kappa_\infty\ge \rho$. Conversely, we always have
\begin{eqnarray*}
\left|\EE[X_t]-\EE[Y_t]\right| & \le & W_1\left(P_t(x,\cdot),P_t(y,\cdot)\right) \ \le \ e^{-\kappa_1 t}|x-y|,
\end{eqnarray*}
and comparing the first-order Taylor expansion on both sides as $t\to 0$ yields
\begin{eqnarray*}
[x-y]\cdot[\nabla U(x)-\nabla U(y)] & \ge & \kappa_1 |x-y|^2.
\end{eqnarray*}
Since the points $x,y\in\dR^d$ are arbitrary, we deduce that $\mathrm{Hess\,}(U)\ge \kappa_1\textrm{Id}$. Thus, we have shown $\kappa_1\le \rho\le \kappa_\infty$, and the conclusion follows because $p\mapsto\kappa_p$ is non-increasing. 
\end{proof}
\begin{remark}[Log-Sobolev constant] Theorem \ref{th:PT} gives $\alpha\ge  2\rho$, which is equivalent to $\beta\ge \rho/2$ thanks to (\ref{MLSI=LSI}). Thus, we recover exactly the Bakry-\'Emery result (\ref{BELSI}).
\end{remark}
For random walks on groups, curvature turns out to be related to the commutativity structure of the generators. In particular, we always have $\rho,\kappa_\infty\ge 0$ on Abelian groups.
\begin{theorem}[Conjugacy-invariant walks are non-negatively curved]\label{th:CIW} For the random walk generated  by a measure $\nu\in\cP(\dX)$ on a finite group  $\dX$, we have
\begin{enumerate}
\item non-negative sectional curvature ($\kappa_\infty\ge 0$, hence also $\kappa_1\ge 0$), as soon as 
\begin{eqnarray}
\label{def:CIsupp}
\forall x,y\in\dX,\qquad \nu(xy)>0 & \Longrightarrow & \nu(yx)>0;
\end{eqnarray} 
\item non-negative Bakry-\'Emery curvature ($\rho\ge 0$), as soon as 
\begin{eqnarray}
\label{def:CI}
\forall x,y\in\dX,\qquad \nu(xy) & = & \nu(yx).
\end{eqnarray}
\end{enumerate}
\end{theorem}
\begin{proof}
Since a transition here consists in left-multiplying the current position by a $\nu-$distributed variable $Z$, we can couple the transitions from two neighboring states $x,y\in\dX$ by setting
\begin{eqnarray*}
(X,Y) & := & (Zx,Zy).
\end{eqnarray*}
Now, under the assumption (\ref{def:CIsupp}), this parallel coupling preserves adjacency since we have
\begin{eqnarray*}
T(X,Y)>0 & \Longleftrightarrow & \nu(YX^{-1})>0 \\
& \Longleftrightarrow & \nu\left(Zyx^{-1}Z^{-1}\right)>0\\
& \Longleftrightarrow & \nu\left(yx^{-1}\right)>0\\
& \Longleftrightarrow & T(x,y)>0,
\end{eqnarray*}
 and the first claim follows.  For the second, we make two simplifying observations. First, by symmetry, the Bakry-\'Emery criterion (\ref{criterion:CD}) needs only be verified at a single point, say the identity element $\id$. Second, since the functions $Lf,\Gamma(f,g)$ and $\Gamma_2(f,g)$ do not change upon adding a constant to $f$ or $g$, we may restrict the proof of  (\ref{criterion:CD}) to test functions $f$ that satisfy $f(\id)=0$. With those simplifications at hand, we easily compute:
\begin{eqnarray}
\label{Gamma2L}
4\Gamma_2 f(\id) & = & \sum_{z,w\in\dX}\nu(z)\nu(w)F(z,w)+2\left(\sum_{z\in\dX}\nu(z)f(z)\right)^2,
\end{eqnarray}
where we have introduced the short-hand
\begin{eqnarray*}
F(z,w) & := & f^2(zw)-4f(z)f(zw)+2f^2(z).
\end{eqnarray*}
Now, under assumption (\ref{def:CI}), the bijection $(z,w)\mapsto(w,w^{-1}zw)$ preserves the measure $\nu\otimes\nu$. Thus, we may  replace $F(z,w)$ in (\ref{Gamma2L}) with $F(w,w^{-1}zw)$ or even with 
\begin{eqnarray*}
\frac 12\left(F(z,w)+F(w,w^{-1}zw)\right)  
& = & \left(f(zw)-f(z)-f(w)\right)^2-2f(z)f(w).
\end{eqnarray*}
Inserting this back into the above computation finally gives
\begin{eqnarray}
\label{CIconclude}
4\Gamma_2 f(\id) & = & \sum_{z,w\in\dX}\nu(z)\nu(w) \left(f(zw)-f(z)-f(w)\right)^2,
\end{eqnarray}which is indeed non-negative, as desired. 
\end{proof}
\begin{remark}[Self-inverse random walks]\label{rk:kappa}The above argument can be substantially refined under the simple additional assumption that  $\nu$ is supported on \red[self-inverse] elements, i.e.
\begin{eqnarray*}
\forall z\in\dX,\qquad \nu(z)>0 & \Longrightarrow & z=z^{-1}.
\end{eqnarray*}
Indeed, keeping only the diagonal terms in (\ref{CIconclude}) and using   $f(z^2)=f(\id)=0$, we find
\begin{eqnarray*}
4\Gamma_2 f(\id ) & \ge &  4\sum_{z\in\dX}\nu^2(z)f^2(z) \ \ge \  8\nu_{\star}\Gamma f(\id),
\end{eqnarray*}
where  $\nu_{\star}$ denotes the smallest non-zero value of $\nu$. This shows that $\rho \ge   2\nu_{\star}$, which is a much stronger conclusion than the second part of the theorem. The first part can also be improved. Indeed, a transition of the lazy chain $\widehat{T}=(T+{\mathrm Id})/2$ corresponds to left-multiplying the current position by $Z^B$, with $Z$ as above and $B$  an independent Bernoulli$(1/2)$ variable. Given two neighbors $x,y\in\dX$, we may thus couple $\widehat{T}(x,\cdot)$ and $\widehat{T}(y,\cdot)$ by setting
\begin{eqnarray*}
(X,Y) & := & \left\{
\begin{array}{ll}
(Z^B x,Z^B y) & \textrm{if } Z\ne yx^{-1}\\
(Z^{B}x,Z^{1-B}y) & \textrm{if }Z = yx^{-1}.
\end{array}
\right.
\end{eqnarray*}
In the first case, we have $\dist(X,Y)=1$ thanks to (\ref{def:CIsupp}), as before. However, in the second case, we now have $X=Y$ because $yx^{-1}=xy^{-1}$. Thus, $\EE[\dist(X,Y)]\le 1-\nu_\star$, showing that $
\kappa_1  \ge  2\nu_\star.
$
When applied  to random walk on the $n-$cube, this gives $\rho\ge 2/n$ and $\kappa_1\ge 2/n$, which happen to exactly reverse the Lichnerowicz estimates in Remarks \ref{rk:Lichnerowicz}-\ref{rk:Lichnerowicz2}. Thus, 
\begin{eqnarray*}
\rho & = & \kappa_1\ = \ \gamma\ =\ 2/n.
\end{eqnarray*}
\end{remark}
Finally, let us turn to Monte Carlo Markov Chains for sampling from a target distribution $\pi$ on $\{-1,1\}^n$, as defined at (\ref{def:Glauber}). Recall that under this dynamics, the $i-$th coordinate of $x$ is flipped at a prescribed rate $c_i(x)$, chosen so as to ensure reversibility w.r.t. $\pi$. In this context, curvature turns out to be related to the product-like nature of $\pi$. 
\begin{theorem}[Non-negative curvature for MCMC]\label{th:MCMC}Consider a generator of the form 
(\ref{def:Glauber}) and suppose that for each $x\in\dX$ and $i\in[n]$, the quantity
\begin{eqnarray*}
\delta_i(x) & := & c_i(x)-\sum_{j\in[n]\setminus\{i\}}\left(c_j(x^i)-c_j(x)\right)_+,
\end{eqnarray*}
is non-negative. Then $\kappa_\infty=0$ and $\kappa_1\ge \min_{x,i}\{\Delta_i(x)+\Delta_i(x^i)\}$, hence in particular
\begin{eqnarray*}
\alpha & \ge & \min_{x,i}\{\delta_i(x)+\delta_i(x^i)\}.
\end{eqnarray*}
Moreover, in the special case $c_i(x)=\sqrt{\frac{\pi(x^i)}{\pi(x)}}$, we also have
\begin{eqnarray*}
\rho & \ge & \min_{x,i}\{\delta_i(x)\}+\frac{1}{2}\min_{x,i}\{\Delta_i(x)+\delta_i(x^i)\}.
\end{eqnarray*}
\end{theorem}
\begin{proof}Upon dividing all rates by a large enough positive constant, we may assume that $L=T-\mathrm{Id}$ for some lazy stochastic matrix $T$, in compliance with out setting. Now, consider two states $x,y\in\dX$ that differ at a single coordinate (say, the $i-$th one), and let us construct a coupling $(X,Y)$ of  $T(x,\cdot)$ and $T(y,\cdot)$ as follows:
\begin{eqnarray*}
\left(X,Y\right) & := & \left\{
\begin{array}{ll}
(x^i,y) & \textrm{with probability }\delta_i(x)\\
(x,y^i) & \textrm{with probability }\delta_i(y)\\
 (x^j,y^j) & \textrm{with probability }c_j(x)\wedge c_j(y), \textrm{ for any }j\in[n]\setminus\{i\}\\
 (x^i,y^j) & \textrm{with probability }\left(c_j(y)-c_j(x)\right)^+, \textrm{ for any }j\in[n]\setminus\{i\}\\
 (x^j,y^i) & \textrm{with probability }\left(c_j(x)-c_j(y)\right)^+, \textrm{ for any }j\in[n]\setminus\{i\}\\
(x,y) & \textrm{otherwise}.
\end{array}
\right.
\end{eqnarray*}
It is not hard to check that this is indeed a coupling of $T(x,\cdot)$ and $T(y,\cdot)$. Moreover, we have $\dist(X,Y)=0$ in the first two cases, and $\dist(X,Y)=1$ in all other cases. Thus, $\PP(\dist(X,Y)\le 1)=1$  and 
$
\EE[\dist(X,Y)]  =  1-\delta_i(x)-\delta_i(y),
$
which readily leads to the claims about $\kappa_\infty$ and $\kappa_1$. While also based on the above coupling, the estimate on $\rho$ is slightly more involved, so we lazily refer the reader to \cite{pedrotti2023contractivecouplingratescurvature} for details.
\end{proof}
\begin{remark}[Weak dependency]In the case $c_i(x)=\sqrt{\frac{\pi(x^i)}{\pi(x)}}$, the condition  $\delta_i(x^i)\ge 0$ reads
\begin{eqnarray}
\label{weakdep}
\sum_{j\in[n]\setminus\{i\}}\left(\sqrt{\frac{\pi(x^{j})\pi(x^i)}{\pi(x)\pi(x)}}-\sqrt{\frac{\pi(x^{ij})}{\pi(x)}}\right)_+ & \le & 1.
\end{eqnarray}
Note that the left-hand side measures the influence of the $i-th$ coordinate on the other ones, under the measure $\pi$. In particular, it vanishes when the $i-$th coordinate is independent of the $n-1$ others, and it is small when $\pi$ has weak dependencies. For example, in the hard-core model with fugacity $\zeta$ on a graph of maximum degree $\Delta$, (\ref{weakdep}) holds as soon as
\begin{eqnarray*}
\Delta  \zeta & \le & 1,
\end{eqnarray*}
while for the Ising model with inverse temperature $\beta$, (\ref{weakdep}) holds as soon as 
\begin{eqnarray*}
\Delta (1-e^{-2\beta})e^{2\Delta \beta} & \le & 1.
\end{eqnarray*} 
Similar conclusions can be derived when $c_i(x)=\frac{\pi(x^i)}{\pi(x^i)+\pi(x)}$ or $c_i(x)=1\wedge \frac{\pi(x^i)}{\pi(x)}$.
\end{remark}
\begin{exo}[Zero-Range Process]Consider the mean-field Zero-Range process defined in Section \ref{sec:interacting}, and assume that the rate functions are non-decreasing, i.e. 
\begin{eqnarray*}
\forall i\in[n],\quad\forall k\in[m],\quad r_i(k)-r_i(k-1) & \ge & \delta,
\end{eqnarray*}
for some $\delta\ge 0$. Prove that $\kappa_\infty\ge 0$ and  $\kappa_1\ge \delta$, so that $\alpha\ge \delta$.
\end{exo}
\newpage
\section{The varentropy approach}
In this final chapter, whose content is borrowed from the recent papers \cite{MR4780485,salez2025cutoffnonnegativelycurveddiffusions,
pedrotti2025newcutoffcriterionnonnegatively}, we return to our main challenge:  developing a unified theory of cutoff. As already explained, the traditional approach is restricted to highly structured models, because it is  based on the overly delicate task of determining exactly when the transition to equilibrium starts and ends. We propose instead to focus on the \emph{duration} of the transition, i.e., on the difference
\begin{eqnarray}
\label{def:wmix}
\wmix(\varepsilon) & := & \tmix(\varepsilon)-\tmix(1-\varepsilon),
\end{eqnarray}
called the \red[width] of the mixing window. 
This is motivated by the general principle, embodied by the Central Limit Theorem and many other universality results in probability, that first-order asymptotics are often highly dependent on the specific details of the model, whereas their second-order counterparts tend to be more robust and amenable to a systematic analysis. Of course, it is unclear a priori that the  difference in (\ref{def:wmix}) can be estimated without having to actually compute each of the two terms separately. Fortunately, this turns out to be possible, thanks to a beautiful information-theoretic statistics known as varentropy. 
\subsection{The varentropy criterion}
Consider a random variable $X$ with density $f$ on our workspace $(\dX,\pi)$. In information theory, the real-valued random variable $-\log f(X)$ is known as the \red[information content], because it quantifies the amount of \emph{suprise}, or \emph{unpredictability} inherent to the observation of $X$. We find it more convenient to drop the minus sign here, so that  the expected information coincides with our definition of entropy:  
\begin{eqnarray*}
\Ent(f) & = & \EE\left[\log f(X)\right].
\end{eqnarray*}
Motivated by the conceptual shift described above, we propose to investigate a second-order version of entropy, obtained by replacing the above expectation with a variance:
\begin{eqnarray*}
\Varent(f) & := &  \Var\left[\log f(X)\right].
\end{eqnarray*}
Because it measures the dispersion of information around the entropy, this quantity is  called \red[varentropy]. It appeared a decade ago in the completely different context of data compression, to quantify the error in the celebrated \emph{Asymptotic Equipartition Property} \cite{inproceedings}. Its main interest for us lies in its ability to reverse Pinsker's inequality, which asserts that
\begin{eqnarray*}
2\TV^2(f) & \le & \Ent(f),
\end{eqnarray*}
where $\TV(f):=\frac 12\|f-1\|_1$ denotes the total-variation distance between the law of $X$ and $\pi$. 
\begin{lemma}[Reversed Pinsker Inequality]\label{lm:reversepinsker}For any density $f\in\cF_1^+$, we have 
\begin{eqnarray*}
\Ent(f) & \le & \frac{1+\sqrt{\Varent(f)}}{1-\TV(f)}.
\end{eqnarray*}
\end{lemma}
\begin{proof}Let $X$ have density $f$, and consider the event $A := \left\{x\in \dX\colon  \log f(x)\ge \theta\right\}$, where $\theta\in\dR$ is a parameter that will be adjusted later. By definition, we then have
\begin{eqnarray*}
\TV(f) & \ge & \PP(X\in A)-\pi(A) \ = \ \sum_{x\in A}\PP(X=x)\left(1-\frac{1}{f(x)}\right) \ \ge \ \PP(X\in A)\left(1-e^{-\theta}\right).
\end{eqnarray*}
We now fix a precision $\varepsilon\in(0,1)$ to be adjusted later, and choose $\theta$ as follows:
\begin{eqnarray*}
\theta & = & \Ent(f)-\frac{\sqrt{\Varent(f)}}{\varepsilon}.
\end{eqnarray*}
Since the random variable $\log X$ has mean $\Ent(f)$ and variance $\Varent(f)$,  Chebychev's inequality ensures that $\PP\left(X\in A\right)  \ge    1-\varepsilon^2$. Inserting this above, we obtain
\begin{eqnarray*}
\TV(f) & \ge & (1-\varepsilon^2)\left(1-e^{-\theta}\right).
\end{eqnarray*}
Finally, let us choose $\varepsilon:=1-\TV(f)$. Rearranging the above inequality, we arrive at
\begin{eqnarray*}
\theta & \le & \log\left(1+\frac{1}{\varepsilon}\right), 
\end{eqnarray*}
and since the right-hand side is at most $1/\varepsilon$, the claimed inequality follows.  
\end{proof}
The interest of this reversed Pinsker inequality will become clear once we combine it with the observation that the convergence to equilibrium is fast (namely, of order the inverse Poincaré constant) when the initial condition already has low entropy. To formalize this, let us generalize our definition of $\tmix$ as follows: for any initial density $f\in\cF_1^+$, we let
\begin{eqnarray}
\tmix(f,\varepsilon) & := & \min\left\{t\ge 0\colon \TV(P_t^\star f) \le \varepsilon\right\},
\end{eqnarray}
denote the $\varepsilon-$mixing-time of our Markov process $(X_t)_{t\ge 0}$ when the initial condition $X_0$ is distributed according to $f$. We simply write $\tmix(x,\varepsilon)$ in the extremal case where $f=f^x$. Note that our earlier definition is then obtained as  $\tmix(\varepsilon)=\max_{x\in\dX}\tmix(x,\varepsilon)$.
\begin{lemma}[Fast mixing from low-entropy initializations]\label{lm:fastmixing} We have
\begin{eqnarray*}
\forall f\in\cF_1^+,\quad\forall \varepsilon\in(0,1),\quad 
\tmix(f,\varepsilon) & \le & \frac{1+\Ent(f)}{\gamma\varepsilon}.
\end{eqnarray*}
\end{lemma}
\begin{proof}
Recall from Section \ref{sec:poincaré} that the Poincaré constant $\gamma$  controls the variance, which in turns controls total-variation thanks to the Cauchy-Schwarz inequality. More precisely,
\begin{eqnarray}
\label{gap:f}
2\TV(P_t^\star f) \ = \ \|P_t^\star f-1\|_1 
& \le & \|P_t^\star f-1\|_2 \ = \ \sqrt{\Var(P_t^\star f)} \ \le \ e^{-\gamma t}\sqrt{\Var(f)},
\end{eqnarray}
which immediately leads to the mixing-time estimate
\begin{eqnarray*}
\tmix(f,\varepsilon)& \le & \frac{1}{2\gamma}\log\left(\frac{\Var(f)}{4\varepsilon^2}\right).
 \end{eqnarray*}
Unfortunately this does not suffice, because  $\log \Var(f)$ can be much larger than $\Ent(f)$ if $f$ takes abnormally large values on a small region. To preclude such pathologies, we replace the initial law $ \mu:=f\dd \pi$ with its conditioned version $\widehat{\mu}:=\mu(\cdot|A)$, where 
\begin{eqnarray*}
A & := & \left\{x\in\dX\colon \log f(x) < 1+\frac{2\Ent(f)}{\varepsilon}\right\},
\end{eqnarray*}
represents the region where $f$ takes \emph{typical} values. Observe that by definition, we then have
\begin{eqnarray*}
\left(1+\frac{2\Ent(f)}{\varepsilon}\right) \mu(A^c) & \le & \sum_{x\in A^c}\mu(x)\log f(x) \\& = & \Ent(f)+\sum_{x\in A}\mu(x)\log f(x)\\ & \le & \Ent(f)+\pi(A)-\mu(A)\\
& \le  & \Ent(f)+\mu(A^c),
\end{eqnarray*}
where at the third line we have used $\log u\le u-1$. After simplification, we are left with
\begin{eqnarray*}
\mu(A^c) & \le & \frac{\varepsilon}{2},
\end{eqnarray*}
confirming our earlier use of the word \emph{typical}. Now, the conditional density  $\widehat{f}:= f{\bf 1}_A/\mu(A)$ is by construction \emph{non-pathological}, in the sense that 
\begin{eqnarray*}
\|\widehat{f}\|_\infty& 
 = & \frac{1}{\mu(A)}\max_{x\in A}f(x) \ \le \ \exp\left\{2+\frac{2\Ent(f)}{\varepsilon}\right\},
\end{eqnarray*}
where we have used $\mu(A)\ge 1/2 \ge 1/e$. This provides a bound on the variance, because  $\Var(\widehat{f})\le \|\widehat{f}\|_2^2\le \|\widehat{f}\|_1\|\widehat{f}\|_\infty=\|\widehat{f}\|_\infty$. Thus, applying  (\ref{gap:f}) to $\widehat{f}$ instead of $f$ yields
\begin{eqnarray*}
\|P_t^\star \widehat{f}-1\|_1  & \le & \exp\left\{1+\frac{\Ent(f)}{\varepsilon}-\gamma t\right\},
 \end{eqnarray*}
which can be made less than $\varepsilon$ by choosing
\begin{eqnarray*}
\label{def:t}
t & := & \frac{1+\Ent(f)}{\gamma\varepsilon}.
\end{eqnarray*}
On the other hand, the trivial contraction property $\|P_t^\star\|_{1\to 1}\le 1$ implies that
\begin{eqnarray*}
\|P_t^\star \widehat{f} - P_t^\star f\|_1 & \le & \|\widehat{f}-f\|_{1}  \ = \ 2\mu(A^c) \ \le \ \varepsilon.
\end{eqnarray*}
By the triangle inequality, we deduce that $\|P_t^\star f-1\|_{1}  \le  2\varepsilon$, i.e. $\TV(P_t^\star f)\le \varepsilon$. 
\end{proof}

We are now in position to establish the following  universal estimate on the  mixing window 
\begin{eqnarray}
\wmix(f,\varepsilon) & := & \tmix(f,\varepsilon)-\tmix(f,1-\varepsilon),
\end{eqnarray}
of any Markov chain starting from any initial density $f$, in terms of varentropy. 
\begin{theorem}[Varentropy controls the width of the mixing window]\label{th:main}We have
\begin{eqnarray*}
\forall f\in\cF_1^+,\quad\forall \varepsilon\in(0,1/2),\quad \wmix(f,\varepsilon) & \le & \frac{2}{\gamma \varepsilon^2}\left(1+\sqrt{V_{f,\varepsilon}}\right),
\end{eqnarray*}
where $V_{f,\varepsilon}:=\Varent(P_t^\star f)$ is the varentropy at time $t=\tmix(f,1-\varepsilon)$. Similarly,
\begin{eqnarray*}
\wmix(\varepsilon) & \le & \frac{2}{\gamma \varepsilon^2}\left(1+\sqrt{V_{\varepsilon}}\right),
\end{eqnarray*}
where $V_\varepsilon:=\max_{x\in\dX}\sqrt{\Varent(P_t^\star f^x)}$ is the worst-case varentropy at time $t = \tmix(1-\varepsilon)$.
\end{theorem}
\begin{proof}Using successively the semi-group property and Lemmas \ref{lm:fastmixing} and \ref{lm:reversepinsker} we can write
 \begin{eqnarray*}
\tmix(f,\varepsilon)  & \le & t+ \tmix(P_t^\star f,\varepsilon)\\
& \le & t+\frac{1+\Ent(P_t^\star f)}{\gamma\varepsilon}\\
& \le & t+\frac{\varepsilon+1+\sqrt{\Varent(P_t^\star f)}}{\gamma\varepsilon^2},
\end{eqnarray*} 
where the third line requires $t\ge \tmix(f,1-\varepsilon)$. Choosing $t=\tmix(f,1-\varepsilon)$ yields exactly the first statement, while the second is obtained by setting $t=\tmix(1-\varepsilon)$ and then taking a maximum over all singleton-supported initial densities. 
\end{proof}
The above result provides a varentropy-based criterion for cutoff, valid from any initialization. More precisely, given a model $(T_n)_{n\ge 1}$ and a sequence of initial densities $(f_n)_{n\ge 1}$ on the corresponding state spaces, we naturally define cutoff from $(f_n)_{n\ge 1}$ as follows:
\begin{eqnarray*}
\forall \varepsilon\in(0,1),\quad
\frac{\tmix^{(n)}(f_n,1-\varepsilon)}{\tmix^{(n)}(f_n,\varepsilon)} & \xrightarrow[n\to\infty]{} & 1,
\end{eqnarray*}
or equivalently, $\wmix^{(n)}(f_n,\varepsilon)\ll \tmix^{(n)}(f_n,\varepsilon)$ for each $\varepsilon\in(0,1/2)$. As usual, we will keep the size parameter $n$ implicit for notational convenience.
\begin{coro}[Varentropy criterion]\label{co:main}
A model exhibits cutoff from $f$ if 
\begin{eqnarray*}
\gamma   \tmix(f,\varepsilon) & \gg & 1+\sqrt{V_{f,\varepsilon}},
\end{eqnarray*}
 for  every fixed $\varepsilon\in(0,1/2)$. Similarly, a model exhibits a worst-case cutoff as soon as
  \begin{eqnarray*}
\gamma   \tmix(\varepsilon) & \gg & 1+\sqrt{V_{\varepsilon}}.
\end{eqnarray*}
\end{coro}
To better appreciate this result, recall that for reversible chains, the condition $\gamma   \tmix(\varepsilon)  \gg  1$ is exactly the product condition discussed in Section \ref{sec:PC}. What Corollary \ref{co:main} asserts is that adding the varentropy correction $\sqrt{V_{\varepsilon}}$ to the right-hand side is enough to upgrade this famous necessary criterion into a sufficient one, and that this is universal: the criterion applies to any model, and from any initial condition. To the best of our knowledge, this is the very first general sufficient condition for the occurrence of cutoff. However, in its present form, it is much more a starting point  than a definitive answer to our main problem, because varentropy  remains to be estimated before it can be effectively  used to predict cutoff. Doing so  on our favorite toy model, at least,  is an easy exercise.
\begin{exo}[Tensorization]Prove that varentropy \red[tensorizes]: for any product density $f=f_1\otimes\cdots \otimes f_n$ on a product space $(\dX,\pi)=(\dX_1,\pi_1)\otimes \cdots\otimes(\dX_n,\pi_n)$, we have
\begin{eqnarray*}
\Varent(f) & = & \Varent(f_1)+\cdots+\Varent(f_n).
\end{eqnarray*}
Deduce that simple random walk on the $n-$cube starting from any state $x\in\dX$ satisfies
\begin{eqnarray*}
\Varent(P_t^\star f^x) & = & \frac{n}{4}\left(1-e^{-\frac{4t}{n}}\right)\left(\log\frac{1+e^{-\frac{2t}{n}}}{1-e^{-\frac{2t}{n}}}\right)^2,
\end{eqnarray*}
and conclude that the worst-case varentropy correction $V_\varepsilon$ remains bounded uniformly in $n$.   
\end{exo}
Of course, this computation easily extends to more general product chains, but models of that sort remain sufficiently simple to allow for explicit computations, and hence do not constitute a very convincing illustration of the varentropy approach. We will now develop a systematic method to estimate the varentropy of non-negatively curved chains, and use it to provide novel, simple and unified answers to several  problems mentioned in Chapter \ref{sec:models}. 
\subsection{A first naive estimate on varentropy}
We henceforth assume that $T$ is weakly reversible, and  equip $\dX$  with the hop-count distance (\ref{def:hopcount}). As we have seen in Chapter \ref{sec:curvature}, Markov chains whose curvature is non-negative in either the Ollivier-Ricci  or  the Bakry-\'Emery sense satisfy the weak local Poincaré inequality
\begin{eqnarray*}
\forall g\in\cF,\quad\forall t\ge 0,\quad\Var\left(g(X_t)\right) & \le & t \,\Lip^2(g),
\end{eqnarray*}
for any deterministic initialization $X_0=x\in\dX$. In particular, the density $f_t$ of $X_t$ satisfies
\begin{eqnarray}
\label{varent:naive}
\Varent(f_t) & \le & t\, \left(\Lip \log f_t\right)^2,
\end{eqnarray}
which reduces our task to estimating the roughness of $f_t=P_t^\star f_0$. Here is a simple, universal bound in terms of our sparsity parameter $d$ defined at (\ref{def:d}). 
\begin{lemma}[Universal roughness estimate]\label{lm:lip}For any non-zero $f\in \cF^+$ and any $t> 0$,
\begin{eqnarray*}
\Lip \log P_t f & \le & 3\log d+2\log_+\left(\frac{\diam(\dX)}{t}\right).
\end{eqnarray*}
Moreover, the same conclusion applies to $P_t^\star$ instead of $P_t$, with $d$ replaced by $d^2$.
\end{lemma}
\begin{proof}
Fix $t>0$ and set $p_k:=e^{-t}t^k/k!$, so that we can write
\begin{eqnarray*}
P_t & = & \sum_{k=0}^\infty p_k T^{k} \ = \ \frac{1}{t}\sum_{k=1}^\infty {kp_{k}}T^{k-1}, 
\end{eqnarray*}
thanks to the Poisson identity $tp_k=(k+1)p_{k+1}$ and a change of indices. In particular, 
\begin{eqnarray*}
\frac{TP_t f}{P_tf} & = & \frac{1}{t}\frac{\sum_{k=0}^\infty k p_{k} T^{k}f}{\sum_{k=0}^\infty p_{k} T^{k}f} \
 \le \ \frac{1}{t}\log\left(\frac{\sum_{k=0}^\infty p_{k}e^k T^{k}f}{\sum_{k=0}^\infty p_{k} T^{k}f}\right),
 \end{eqnarray*}
by Jensen's inequality. Let us now estimate the ratio inside the logarithm. For the numerator, we simply use the crude bound $T^kf\le \|T^k f\|_\infty\le \|f\|_\infty$ to obtain
\begin{eqnarray*}
\sum_{k=0}^\infty p_{k}e^k T^{k}f & \le & \|f\|_\infty e^{(e-1)t},
\end{eqnarray*}
while for the denominator, we write
\begin{eqnarray*}
\sum_{k=0}^\infty p_{k} T^{k}f & \ge & \|f\|_\infty \min_{(x,y)\in\dX^2}\left\{\sum_{k=0}^\infty p_{k} T^{k}(x,y)\right\},
\end{eqnarray*}
and  estimate the minimum on the right-hand side in the following crude way: given $(x,y)\in\dX^2$, we consider a path $(x_0,\ldots,x_\ell)$ from $x_0=x$ to $x_\ell=y$ of length  $\ell=\dist(x,y)$. Using the definition of $d$  and the classical bound  $\ell!\le \ell^\ell$ (with $0^0=1$), we have
\begin{eqnarray*}
\sum_{k=0}^\infty p_kT^k(x,y) & \ge & p_\ell T(x_0,x_1)\cdots T(x_{\ell-1},x_\ell)\ \ge \ e^{-t}\left(\frac{t}{d\ell}\right)^\ell,
 \end{eqnarray*} 
and since $u\mapsto u\log u$ is convex, the right-hand side is minimized when $\ell=0$ or $\ell=\diam(\dX)$. In view of the last four displays, we conclude that
\begin{eqnarray*}
\left\|\frac{TP_t f}{P_tf} \right\|_\infty & \le & e+\frac{\diam(\dX)}{t}\log_+\left(\frac{d\diam(\dX)}{t}\right).
 \end{eqnarray*}
On the other hand,  we  have by definition
\begin{eqnarray*}
\left\|\frac{TP_t f}{P_tf} \right\|_\infty & = & \max_{x\in\dX}\left\{\sum_{y\in\dX}\frac{T(x,y)P_tf(y)}{P_tf(x)}\right\} \ \ge \  \max_{(x,y)\in E}\left\{\frac{T(x,y)P_tf(y)}{P_tf(x)}\right\} \ \ge \ \frac{e^{\Lip\log P_tf}}{d}.
\end{eqnarray*}
Combining those two bounds yields 
\begin{eqnarray*}
\Lip\log P_tf & \le & \log\left[de+\frac{d\diam(\dX)}{t}\log_+\left(\frac{d\diam(\dX)}{t}\right)\right],
\end{eqnarray*}
which easily leads to the main claim. Finally,  the identity 
\begin{eqnarray*}
\forall (x,y)\in\dX^2,\qquad T^\star(x,y)T^\star(y,x) & = & T(x,y)T(y,x),
\end{eqnarray*}
shows that the non-zero entries of $T^\star$ are at least $d^{-2}$, hence the second claim. 
\end{proof}
Combining this with the local Poincaré inequality (\ref{varent:naive}) and the crude diameter bound of Lemma \ref{lm:diammix}, we obtain a first non-trivial application of our varentropy criterion.  
\begin{coro}[Cutoff for non-negatively curved chains]\label{coro:var}For  weakly reversible chains with non-negative curvature, the worst-case varentropy correction in Theorem \ref{th:main} satisfies
\begin{eqnarray*}
\forall \varepsilon\in(0,1/2),\quad V_{\varepsilon} & \le & C_\varepsilon (\log d)^2\,\tmix,
\end{eqnarray*}
for a constant $C_\varepsilon\in(0,\infty)$ that depends only on $\varepsilon$.   In particular, cutoff occurs as soon as,
\begin{eqnarray}
\label{criterion:naive}
\tmix& \gg & \left(\frac{\log d}{\gamma}\right)^2.
\end{eqnarray}
\end{coro}
As we have seen, the mixing time is bounded below by the diameter (Lemma \ref{lm:diammix}), which is in turns easily bounded below by $(\log |\dX|)/\log d$. Thus, the criterion (\ref{criterion:naive}) holds as soon as the spectral gap is not too small, in the following sense:
\begin{eqnarray}
\label{criterion:naive2}
\gamma & \gg & \sqrt{\frac{(\log d)^{3}}{\log |\dX|}}.
\end{eqnarray}
In particular, this applies to random Abelian Cayley graphs. Indeed, we have seen in Theorem \ref{th:CIW} that random walks on Abelian groups are non-negatively curved, and a celebrated result of N. Alon and Y. Roichman \cite{MR1262979}, refined by  I. Pak \cite{MR1729149}, A. Naor \cite{MR2942733}, and finally J. Hermon and S. Olesker-Taylor  \cite{hermon2021cutoff}, asserts that on a finite group $\dX_n$ with diverging size,  the random Cayley graph $G_n=\mathrm{Cay}(\dX_n,S_n\cup S_n^{-1})$ obtained by choosing  $S_n\subseteq\dX_n$ uniformly at random among all sets of size $d_n\ge (1+\delta)\log_2|\dX_n|$ satisfies $\gamma(G_n)\ge C_\delta$ with high probability, where $C_\delta>0$ is a constant that depends only on $\delta$. Thus, we obtain the following striking result. 
\begin{coro}[Cutoff on almost all Abelian Cayley graphs]\label{co:abelian}For each $n\ge 1$, let $\dX_n$ be a finite Abelian group and let $G_n=\mathrm{Cay}(\dX_n,S_n\cup S_n^{-1})$ be the random Cayley graph obtained by choosing  $S_n\subseteq\dX_n$ uniformly at random among all sets of size $d_n$. Consider the regime
\begin{eqnarray*}
d_n \ge  (1+\delta)\log_2 |\dX_n| & \textrm{ and } & \log d_n \ll \left(\log |\dX_n|\right)^{\frac 13},
\end{eqnarray*}
where $\delta>0$ is any fixed constant. Then, the random walk on $G_n$ exhibits cutoff in probability:
\begin{eqnarray*}
\forall \varepsilon\in(0,1),\quad \frac{\tmix^{(n)}(1-\varepsilon)}{\tmix^{(n)}(\varepsilon)} & \xrightarrow[n\to\infty]{\PP} & 1.
\end{eqnarray*}
\end{coro}
Note that the  required lower-bound on $d_n$ is optimal, since the binary group $\dX_n:=\dZ_2^n$ can not be generated by less than $d_n=\log_2|\dX_n|$ elements. The problem of establishing cutoff for random Cayley graphs has a long history (see the  survey \cite{MR2121795}), initiated with a conjecture raised by D. Aldous and P. Diaconis in an extended version of \cite{aldous1986shuffling}. The dense regime $d_n\gg\log |\dX_n|$ is relatively easy, and was settled several years ago by C. Dou and M. Hildebrand \cite{MR1404540,MR1296428}. The sparse regime $d_n=O(\log |\dX_n|)$, in contrast,  was tackled only very recently in an impressive series of works by J. Hermon and S. Olesker-Taylor  \cite{hermon2021cutoff,nonabelian,hermon2021results,hermon2021geometry}. We emphasize that their approach relies on a very detailed analysis of the typical asymptotic structure of random Abelian Cayley graphs, whereas  our Corollary \ref{coro:var} deterministically applies to \emph{any} Abelian Cayley graph with a reasonably large spectral gap. 
\subsection{The Information-theoretic Differential Inequality}
To push the varentropy approach further, we now introduce and exploit a recently discovered differential relation between entropy and varentropy, whose general form is as follows. 
\begin{definition}[IDI]
A Markov chain $(X_t)_{t\ge 0}$ satisfies an \red[information-theoretic differential inequality] (IDI) with rate $\psi\colon(0,\infty)\to(0,\infty)$ if 
\begin{eqnarray}
\label{def:IDI}
\forall t>0,\qquad \Varent(f_t) & \le & -\psi(t)\,\frac{\dd\Ent(f_t)}{\dd t},
\end{eqnarray}
where $f_t$ denotes the density of $X_t$ w.r.t. equilibrium.
\end{definition}
To motivate this intriguing definition, let us show right away how information-differential inequalities lead to powerful quantitative estimates on the width of the mixing window.
\begin{lemma}[Width estimates]\label{lm:IDI}If (\ref{def:IDI}) holds for some initial density $f_0$, then 
\begin{eqnarray*}
\forall\varepsilon\in(0,1/2),\quad \wmix(f_0,\varepsilon) & \le & \frac{1}{\gamma\varepsilon^3}+\sqrt{\frac{4m_\varepsilon}{\gamma\varepsilon^3}},
\end{eqnarray*}
where $m_\varepsilon$ denotes the supremum of $\psi$ inside the mixing window $[\tmix(f_0,1-\varepsilon),\tmix(f_0,\varepsilon)]$. Using $\alpha$ rather than $\gamma$, we also have  the alternative estimate
\begin{eqnarray*}
\forall\varepsilon\in(0,1/2),\quad \wmix(f_0,\varepsilon) & \le & \frac{1}{\alpha}\log\left(\frac{e\alpha(1+m_\varepsilon)}{2\varepsilon^4}\right).
\end{eqnarray*}
\end{lemma}
\begin{proof}We use the short-hands $t_0:=\tmix(f_0,1-\varepsilon)$,  $t_1:=\tmix(f_0,\varepsilon)$, and $H(t):=\Ent(f_t)$. Combining the IDI (\ref{def:IDI}) with the reverse Pinsker inequality  (Lemma \ref{lm:reversepinsker}), we obtain 
\begin{eqnarray*}
m_\varepsilon H'(t) & \le & -\left(\varepsilon H(t)-1\right)^2, 
\end{eqnarray*}
for all  $t\in[t_0,t_1]$. Integrating this with respect to $t$, we deduce that
\begin{eqnarray*}
H(t)  & \le & \frac{1}{\varepsilon}+\frac{m_\varepsilon}{\varepsilon^2(t-t_0) +\frac{\varepsilon}{\varepsilon H(t_0)-1}}\\& \le &
\frac{1}{\varepsilon}+\frac{m_\varepsilon}{\varepsilon^2(t-t_0)}.
\end{eqnarray*}
Note that, technically, this computation is only valid  until the time at which $H(t)=1/\varepsilon$. However, the conclusion in the second line is trivially true after that time as well. Now, applying Lemma \ref{lm:fastmixing} with $f$ replaced by $f_t$, we have for any $t\in[t_0,t_1]$,
\begin{eqnarray*}
t_1 & \le & t+\frac{1+H(t)}{\gamma \varepsilon}\\
& \le & t+\frac{1}{\gamma\varepsilon}+\frac{1}{\gamma\varepsilon^2}+\frac{m_\varepsilon}{\gamma \varepsilon^3(t-t_0)}\\
& \le & t+\frac{1}{\gamma\varepsilon^3}+\frac{m_\varepsilon}{\gamma \varepsilon^3(t-t_0)},
\end{eqnarray*}
because $\varepsilon\le 1/2$. Choosing $t:=t_0+\sqrt{\frac{m_\varepsilon}{\gamma \varepsilon^3}}$ yields exactly the first claim. Note that technically, this choice is only valid if $t\le t_1$. However, $t>t_0$ would mean $t_1-t_0<\sqrt{\frac{m_\varepsilon}{\gamma \varepsilon^3}}$, which is an even better conclusion than the claimed one. Finally, we  can replace the Poincaré-based mixing-time estimate of Lemma \ref{lm:fastmixing} with the MLS-based estimate
\begin{eqnarray*}
\tmix(f,\varepsilon) & \le & \frac{1}{\alpha}\log_+\left(\frac{\Ent f}{2\varepsilon^2}\right),
\end{eqnarray*}
which readily follows from the definition of $\alpha$  and Pinsker's inequality. Applying this to $f_t$ instead of $f$ as above, we obtain
\begin{eqnarray*}
t_1 & \le & t+\frac{1}{\alpha}\log\left(\frac{1}{2\varepsilon^3}+\frac{m_\varepsilon}{2\varepsilon^4(t-t_0)}\right),
\end{eqnarray*}
and choosing $t=t_0+\frac{1}{\alpha}$ leads to the second claim. Here again, this choice is technically valid only if $t\le t_1$, but otherwise the conclusion is even better.
\end{proof}
As promised, we  now illustrate the power of the IDI approach by producing simple cutoff criteria for chains with non-negative Bakry-\'Emery curvature. Let us start with the emblematic case of diffusions, for which our conclusion takes a particularly pleasant form.
\subsection{Cutoff for non-negatively curved diffusions}
In the case of diffusions, the following result shows that cutoff reduces to the product condition. This considerably simplifies, unifies and generalizes several proofs of cutoff that were obtained for specific models with enough structure to allow for explicit computations. 
\begin{theorem}[Non-negatively curved diffusions]\label{th:cutoffdiff}Any diffusion with $\rho\ge 0$ started from a deterministic  state $x\in\dX$ satisfies the IDI (\ref{def:IDI})  with rate 
$
\psi(t) =  2t.
$
In particular, 
\begin{eqnarray*}
\forall\varepsilon\in(0,1/2),\quad \wmix(x,\varepsilon) & \le & \frac{1}{\gamma\varepsilon^3}+\sqrt{\frac{8\tmix(x,\varepsilon)}{\gamma\varepsilon^3}}.
\end{eqnarray*}
Thus,  there is cutoff from $x$ as soon as $\gamma\tmix(x,\varepsilon)\to\infty$ for some $\varepsilon\in(0,1)$. Similarly,
\begin{eqnarray*}
\forall\varepsilon\in(0,1/2),\quad \wmix(\varepsilon) & \le & \frac{1}{\gamma\varepsilon^3}+\sqrt{\frac{8\tmix(\varepsilon)}{\gamma\varepsilon^3}},
\end{eqnarray*}
as soon as $\dX$ is compact, hence a worst-case cutoff occurs if and only if $\gamma\tmix\to\infty$.
\end{theorem}
\begin{proof}By definition of $\rho$, we have the local Poincaré inequality 
\begin{eqnarray*}
\Var(g(X_t)) & \le & \left(\int_0^t2e^{-2\rho u}\dd u\right)\,\EE[\Gamma g(X_t)],
\end{eqnarray*}
for any observable $g\in\cF$, any deterministic initial state $X_0=x$, and any time $t\ge 0$. Choosing $g=\log f_t$ and using the assumption $\rho\ge 0$, we obtain
\begin{eqnarray*}
\Varent(f_t) & \le & 2t\,\EE[\Gamma \log f_t(X_t)].
\end{eqnarray*}
We now crucially use the chain rule (\ref{chain}), specific to diffusions, to write
\begin{eqnarray*}
\EE[\Gamma \log f_t(X_t)] & = & \EE\left[\frac{Lf_t(X_t)}{f_t(X_t)}-(L\log f_t)(X_t)\right]\\
& =&  \int_{\dX}\left({Lf_t}-f_t L\log f_t \right)\dd\pi\\
& = & \cE(f_t,\log f_t)\\
& = & -\frac{\dd}{\dd t}\Ent(f_t),
\end{eqnarray*}
where the second equality uses the fact that $X_t$ has density $f_t$ w.r.t. $\pi$. Thus, the  IDI (\ref{def:IDI}) holds with rate $\psi(t)={2t}$, and  Lemma \ref{lm:IDI} readily leads to the claimed estimate on $\wmix(x,\varepsilon)$. The second one follows by taking a maximum over all initial states $x\in\dX$.
\end{proof}

\begin{remark}[Positive curvature]When the Bakry-\'Emery curvature $\rho$ is positive, the above argument shows that the IDI (\ref{def:IDI}) holds with  rate $\Psi(t)=1/\rho$.  When combined with the Lichnerowicz estimate $\gamma\ge\rho$ (Remark \ref{rk:Lichnerowicz2}), this leads to the neater result
\begin{eqnarray*}
\forall x\in\dX,\quad \forall\varepsilon\in(0,1/2),\quad \wmix(x,\varepsilon) & \le & \frac{3}{\rho\varepsilon^3}.
\end{eqnarray*} 
\end{remark}
It is important to recall that Theorem \ref{th:cutoffdiff} can not  extend verbatim to discrete state spaces: indeed, we have seen in Section \ref{sec:PC} that the product condition fails to imply cutoff even within particularly  nice classes of models, such as random walks on Abelian groups (which have non-negative Bakry-\'Emery curvature by Theorem \ref{th:CIW}). However, the counter-examples were fully connected ($E=\dX^2$), which we led us to suspect that a control on sparsity should enter the product condition in order to imply cutoff. This is confirmed in the next section.
\subsection{Cutoff for non-negatively curved chains}
The purpose of this final section is to state, prove and apply the following general result, which can be seen as an answer to our main challenge (Problem \ref{pb:main}) for chains with non-negative curvature.  Extending it to negatively curved chains -- and, in particular, to random walks on expanders (Problem \ref{pb:exp}) -- remains completely open at the time of writing.
\begin{theorem}[Cutoff for non-negatively curved chains]Any weakly reversible Markov chain with $\rho\ge 0$, starting from a  deterministic state $x\in\dX$  satisfies the IDI (\ref{def:IDI}) with rate  
\begin{eqnarray*}
\psi(t) & = & 16t\log d+4t\log_+\left(\frac{\diam(\dX)}{t}\right).
\end{eqnarray*}
Consequently, for any  $\varepsilon\in (0,1/2)$, we have the width estimates
\begin{eqnarray*}
\wmix(x,\varepsilon) & \le & c_\varepsilon\max\left\{\frac{1}{\gamma},\sqrt{\frac{\tmix(x,\varepsilon)\log d}{\gamma}},\sqrt{\frac{\tmix(x,\varepsilon)}{\gamma}\log\left(\frac{\diam(\dX)}{\tmix(x,\varepsilon)}\right)}\right\};\\
\wmix(x,\varepsilon) & \le & \frac{c_\varepsilon}{\alpha}\max\left\{\log\log d,\log(\alpha\tmix(x,\varepsilon)\right\}.
\end{eqnarray*}
 where $c_\varepsilon$ depends only on $\varepsilon$. 
Similarly, from a worst-case starting point, we have
\begin{eqnarray*}
\wmix(\varepsilon) & \le & c_\varepsilon\max\left\{\frac{1}{\gamma},\sqrt{\frac{\tmix \log d}{\gamma}}\right\};\\
\wmix(\varepsilon) & \le & \frac{c_\varepsilon}{\alpha}\max\left\{\log\log d,\log(\alpha\tmix)\right\}.
\end{eqnarray*}
In particular, there is cutoff from $x$ as soon as  for some fixed $\varepsilon\in(0,1)$,
\begin{eqnarray*}
\gamma\tmix(x,\varepsilon) \ \gg \ \max\left\{\log d,\log\left(\gamma\diam(\dX)\right)\right\} & \textrm{ or } & \alpha\tmix(x,\varepsilon) \ \gg \ {\log \log d},
\end{eqnarray*}
and a worst-case cutoff as soon as 
\begin{eqnarray*}
\gamma\tmix \ \gg \ \log d & \textrm{ or } & \alpha\tmix \ \gg \ {\log \log d}.
\end{eqnarray*}
\end{theorem}
\begin{proof}Fix $t\ge 0$. As in the proof of Theorem \ref{th:cutoffdiff}, the assumption $\rho\ge 0$ implies 
\begin{eqnarray*}
\forall t\ge 0,\qquad \Varent(f_t) & \le & 2t\EE[\Gamma\log f_t(X_t)].
\end{eqnarray*}
While the chain-rule based identity $\EE\left[\Gamma \log f_t(X_t)\right] = \cE(f_t,\log f_t)$ no longer holds, we can use its approximate version (Lemma \ref{lm:chain}) and our roughness estimate (Lemma \ref{lm:lip}) to write
\begin{eqnarray*}
\EE\left[\Gamma \log f_t(X_t)\right] & \le & (1+\Lip\log f_t)\cE(f_t,\log f_t)\\
& \le & \left(8\log d+2\log_+\frac{\diam(\dX)}{t}\right)\cE(f_t,\log f_t).
\end{eqnarray*}
Combining those two facts leads to the IDI (\ref{def:IDI}) with  rate
\begin{eqnarray*}
\psi(t) & = &  16t\log d+4t\log_+\left(\frac{\diam(\dX)}{t}\right)\\
& \le & \max\left\{32t\log d,8t\log\left(\frac{\diam(\dX)}{t}\right)\right\}.
\end{eqnarray*}
The first term in this maximum is clearly increasing in $t$, and so is the second in the range where it is larger than the first. Thus, the maximum is increasing in $t$, hence 
\begin{eqnarray*}
m_\varepsilon & \le & 32\max\left\{\tmix(x,\varepsilon)\log d, \tmix(x,\varepsilon)\log\left(\frac{\diam(\dX)}{\tmix(x,\varepsilon)}\right)\right\}.
\end{eqnarray*}
Applying Lemma \ref{lm:IDI} readily leads to the first claimed estimate on $\wmix(x,\varepsilon)$. For the second, we additionally use Remark \ref{rk:diammls} and $\log u\le u-1$ to write
\begin{eqnarray*}
m_\varepsilon & \le & c\log d\max\left\{\tmix(x,\varepsilon), \frac{1}{\alpha}\right\},
\end{eqnarray*}
for some universal constant $c$. Finally, the worst-case version follows by taking a maximum over  $x\in\dX$ and recalling that $\tmix(\varepsilon)$ is of order at least  $\diam(\dX)$, by Lemma \ref{lm:diammix}. 
\end{proof}
Note that the criterion $\gamma\tmix\gg\log d$ is considerably weaker than the one in  Corollary \ref{coro:var}, which already had notable applications. But even more useful is the criterion  $\alpha\tmix\gg\log\log d$, which turns out to be satisfied by many important models, as we will now show. Let us start with our  toy example, simple random walk on the $n-$cube.
\begin{example}[Random walk on the hypercube]\label{ex:hypercube}For simple random walk on the $n-$cube, 
\begin{equation*}
d \ = \ n, \qquad \alpha\ =\ 1/n,\qquad \tmix \ =\ \Theta(n\log n).
\end{equation*}
Thus, the criterion $\alpha\tmix \gg {\log \log d}$ is satisfied and cutoff follows. 
\end{example}
As a less trivial example, consider the random transposition walk, for which the occurrence of a cutoff  is a celebrated historical result due to Diaconis and Shahshahani \cite{diaconis1981generating}.
\begin{example}[Random transpositions]\label{ex:transpositions}Consider the random walk on the symmetric group of order $n$ generated by the uniform distribution over transpositions. Then, 
\begin{equation*}
d \ = \ \Theta(n^2), \qquad \alpha\ =\ \Theta(1/n),\qquad \tmix \ =\ \Theta(n\log n),
\end{equation*}
see \cite{MR2023890,MR2094147}. Thus, our  criterion $\alpha\tmix \gg {\log \log d}$ is again satisfied and cutoff follows.  
\end{example}
More generally, one can replace the set of transpositions in the above example by any conjugacy class whose \emph{complexity} (number of non-fixed points) is not too large.
\begin{example}[Random walks generated by a conjugacy class]On the symmetric group of order $n$, consider the random walk generated by the uniform distribution on a conjugacy class $S$. Let  $k$ denote the number of non-fixed points in any member of $S$. Then, 
\begin{equation*}
d \ = \ \Theta(n^k), \qquad \alpha=\Theta\left(\frac{k}{n}\right),\qquad \tmix \ =\ \Theta\left(\frac{n\log n}{k}\right),
\end{equation*}
so that our criterion is satisfied as long as $k=n^{o(1)}$, see again \cite{MR2094147}.  Note that the previous example corresponds to the special case where $k=2$. Interestingly,  cutoff is known to occur in the more general regime where $k=o(n)$, as conjectured by   Diaconis and Shahshahani \cite{diaconis1981generating}, and recently proved by Berestycki and \c{S}eng\"{u}l  \cite{MR3936154}. 
\end{example}
Finally, let us turn to Monte Carlo Markov chains, focusing on two celebrated models.

\begin{example}[Ising model on a graph]Recall that the Ising model with inverse temperature $\beta\ge 0$ on a finite graph $G=(V_G, E_G)$ is the probability measure
\begin{eqnarray*}
\pi(x) \ \propto \ \exp\left\{\beta \sum_{\{i,j\}\in  E_G}x_ix_j\right\} & \textrm{ on } &  \dX=\{-1,1\}^{V_G}.
\end{eqnarray*}
Consider the sampler (\ref{def:Glauber}), in which the $i-$th coordinate of $x$ is flipped at rate $c_(x)\propto \sqrt{\frac{\pi(x^i)}{\pi(x)}}$. We have seen in Theorem \ref{th:MCMC} that   $\rho\ge 0$ as soon as
\begin{eqnarray}
\label{assume:beta}
\Delta (1-e^{-2\beta})e^{2\Delta \beta} & \le & 1,
\end{eqnarray}
where $\Delta$ denotes the maximum degree in  $G$. Moreover, under this condition, Theorem \ref{th:MCMC} also guarantees that $\alpha$ is of the same order as in the basic case where $\pi$ is uniform, and so are $d$ and $\tmix$. Thus, cutoff occurs along any sequence of bounded-degree graphs with diverging size, in the high-temperature regime (\ref{assume:beta}).
\end{example}

\begin{example}[Hard-core model on a graph]Recall that the Hard-core model with fugacity $\zeta>0$ on a finite graph $G=(V_G, E_G)$ is the probability measure 
\begin{eqnarray*}
\pi(x)  \propto \zeta^{\sum_{i\in V_G}x_i}
& \textrm{ on } & 
\dX \ = \left\{x\in \{0,1\}^{ V_G}\colon \forall \{i,j\}\in E_G, x_ix_j=0\right\}.
\end{eqnarray*}
Consider again the sampler (\ref{def:Glauber}), in which the $i-$th coordinate of $x$ is flipped at rate $c_i(x)\propto\sqrt{\frac{\pi(x^i)}{\pi(x)}}$. Again, Theorem \ref{th:MCMC} guarantees that $\rho\ge 0$ as soon as 
\begin{eqnarray}
\label{assume:lambda}
\zeta\Delta & \le & 1,
\end{eqnarray}
where $\Delta$ denotes the maximum degree in  $G$.  Moreover, under this condition, the same result implies that  that $\alpha$ is of the same order as in the basic case where $\pi$ is uniform, and so are $d$ and $\tmix$. Thus, cutoff occurs along any sequence of bounded-degree graphs with diverging sizes, throughout the low-fugacity regime (\ref{assume:lambda}).
\end{example}

\bibliographystyle{plain}
\bibliography{mix}
\end{document}